\documentclass[11pt]{article}
\usepackage{hyperref}
\usepackage{a4wide, times}
\usepackage{amsfonts, amssymb, amsmath, amsthm}
\usepackage{graphicx}
\usepackage[latin1]{inputenc}
\usepackage{color}
\usepackage{epstopdf}
\usepackage{subfigure}
\usepackage{mathabx}
\numberwithin{equation}{section}
\usepackage{booktabs}
\usepackage{float}
\usepackage{caption}
\usepackage{algorithm}
\usepackage{algpseudocode}
\usepackage{booktabs}
\usepackage{graphicx}
\makeatletter
\newcommand*{\centerfloat}{%
\parindent \z@ \leftskip \z@ \@plus 1fil \@minus \textwidth \rightskip\leftskip \parfillskip
\z@skip}
\makeatother
\newcommand{\half}{\frac{1}{2}}

\newtheorem{theorem}{Theorem}

\newtheorem{lemma}[theorem]{Lemma}

\usepackage{natbib}
\bibpunct{(}{)}{,}{}{}{,} 
\def\R{\mathbb{R}}

\title{Numerical methods for solving PIDEs\\
arising in swing option pricing\\ under a two-factor
mean-reverting model with jumps}
\author{
Mustapha Regragui\thanks{Department of Mathematics, Computer Science and Statistics, Ghent University, 9000 Ghent, Belgium.\hfill \break Email: mustapha.regragui@ugent.be},\, 
Karel J. in 't Hout\thanks{Department of Mathematics, University of Antwerp, Middelheimlaan 1, 2020 Antwerp, Belgium.\hfill\break Email: karel.inthout@uantwerpen.be},\,
Mich{\`e}le Vanmaele\thanks{Department of Mathematics, Computer Science and Statistics, Ghent University, 9000 Ghent, Belgium.\hfill\break Email: michele.vanmaele@ugent.be},\,
Fred Espen Benth\thanks{Department of Data Science and Analytics, BI Norwegian Business School, N-0484 Oslo, Norway.\hfill\break Email: fred.e.benth@bi.no}
}

\begin{document}
    \maketitle
\begin{abstract}
    This paper concerns the numerical valuation of swing options with discrete action times under a linear two-factor mean-reverting model with jumps. The resulting sequence comprises two-dimensional partial integro-differential equations (PIDEs); each equation is convection-dominated  and possesses a nonlocal integral term due to the presence of jumps. Further, the initial function is nonsmooth. We propose various second-order numerical methods that can adequately handle these challenging features. The stability and convergence of these numerical methods are analysed theoretically. By ample numerical experiments, we confirm their second-order convergence behaviour. In addition to the numerical valuation of swing options, a study of their Delta Greeks and the optimal exercise policy is performed.
\end{abstract}
    \section{Introduction}
    Electricity is traded through several types of financial derivatives contracts, such as forwards, futures, swaps and swing options. This paper deals with the valuation of swing options. This type of contract gives the holder the right to buy electricity multiple times at a fixed price under some constraints, for example the holder cannot buy more than a certain amount of energy during the entire life time of the option and also during each exercise period of the option. 
    
    In the literature, there are different formalisations of swing options. In \cite{Kjaer}, the contract is seen as a multi-exercise Bermudan option where the holder can exercise at multiple, predetermined dates and the option price is  the solution of a sequence of parabolic partial integro-differential equations. In \cite{Dahlgren}, the contract is formalised as a multi-exercise American option, where the holder can exercise at any time as long as a certain waiting time between two successive exercise times is respected, and the option valuation is about solving partial integro-differential complementarity problems. Next, the contract can be formalised such that the holder can exercise in continuous time, see \cite{Benth2011OnTO,eriksson2013swing}, which leads to the study of a Hamilton--Jacobi--Bellman (HJB) type equation. In this paper, we will focus on the formalisation where we have a finite number of fixed, predetermined exercise dates. 
    
    There are several ways of pricing swing options. Lattice-based methods, see \cite{Jaillet}, and Monte Carlo type simulations, see \cite{Ibanez}, have been used. Also, considering an expansion of the density of the underlying price process in terms of convenient basis functions, several methods have been derived for this kind of options, see \cite{Zhang2010AnEP,Kribky}. Another approach, which forms the focus of this work, is to solve a sequence of parabolic partial integro-differential equations (PIDEs) as in, e.g., \cite{Kjaer} and \cite{CALVOGARRIDO2016,CALVOGARRIDO201977}. One of the advantages of using a PIDE approach is that one can compute the option price for a whole set of spot prices at once, while Monte Carlo and lattice-based methods can only provide the option price for one spot price.

    The electricity price possesses a mean reversion property with spikes and (daily, weekly and annually) seasonal patterns. One of the first models for this price was proposed by \cite{Lucia} in the form of a geometric Ornstein--Uhlenbeck process which has a mean-reverting property. However, this model has the disadvantage of not incorporating spikes. In \cite{benth2007non}, the electricity spot price is modelled as a linear combination of an Ornstein--Uhlenbeck process and pure mean-reverting jump processes. \cite{kluge} considered an exponential form of this model. In this paper, we follow the approach of \cite{eriksson2013swing} and consider an affine two-factor model with finite activity jumps. The choice of an affine model is motivated by the fact that the electricity market exhibits negative prices, especially in recent years.
    
    Under the affine two-factor spot model, the pertinent two-dimensional PIDEs are convection-dominated in the first direction and have pure convection and a nonlocal integral part in the second direction.
    For its numerical solution, the method of lines is employed, consisting of a discretisation in space followed by a discretisation in time. When the PIDEs are discretised in the spatial domain by finite difference schemes, this results in a large semidiscrete system of ordinary differential equations (ODEs). This system of ODEs is subsequently discretised in the temporal domain using a time-stepping scheme of the operator splitting kind, where the partial differential part is treated implicitly and the integral part is treated explicitly.
    
  The \textit{first contribution} of this paper is the design and analysis of efficient and robust numerical methods that adequately address two key challenges: (i) the handling of the convection-dominated nature of the problem combined with the nonsmooth initial function, and (ii) the treatment of the nonlocal integral term.  

To handle the convection-dominated feature together with the nonsmooth payoff function, we consider two approaches. The first approach is to apply the semi-Lagrangian method. The second approach is to discretise the convection term by carefully chosen finite difference schemes. It is well-known that classical second-order central schemes can lead to spurious oscillatory behaviour. Accordingly, in this paper, we shall explore and compare various second-order upwind schemes, notably the QUICK scheme.  

For the integral term, we consider a second-order spatial discretisation. Next, we present two temporal discretisation schemes that handle the integral part explicitly through a fixed-point iteration. 

The \textit{second contribution} of this paper is a theoretical analysis of the stability and convergence of the proposed numerical methods.  

   The \textit{third contribution} of this paper consists of ample numerical experiments to assess the robustness and accuracy and study the order of convergence of the proposed numerical methods.

  The \textit{fourth contribution} of this paper is the study of swing options using a linear two-factor model for the underlying electricity price. Using the proposed numerical methods, we focus first on a rigorous valuation of this type of options. Then, we study the impact of the model parameters on the option value and also the behaviour of the optimal exercise policy.
    
    This paper is organised in the following way.  Section \ref{sec2} presents the electricity spot price model under consideration and the formulation of the option pricing problem as a sequence of two-dimensional PIDEs. Section \ref{SecSD} concerns the spatial discretisation with special attention to the convection and integral parts.   
    Section \ref{TempScheme} presents the temporal discretisation schemes. These schemes are all second-order and treat the integral term in an explicit way by means of a fixed-point iteration. Section \ref{theory} is devoted to the theoretical analysis of the stability and convergence properties of the schemes. In Section \ref{sec6}, we present ample numerical experiments, especially to study the observed order of convergence. 
    In addition, the behaviour of the Delta Greeks and the optimal exercise policy are investigated.
Finally, Section 7 gives our conclusions.

    \section{Swing option price modelling}
    \label{sec2}
    \subsection{The electricity spot price}
     Let $(\Omega, \mathcal{F},(\mathcal{F}_{t})_{t}, \mathbb{P})$ be a complete     filtered probability space satisfying the usual conditions, with $\mathbb{P}$ the historical or real world probability measure. We assume that there exists an equivalent pricing measure $\mathbb{Q}\sim \mathbb{P}$. 
     
     As in  \cite{eriksson2013swing} we consider the following   linear two-factor model for the electricity spot  price, $S$, adapted to the filtration $(\mathcal{F}_{t})_{t}$ but with dynamics modelled directly under $\mathbb{Q}$\footnote{We could have started, as, e.g., in \cite{Kjaer}, with the $\mathbb{P}$-dynamics and then considered a measure change to get the corresponding $\mathbb{Q}$-dynamics. But as we are interested in the  numerical valuation of derivatives we work directly under the measure~$\mathbb{Q}$.}:
    \begin{align*}
            S_{t}  & = X_{t}+ Y_{t},   \\
        dX_{t} & =\alpha (\mu - X_{t})dt + \sigma dW_{t}, \\
        dY_{t} & = -\beta Y_{t}dt + J_{t}dN_{\lambda,t}.  
    \end{align*}
    
    $X$ is an Ornstein--Uhlenbeck process with mean reversion level $\mu$ and mean
    reversion speed $\alpha$ which depicts the mean reversion property of the electricity
    price, $W$ represents a standard Brownian motion. To incorporate the spikes in
    the prices, $Y$ is a mean-reverting process with a jump component where
    $N_{\lambda}$ is a Poisson process with constant intensity $\lambda$ and $J$
    is the jump size process. The jump size distribution is given by a bounded density function
    $f$. We assume $J$, $N_{\lambda}$ and $W$ to be mutually independent. The mean
    reversion speeds $\alpha$ and $\beta$ and the volatility $\sigma$ are positive constants, while the mean reversion
    level $\mu$ can be either a constant or a time-dependent periodic deterministic
    function characterising the seasonality of the energy price. For ease of presentation, we will assume in this paper that $\mu$ is constant.
  
    The electricity spot price could be modelled in exponential form, i.e., $S_{t}
    =\exp (X_{t}+ Y_{t})$, see, e.g., \cite{kluge,CALVOGARRIDO2016, CALVOGARRIDO201977}. However, we
    prefer the affine form, i.e., $S_{t}=X_{t}+Y_{t}$, which allows the price to
    become negative. 
    The occurrence of negative prices was observed in the day-ahead market and
    is generally due to oversupply combined with low demand, inflexible power
    stations (e.g.\ nuclear reactors) and cheap renewable power. A study by the
    Belgian Federal Commission for Electricity and Gas Regulation (CREG), see \cite{creg}, indicated
    that, in 2020, the cumulative number of hours with negative prices reached 136 hours in Belgium, 102 hours in France,
    319 hours in Germany and 97 hours in the Netherlands.
    \subsection{Formulation of the problem}
    We consider a Bermudan swing option with a predetermined finite number $N_a$ of discrete action times (c.f., \cite{Kjaer}).
    
    We assume that the swing option has the following properties:
    \begin{enumerate}
        \item The fixed strike price is $K$ and the maturity time is $T$.

        \item Swing action times are of the form $T_{n}= n\Delta T$ ($n=1,\ldots
            ,N_a$) with $\Delta T =\frac{T}{N_a}$.

        \item (Local constraint) At each swing action time, the holder has the
            right to buy at most $L$ units of energy for
            the price $K$.

        \item (Volume constraint) The total amount of units bought should not exceed
            a predetermined global upper bound $M$ over the lifetime $[0,T]$ of the
            option.
    \end{enumerate}
    We model the option value as a solution to an optimal stochastic control problem with multiple stopping times.
    \newline
    Let $\mathcal{A}^{N_a}$ be a class of \emph{admissible strategies} consisting of all $\mathbb{F}$-adapted processes $(u_{t})_{0\leq t \leq T}\in L^{2}(\Omega\times[0,T])$ that admit the representation 
    \[
    u_{t}= \sum_{n = 1}^{N_a-1}a_{n}1_{[T_n,T_{n+1}[}(t) + a_{N_a}\delta_{t,T_{N_a}},
    \]
    where $\delta$ denotes the Kronecker symbol, and the $\mathcal{F}_{T_n}$-measurable random variables $a_n$, $n=1,\ldots, N_a$, represent the number of units bought at action time $T_{n}$ satisfying the constraints $a_{n}\in \{0, 1,\ldots, L\}$ and
    \[
    \sum_{n=1}^{N_a}u_{T_{n}}=\sum_{n=1}^{N_a}a_n\leq M .
    \]
 
  Denote by $Z$ the process of the amount
    $Z_{t}$ of energy bought up to time $t$ (where time $t$ is not included).
    The option value function at a swing action time $T_{i}$, $i\in\{1,\ldots,N_a\}$, can be expressed as the conditional expected
    present value of the payoff given a control process $u\in \mathcal{A}
    ^{N_a}$ from the swing action time $T_{i}$ up to maturity time $T$, and given
    that the two factors $X$ and $Y$ and the amount of energy $Z$ at the swing
    action time $T_{i}$ have values $x$, $y$ and $z$ respectively:
    \[
        v (x,y,z,T_{i}) = \sup_{u\in \mathcal{A}^{N_a}} \mathbb{E}_{\mathbb{Q}}\big[ \, \sum_{n=i}^{N_a}e^{-r(T_n-T_i)}(
        X_{T_n}+Y_{T_n}-K)u_{T_{n}}\mid X_{T_i}= x,Y_{T_i}= y, Z_{T_i}= z\,\big],      
    \]
    where $r$ is the risk-free interest rate.
    At $T_0=0$, the amount of purchased energy is zero. Then, the option value can also be expressed as:  
    \[
        v (x,y,0,0) = \sup_{u\in \mathcal{A}^{N_a}} \mathbb{E}_{\mathbb{Q}}\big[ \, \sum_{n=1}^{N_a}e^{-r(T_n-T_0)}(
        X_{T_n}+Y_{T_n}-K)u_{T_{n}}\mid X_{T_0}= x,Y_{T_0}= y, Z_{T_0}= 0\,\big].     
    \]
   
    Using the dynamic programming principle together with the Feynman--Kac theorem as in,  e.g., \cite{Kjaer}, the option value function satisfies a sequence of PIDEs coupled with conditions at the exercise dates:
    \begin{equation}
        \begin{cases}
            \displaystyle \partial_{t}v(x,y,z,t) + \frac{\sigma^2}{2}\partial_{xx}v(x,y,z,t) + \alpha(\mu - x) \partial_{x}v(x,y,z,t) -\beta y\partial_{y}v(x,y,z,t) - (r+\lambda)v(x,y,z,t) \\[3mm]
            \displaystyle \qquad \qquad \qquad \qquad \qquad+\lambda\int_{\mathbb{R}}v(x,y +\xi,z,t)f(\xi)d\xi =0 , \qquad T_{i-1}< t < T_{i}, \\[3mm]
            \displaystyle v(x,y,z,T_i) = \sup_{u\in \mathcal{A}^{N_a}}\sum_{n=i}^{N_a}e^{-r(T_n-T_i)}\mathbb{E}_{\mathbb{Q}}\big[(X_{T_n}+ Y_{T_n}- K)u_{T_{n}}\mid X_{T_i}= x, Y_{T_i}= y, Z_{T_i}= z\big],                                  \\
        \end{cases}
        \label{Eq:2}
    \end{equation}
    for $i\in\{1,\ldots,N_a\}$ and $(x,y,z)\in \mathbb{R}\times\mathbb{R}\times\{0,1,\ldots,M\}$.
    
    Between any two successive swing action times $T_{i-1}$ and $T_{i}$, the option value
    function is the solution of a parabolic PIDE with a terminal condition at
    time $T_{i}$. As this PIDE is the same in each interval $]T_{i-1},T_{i}[$, we start our numerical solution approach for \eqref{Eq:2} by considering the simpler problem:
    \begin{equation}
        \begin{cases}
            \displaystyle \partial_{t}v(x,y,t)= \frac{\sigma^2}{2}\partial_{xx}v(x,y,t) + \alpha(\mu - x) \partial_{x}v(x,y,t) -\beta y\partial_{y}v(x,y,t) - (r+\lambda)v(x,y,t)\\[3mm]
            \displaystyle \qquad \qquad \qquad \qquad \qquad+\lambda\int_{\mathbb{R}}v(x,y +\xi,t)f(\xi)d\xi , \qquad t>0 ,                             \\[3mm]
            v(x,y,0)= \max(x+y -K,0).
        \end{cases}
        \label{Eq:1}
    \end{equation} 
    In \eqref{Eq:1} time has been reversed, as it is  preferred to have an initial condition instead
    of a terminal condition for the PIDE. Note that herein the maximum function is used as it is suboptimal to exercise when the payoff is negative. Further, the variable $z$, that
    represents the amount of energy, is dropped from the option value function $v$ because we
    are only interested in one interval between exercise dates and $z$ stays constant in such interval. 
   
    \section{Spatial discretisation} \label{SecSD}
    For the numerical solution of problem \eqref{Eq:1}, we apply the method of
    lines, consisting of a discretisation in space followed by a discretisation in time. This
    section deals with the spatial discretisation. We successively consider the diffusion-reaction part, the integral part and the convection part of \eqref{Eq:1}. The temporal
    discretisation will be discussed in the next section.
    \subsection{Diffusion-reaction part} \label{sectionDiffusion}
    The spatial domain $\mathbb{R}^{2}$ is truncated to a bounded set $[x_{\min},
    x_{\max}]\times[y_{\min},y_{\max}]$, where $x_{\rm max} > 0$, $y_{\rm max} > 0$,  $x_{\rm min} < 0$ and $y_{\rm min} < 0$ are all taken sufficiently large in absolute value.
    For the $x$-direction and $y$-direction, we impose linear boundary conditions:
    \[
        \frac{\partial^{2}v}{\partial x^{2}}|_{x=x_{\min}}= \frac{\partial^{2}v}{\partial
        x^{2}}|_{x=x_{\max}}= 0\qquad \text{and}\qquad \frac{\partial^{2}v}{\partial
        y^{2}}|_{y=y_{\min}}= \frac{\partial^{2}v}{\partial y^{2}}|_{y=y_{\max}}=
        0 .
    \]
    These conditions, which are common in computational finance, are also natural in our present application.
   
    Let integers $m_1,m_2\geq 1$ and the parameter $d>0$ be given. We use a smooth nonuniform Cartesian grid $\{(x_{i},y_{j})\in [x_{\min},x_{\max}]\times
    [y_{\min},y_{\max}] \mid 0\leq i \leq m_{1}, 0\leq j\leq m_{2}\}$ such that
    a large portion of grid points is contained in a region of (financial and numerical) interest, see, e.g., \cite{houtLamotte2023}.
    \newline
    In the $x$-direction, a smooth nonuniform mesh
    $x_{\min}=x_{0}<x_{1}< \cdots< x_{m_1}=x_{\max}$ is defined by
    \begin{equation*}
        x_i=
        \begin{cases}
            -\frac{1}{2}K + d\cdot\sinh(\xi_{x,i} - \xi_{x,1,\text{int}}), & \text{whenever\ } \xi_{x,i} \leq \xi_{x,1,\text{int}}, \\[2mm]
            x_{i-1} + d\cdot\Delta \xi_{x}, & \text{whenever\ } \xi_{x,1,\text{int}}< \xi_{x,i} < \xi_{x,2,\text{int}}, \\[2mm]
            \frac{3}{2}K + d\cdot\sinh(\xi_{x,i} - \xi_{x,2,\text{int}}), & \text{whenever\ }  \xi_{x,2,\text{int}}\leq \xi_{x,i},
        \end{cases}
    \end{equation*}
    where $\xi_{x,\min}=\xi_{x,0}< \cdots< \xi_{x,m_1}=\xi_{x,\max}$ are
    equidistant points, $\Delta \xi_x = \xi_{x,1}-\xi_{x,0}$, $\xi_{x,1,\text{int}}= -\frac{K}{2d}$, $\xi_{x,2,\text{int}}= \frac{3K}{2d}$, $\xi_{x,\min}
    = \xi_{x,1,\text{int}}+ \sinh^{-1}(\frac{x_{\min}}{d}- \xi_{x,1,\text{int}})$ and $\xi_{x,\max}
    = \xi_{x,2,\text{int}}+ \sinh^{-1}(\frac{x_{\max}}{d}- \xi_{x,2,\text{int}})\,.$
    \newline
    In the $y$-direction, a smooth nonuniform mesh $y_{\min}=y_{0}<y_{1}< \cdots< y_{m_2}=y_{\max}$ is defined by
    \begin{equation*}
        y_j=
        \begin{cases}
            -K + d\cdot\sinh(\xi_{y,j} - \xi_{y,1,\text{int}}), & \text{whenever\ } \xi_{y,j} \leq \xi_{y,1,\text{int}}, \\[2mm]
            y_{j-1} + d\cdot\Delta \xi_{y}, & \text{whenever\ } \xi_{y,1,\text{int}}< \xi_{y,j} < \xi_{y,2,\text{int}}, \\[2mm]
            K + d\cdot\sinh(\xi_{y,j} - \xi_{y,2,\text{int}}), & \text{whenever\ }  \xi_{y,2,\text{int}}\leq \xi_{y,j},
        \end{cases}
    \end{equation*}
    where $\xi_{y,\min}=\xi_{y,0}< \cdots< \xi_{y,m_2}=\xi_{y,\max}$ are
    equidistant points, $\Delta \xi_y = \xi_{y,1}-\xi_{y,0}$, $\xi_{y,1,\text{int}}= -\frac{K}{d}$, $\xi_{y,2,\text{int}}= \frac{K}{d}$, $\xi_{y,\min}
    = \xi_{y,1,\text{int}}+ \sinh^{-1}(\frac{y_{\min}}{d}- \xi_{y,1,\text{int}})$ and $\xi_{y,\max}
    = \xi_{y,2,\text{int}}+ \sinh^{-1}(\frac{y_{\max}}{d}- \xi_{y,2,\text{int}})\,.$
    
    The grid is uniform with a relatively small spatial mesh width inside the \emph{region of financial interest} $[-\frac{1}{2}K,\frac{3}{2}K]\times[-K,K]$ and nonuniform outside. The parameter $d$ controls the fraction of points $(x_i,y_j)$ inside  $[-\frac{1}{2}K,\frac{3}{2}K]\times[-K,K]$. In this paper, we heuristically choose $d=\frac{K}{5}$.  
    Figure~\ref{FigGrid} shows a sample spatial grid for $m_{1}=m_{2}=50, K=50, x_{\min} = -50,x_{\max}=200,y_{\min}=-150,y_{\max} = 150$.
    \begin{figure}
        \centerfloat
        \includegraphics[trim=120 0 80 0, clip, scale=0.4]{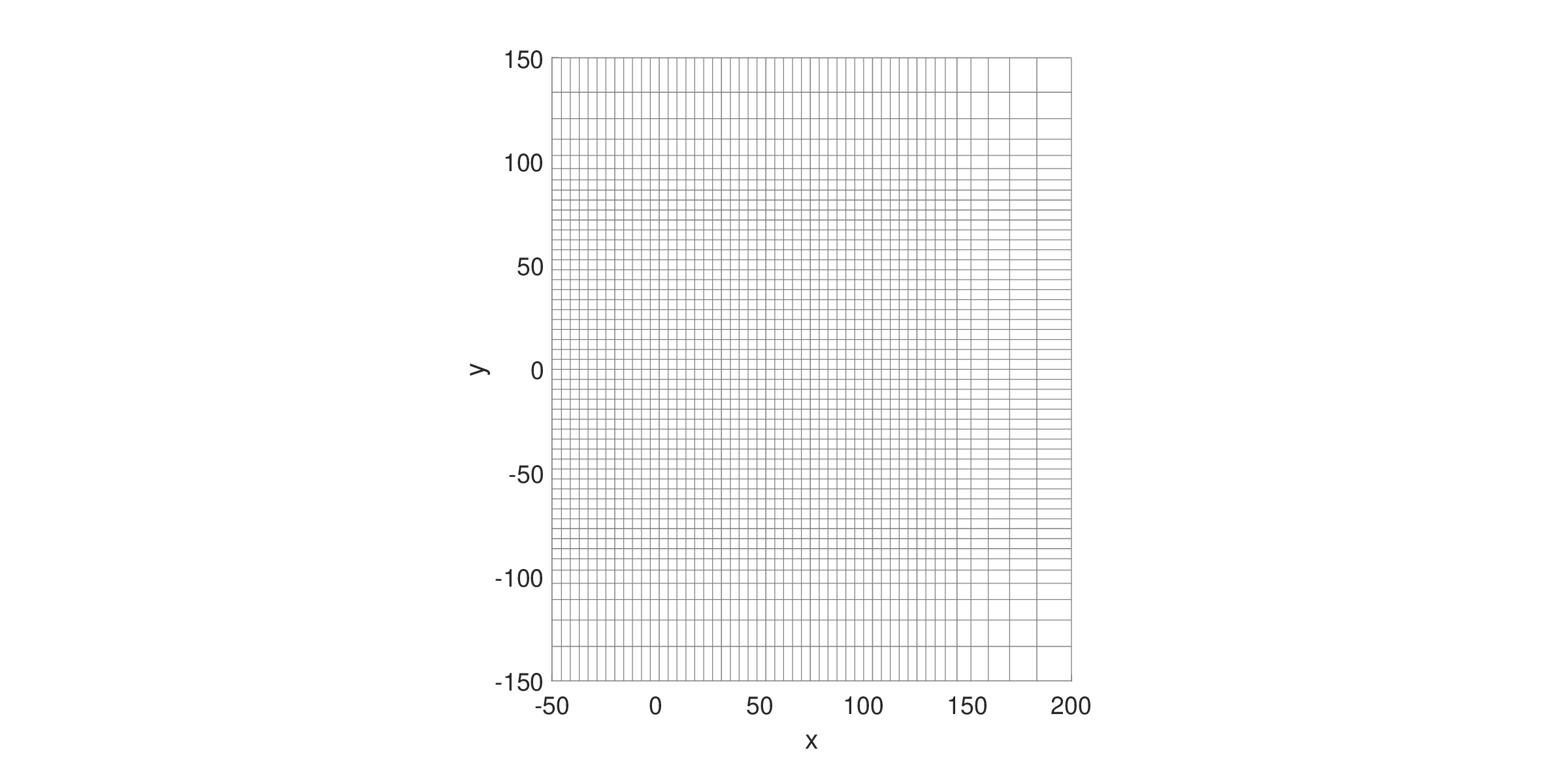}
        \caption{Sample spatial grid for the parameter values $m_{1}=m_{2}=50,K=50, x_{\min} = -50,x_{\max}=200,y_{\min}=-150,y_{\max} = 150$, $d=10$.}
        \label{FigGrid}
    \end{figure}
    
    We denote the semidiscrete approximation of $v(x_{i},y_{j},t)$ by
    $V_{i,j}(t)$ and define the corresponding vector
    \[
        V(t)=(V_{0,0}(t),V_{0,1}(t),\ldots,V_{m_1,m_2-1}(t),V_{m_1,m_2}(t))\in \mathbb{R}
        ^{(m_1+1)(m_2+1)}.
    \]
    The diffusion term in \eqref{Eq:1} is approximated by the second-order central finite difference formula:
    \[
        \partial_{xx}v(x_{i},y_{j},t) \approx \omega_{i,-1}V_{i-1,j}(t) + \omega_{i,0}
        V_{i,j}(t) + \omega_{i,1}V_{i+1,j}(t),\quad 1\leq i\leq m_{1}-1,\; 0\leq j\leq m_{2}
    \]
    with coefficients
    \begin{equation*}\label{weights}
        \omega_{i,-1}=\frac{2}{\Delta x_{i-1}(\Delta x_{i-1}+\Delta x_{i})}\;,\;\omega
        _{i,0}=\frac{-2}{\Delta x_{i-1}\Delta x_{i}}\;,\;\omega_{i,1}=\frac{2}{\Delta
        x_{i}(\Delta x_{i-1}+\Delta x_{i})},
    \end{equation*}
    and $\Delta x_{i}$ is the mesh width in the $x$-direction: $\Delta x_{i}= x_{i+1}-x_{i}$.
    
    The discretisation matrix corresponding to the diffusion-reaction part in \eqref{Eq:1} can be written
    as a Kronecker product:
    \begin{equation}\label{eqA}
         A^D = \Bigl(\frac{\sigma^{2}}{2}D_{2}-(r+\lambda)I_{1}\Bigl)\otimes I_{2}.
    \end{equation}
    Here, $I_{1},I_{2}$ are identity matrices of sizes $(m_{1}+1)\times(m_{1}+1)$ and $(
    m_{2}+1)\times(m_{2}+1)$, respectively.
    $D_{2}:= \mbox{trid}\,[\omega_{i,-1},\omega_{i,0},\omega_{i,1}]$ is a $(m_1+1)\times(m_1+1)$ tridiagonal
    matrix that represents the numerical differentiation of order two in the $x$-direction.
    In view of the linear boundary conditions, the elements in the top and bottom rows of the
    matrix $D_{2}$ are all equal to zero, i.e.,  $\omega_{0,\cdot} = \omega_{m_1,\cdot}=0$.
    \subsection{Integral part}
    \label{sec3.2} To approximate the integral part in \eqref{Eq:1}, we truncate the integration
    domain $\mathbb{R}$ to $[y_{\min},y_{\max}]$ and use linear interpolation
    for the semidiscrete approximation between any given two consecutive grid points in the $y$-direction. Hence,
    starting from the integral at the grid point $(x_{i},y_{j})$, $0\leq i\leq m_1$, $0\leq j\leq m_2$, we get:
    \begin{align*}
        \lambda\int_{-\infty}^{+\infty}v(x_{i},y_{j}+\xi,t)f(\xi)d\xi
        & = \lambda\int_{-\infty}^{+\infty}v(x_{i},\xi,t)f(\xi-y_{j})d\xi \\                    & \approx\lambda\int_{y_0}^{y_{m_2}}v(x_{i},\xi,t)f(\xi-y_{j})d\xi  \\ 
        & \approx \lambda\sum_{\ell=0}^{m_{2}-1}\int_{y_{\ell}}^{y_{\ell+1}}\Bigl(\frac{y_{\ell+1}-\xi}{\Delta y_{\ell}}V_{i,\ell}(t) + \frac{\xi-y_{\ell}}{\Delta y_{\ell}}V_{i,\ell+1}(t)\Bigl)f(\xi-y_{j})d\xi \\
        & = \lambda\sum_{\ell=0}^{m_{2}}\Tilde{B}_{j,\ell}V_{i,\ell}(t)  \\
        & =:\mathcal{J}_{i,j}(t),
    \end{align*}
    where $\Delta y_{\ell}$ is the mesh width
    $\Delta y_{\ell}= y_{\ell+1}-y_{\ell}$, and
    \begin{equation}\label{eqAtilde}
        \begin{cases}
            \Tilde{B}_{j,0}= \dfrac{y_{1}F_{0,j,0} - F_{1,j,0}}{\Delta y_{0}}                                                                                                    \\[2mm]                                                                           
            \Tilde{B}_{j,\ell}= \dfrac{F_{1,j,\ell-1}-y_{\ell-1}F_{0,j,\ell-1}}{\Delta y_{\ell-1}}+ \dfrac{y_{\ell+1}F_{0,j,\ell} - F_{1,j,\ell}}{\Delta y_{\ell}},\quad 1\leq\ell\leq m_{2}-1,   \\[2mm]
             \Tilde{B}_{j,m_2}= \dfrac{F_{1,j,m_2-1}-y_{m_2-1}F_{0,j,m_2-1}}{\Delta y_{m_2-1}}
        \end{cases}
    \end{equation}
    with $F_{0,j,\ell}= \displaystyle \int_{y_{\ell}}^{y_{\ell+1}}f(\xi-y_{j})d\xi, \, F_{1,j,\ell}
    = \int_{y_{\ell}}^{y_{\ell+1}}\xi f(\xi-y_{j})d\xi$.
    \newline
    Let $B= I_{1}\otimes\lambda\Tilde{B}$ be the semidiscrete jump matrix with $\Tilde{B}= (\Tilde{B}_{j,\ell})_{0\leq j,\ell\leq m_2}$. Then,  $BV$ denotes the
    approximation of the integral where the value of the vector $BV$ at the entry
    $i\cdot (m_{2}+1)+ j$ is exactly $\mathcal{J}_{i,j}(t)$.
    To obtain a more accurate approximation of the integral term in our numerical experiments, we account for the contribution of the integrand outside the truncated spatial domain $[y_{\min}, y_{\max}]$. For the approximation of the option value function outside this domain, we apply linear extrapolation.
    This approach improves the accuracy of the numerical quadrature used to evaluate the integral term.
    The computational complexity  of the above approximation is $m_{1}m_{2}^{2}$. We remark, however, that in the special case of the Kou-type jump model the integral can approximated with linear complexity $m_{1}m_{2}$ by an algorithm due to \cite{Toivanen}.
    \subsection{Convection part}
    Aside from the nonlocal integral term in \eqref{Eq:1}, a significant challenge arises due to the nature of the asset price model: the problem is convection-dominated in the $ x$-direction and exhibits pure convection in the $ y $-direction. This feature is attributed to the electricity price dynamics, which exhibits fast mean reversion characterised by large values of the parameters $ \alpha $ and, especially, $ \beta $. These values, detailed in the parameter sets provided in Table \ref{paramSets} of Section \ref{6.1}, result in a highly convection-dominated problem as the diffusion coefficient, $ \frac{\sigma^2}{2}$, is constant and relatively small in comparison with the convection coefficients. 

Moreover, the presence of nonsmooth initial data coupled with a convection-dominated problem leads to the formation of a region of nonsmoothness of the solution, characterised by sharp gradients. This poses difficulties for central finite difference formulas, which may suffer from spurious oscillations.

To address the convection dominance behaviour, we explore two numerical strategies: a semi-Lagrangian method and a suitable semidiscretisation method. Both techniques require effective interpolation or discretisation schemes that ensure adequate accuracy while maintaining numerical stability. 
     
    \subsubsection{Semi-Lagrangian method} \label{sec3.3.2} 
    The semi-Lagrangian method is a well-known and powerful numerical
    tool for solving transport and convection-dominated problems. 
    One employs the characteristic curve $(x(t),y(t),t)$ such that:
    \begin{itemize}
    \item $x(t)$ satisfies:
            \[
                \frac{\partial v(x(t), y,t)}{\partial t}= \frac{\partial v}{\partial
                t}(x(t), y,t) -\alpha (\mu - x)\frac{\partial v}{\partial x}(x(t)
                , y,t)
            \]
            which holds if $x(t) = \mu (1-e^{-\alpha(s-t)}) +
            x(s)e^{-\alpha(s-t)}$ whenever $t\leq s$.
        \item $y(t)$ satisfies:
            \[
                \frac{\partial v(x, y(t),t)}{\partial t}= \frac{\partial v}{\partial
                t}(x, y(t),t) +\beta y \frac{\partial v}{\partial y}(x, y(t),t)
            \]
            which holds if $y(t) = y(s)\exp\{-\beta(s-t)\}$ whenever $t\leq s$.       
    \end{itemize}
    Thus, on the characteristic curve, the PIDE \eqref{Eq:1} can be rewritten a         s:
    \begin{equation}
        \partial_{t}v(x(t),y(t),t) = \frac{\sigma^{2}}{2}\partial_{xx}v(x(t),y(t
        ),t) - (r+\lambda)v(x(t),y(t),t) +\lambda\int_{\mathbb{R}}v(x(t),y(t) +\xi,t)f(\xi)d\xi
        .
    \end{equation}
    We refer to Section \ref{sec4.1} for its temporal discretisation.
    \subsubsection{Semidiscrete approximation}\label{sec3.3.3} 
     Instead of using a semi-Lagrangian method, we can directly approximate
    the convection terms $\frac{\partial v}{\partial x}$ and $\frac{\partial v}{\partial
    y}$ at the grid points $(x_{i},y_{j})$, $0\leq i\leq m_1$, $0\leq j\leq m_2$,   using one of the following schemes:
    \begin{description}
            \item[Second-order upwind scheme] \label{BDF}  \mbox{ }\\
            By Taylor expansion, one obtains the formula for the second-order upwind scheme in the case of nonuniform spatial grids. It is a second-order finite difference approximation of the first-order
            derivative with three-point stencils:
            \begin{equation}\label{eqSUSx}
            \begin{split}
               & \frac{\partial v}{\partial x}(x_i,y_j,t)\\
               &\approx
               \begin{cases}
                    \dfrac{\Delta x_{i-2}+ \Delta x_{i-1}}{\Delta x_{i-2}\Delta x_{i-1}}(V_{i,j}-V_{i-1,j}) - \dfrac{\Delta x_{i-1}}{\Delta x_{i-2}(\Delta x_{i-2} + \Delta x_{i-1})}(V_{i,j}-V_{i-2,j})
                    & \; \text{if}\; \alpha(\mu -  x_{i})< 0,    \\[3mm]
                    \dfrac{\Delta x_{i} + \Delta x_{i+1}}{\Delta x_i\Delta x_{i+1}}(V_{i+1,j}-V_{i,j}) - \dfrac{\Delta x_{i}}{\Delta x_{i+1}(\Delta x_{i}+\Delta x_{i+1})}(V_{i+2,j}-V_{i,j})
                    & \; \text{if}\; \alpha(\mu -  x_{i})>0 \\
                \end{cases}
                \end{split}
            \end{equation}
            and
            \begin{equation}\label{eqSUSy}
             \begin{split}
               &   \frac{\partial v}{\partial y}(x_i,y_j,t) \\
               & \approx
                \begin{cases}
                    \dfrac{\Delta y_{j-2} + \Delta y_{j-1}}{\Delta y_{j-2}\Delta y_{j-1}}(V_{i,j}-V_{i,j-1}) - \dfrac{\Delta y_{j-1}}{\Delta y_{j-2}(\Delta y_{j-2} + \Delta y_{j-1})}(V_{i,j}-V_{i,j-2})
                    & \;\; \text{if}\; -\beta y_{j}<0  ,      \\[3mm]
                    \dfrac{\Delta y_j + \Delta y_{j+1}}{\Delta y_j\Delta y_{j+1}}(V_{i,j+1}-V_{i,j}) - \dfrac{\Delta y_{j}}{\Delta y_{j+1}(\Delta y_{j}+\Delta y_{j+1})}(V_{i,j+2}-V_{i,j})
                    & \;\; \text{if}\; -\beta y_{j}>0\,. \\
                \end{cases}
                  \end{split}
            \end{equation}
            Here, for ease of presentation, we omitted the argument $t$ of $V$.
        \item[QUICK scheme]\label{quick}
        \mbox{ }\\
        The QUICK (Quadratic Upstream Interpolation for Convective Kinematics) scheme, see \cite{LEONARD197959}, is a second-order method based on quadratic interpolation. It is commonly used in computational fluid dynamics (CFD) for solving convection-diffusion equations. For clarity, the QUICK scheme employed in this paper is formulated using the finite difference approach, although it was originally introduced and is more commonly applied within the finite volume framework. Let quadratic polynomials $p_{x,i,j}$ and $p_{y,i,j}$ be defined by:
        \begin{align*}
            p_{x,i,j}(x)&= \frac{(x\!-\!x_{i})(x\!-\!x_{i+1})}{\Delta x_{i-1}(\Delta x_{i-1}+\Delta x_{i})}V_{i-1,j}\!-\!\frac{(x\!-\!x_{i-1})(x\!-\!x_{i+1})}{\Delta x_{i-1}\Delta x_{i}}V_{i,j}\!+\!\frac{(x\!-\!x_{i-1})(x\!-\!x_{i})}{(\Delta x_{i-1}+\Delta x_{i})\Delta x_{i}}V_{i+1,j},\\
            p_{y,i,j}(y)&= \frac{(y\!-\!y_{j})(y\!-\!y_{j+1})}{\Delta y_{j-1}(\Delta y_{j-1}+\Delta y_{j})}V_{i,j-1}\!-\!\frac{(y\!-\!y_{j-1})(y\!-\!y_{j+1})}{\Delta y_{j-1}\Delta y_{j}}V_{i,j}\!+\!\frac{(y\!-\!y_{j-1})(y\!-\!y_{j})}{(\Delta y_{j-1}+\Delta y_{j})\Delta y_{j}}V_{i,j+1}.
        \end{align*}
        Then, the QUICK scheme for the first-order derivative is given by: 
        \begin{equation}\label{quickx}
            \begin{split}
                \frac{\partial v}{\partial x}(x_i,y_j,t)
               \approx
               \begin{cases}
                    \dfrac{p_{x,i,j}(x_{i+1/2})-p_{x,i-1,j}(x_{i-1/2})}{x_{i+1/2}-x_{i-1/2}}
                    & \; \text{if}\; \alpha(\mu -  x_{i})< 0,    \\[3mm]
                    \dfrac{p_{x,i+1,j}(x_{i+1/2})-p_{x,i,j}(x_{i-1/2})}{x_{i+1/2}-x_{i-1/2}}
& \; \text{if}\; \alpha(\mu -  x_{i})>0, \\
                \end{cases}
                \end{split}
            \end{equation}

            \begin{equation}\label{quicky}
            \begin{split}
                \frac{\partial v}{\partial y}(x_i,y_j,t)
               \approx
               \begin{cases}
                    \dfrac{p_{y,i,j}(y_{j+1/2})-p_{y,i,j-1}(y_{j-1/2})}{y_{j+1/2}-y_{j-1/2}}
                    & \; \text{if}\; -\beta y_{j}< 0,    \\[3mm]
                    \dfrac{p_{y,i,j+1}(y_{j+1/2})-p_{y,i,j}(y_{j-1/2})}{y_{j+1/2}-y_{j-1/2}}
& \; \text{if}\; -\beta y_{j}>0, \\
                \end{cases}
                \end{split}
            \end{equation}
            where $x_{i+1/2}=\frac{x_{i+1}+x_i}{2}$ and $y_{j+1/2}=\frac{y_{j+1}+y_j}{2}$.
            \end{description}
        The above semidiscretisation schemes, second-order upwind and QUICK,  can be assembled in matrix form. The values at ghost points, for example $V_{m_1+1,j}$ and $V_{m_1+2,j}$, are defined by linear extrapolation of values inside the truncated domain, for instance $V_{m_1+1,j} = 2V_{m_1,j}-V_{m_1-1,j}$ and $V_{m_1+2,j}= 3V_{m_1,j}-2V_{m_1-1,j}$. The matrices can be expressed in the following way: 
        \begin{equation}\label{Ax}
        A_x = D_x\Tilde{A}_x\otimes I_2 ,   
        \end{equation}
         and 
         \begin{equation}\label{Ay}
             A_y =I_1\otimes D_y\Tilde{A}_y,
         \end{equation}
          where $\Tilde{A}_x$ and $\Tilde{A}_y$ represent the matrices of numerical differentiation of order one in the $x$- respectively $y$-direction stemming from either the second-order upwind scheme or the QUICK scheme. Next, $D_x$ and $D_y$ are diagonal matrices with $D_{x,i,i}=\alpha(\mu - x_i)$ and $D_{y,j,j}=-\beta y_j$ for $0\leq i\leq m_1$ and $0\leq j\leq m_2$.
    
    The spatial discretisation of \eqref{Eq:1} defined in this section leads to the system of ODEs
    \begin{equation}\label{ODE}
        \frac{d V}{dt}(t)=(A + B)V(t),
    \end{equation}
    where
    \begin{equation}\label{A}  
    A = A^{D} + A_x+ A_y.
    \end{equation}
    The initial vector is given by
    \begin{equation}\label{v0}
        V(0)  = \max(\hat{x}\otimes e_{y}+ e_{x}\otimes \hat{y} - K\cdot e_{x}\otimes e_{y},0),
    \end{equation}  
    where $\hat{x}= (x_0,\ldots, x_{m_1})$ and $\hat{y}= (y_0,\ldots, y_{m_2})$, and  $e_{x},e_{y}$ denote vectors of ones of size $m_1+1$ and $m_2+1$ respectively.
    \section{Temporal discretisation}\label{TempScheme}
    In this section we present different schemes for the temporal discretisation of the semidiscrete problem.
    Note that for the semi-Lagrangian approach as well as the semidiscrete approach the jump matrix $B$ derived from the spatial discretisation in Section \ref{sec3.2} is a full matrix. Thus, we will avoid using a temporal
    scheme where one needs to solve a linear system involving this matrix.
    We consider the temporal discretisation schemes described in the subsections below. Let integer $N\geq 1$ be given and the step size $\Delta t =\frac{T}{N}$. Let $V^{n}$ denote
    the approximation of $V(t_{n})$ at the temporal grid point
    $t_{n}= n\Delta t$ ($n=1,2,\ldots,N$) with $V^0 = V(0)$. Let $I= I_1\otimes I_2$.
    \subsection{Temporal scheme for the semi-Lagrangian approach}
    \label{sec4.1}  
            The Crank--Nicolson scheme with fixed-point iteration combined with the semi-Lagrangian
            approach was proposed in \cite{asianForsyth} for the numerical
            valuation of Asian options. Applied to our case, we have
            \begin{equation} \label{slcnfi}
                (I-\tfrac{1}{2}\Delta t A^D)Y_{\ell}=\mathcal{I}[V^{n}
                ]+\tfrac{1}{2}\Delta t \mathcal{I}[A^DV^{n}] +\tfrac{1}{2}\Delta t\mathcal{I}
                [BV^{n}] + \tfrac{1}{2}\Delta tBY_{\ell-1},
            \end{equation}
            for $\ell=1,\ldots, \ell_{max}$ and $V^{n+1}= Y_{\ell_{max}}$. Here, $Y
            _{0}= 2V^{n}-V^{n-1}$ if $n\geq 1$ and $Y_0 = V^0$ if $n=0$. The following stopping criterion is used for the fixed-point iteration: 
            \[
            \underset{k}{\max}\frac{\mid Y_{\ell,k}-Y_{\ell-1,k}\mid}{\max(1,\mid
            Y_{\ell,k}\mid) }< 10^{-7} .
            \]
    In the scheme \eqref{slcnfi}, 
the operator $\mathcal{I}$ denotes the two-dimensional cubic spline interpolation, 
which yields approximations at the departure points $(\mu(1-e^{-\alpha\Delta t}) + x_{i}e^{-\alpha\Delta t},y_{j}e^{-\beta \Delta t})$, 
using known approximations at the grid points $(x_{i},y_{j})$ for $0\leq i \leq m_{1}$ and $ 0\leq j\leq m_{2}$. Hereafter, the scheme \eqref{slcnfi} is referred to as the semi-Lagrangian Crank--Nicolson scheme with fixed-point iteration (\textbf{SLCNFI}).
    \subsection{Temporal schemes for the semidiscretisation approach}\label{secODE}
We consider two temporal discretisation schemes for the semidiscrete system \eqref{ODE}:
    \begin{description}
        \item [Crank--Nicolson scheme with fixed-point iteration (CNFI)]             \mbox{ }\\
            The combination of the Crank--Nicolson scheme for the convection-diffusion-reaction
            part with  a fixed-point iteration for the integral part was proposed in \cite{randal} and analysed in
            \cite{cnfi}:
            \begin{equation}
                \label{CNfpi}(I-\tfrac{1}{2}\Delta t A)Y_{\ell}=(I+\tfrac{1}{2}\Delta
                t A)V^{n}+\tfrac{1}{2}\Delta t(BV^{n}+ BY_{\ell-1})\, ,
            \end{equation}
            for $\ell=1,\ldots, \ell_{\max}$. We use the same starting vector and stopping criterion as in the semi-Lagrangian
            approach. We apply  Rannacher time-stepping for the first two steps using half time steps $\frac{1}{2}\Delta t$ by computing $V^{1}$ and $V^{2}$ using the backward Euler scheme with fixed-point iteration on the integral part. This technique is a well known remedy for the adverse impact of the nonsmoothness of the initial function on the convergence of the Crank--Nicolson scheme due to its lack of $L$-stability (see \cite{Rannacher1984}).

    \item[Diagonally implicit Runge-Kutta scheme with fixed-point iteration (DIRKFI)] \mbox{ }\\
    This scheme, studied by \cite{hout2024efficientnumericalmethodamerican}, combines the DIRK scheme with a penalty/fixed-point iteration for the numerical valuation for American-style options under the two-asset Kou-type jump-diffusion model. In our case, we obtain
    \begin{equation}
        \label{DIRK}\left\{
        \begin{array}{lll}
           W_1 = V^{n} + (1-\theta)\Delta t(AV^{n} + BV^{n}), \\
            (I-\theta\Delta tA)Y_{\ell} = W_1 + \theta\Delta tBY_{\ell-1} \quad (\ell = 1,\ldots, \ell_{\max}), \\
            \widehat{Y} = Y_{\ell_{\max}}, \\
            W_2 =V^{n} + \tfrac{1}{2}\Delta t(AV^{n} + BV^{n}) + (\tfrac{1}{2}-\theta)\Delta t(A\widehat{Y} + B\widehat{Y}), \\
            (I-\theta\Delta tA)Z_{\ell} = W_2 + \theta\Delta tBZ_{\ell-1}  \quad (\ell = 1,\ldots, \ell_{\max}), \\
            V^{n+1} = Z_{\ell_{\max}}.
        \end{array}\right.
    \end{equation}
    At each time step, there are two fixed-point iteration processes. For the starting vectors, $Y_0 = 2V^n - V^{n-1}$ if $n\geq 1$, $Y_0= V^0$ if $n=0$ and $Z_0 = \hat{Y}$. For both processes, the same stopping criterion as in the SLCNFI scheme is used. The scheme has a second-order consistency for any $\theta$ and is $L$-stable if and only if $\theta = 1 \pm \frac{\sqrt{2}}{2}$, see \cite{CashDIRK}. The choice $\theta = 1 - \frac{\sqrt{2}}{2}$ yields a smaller error constant as observed in \cite{hout2024efficientnumericalmethodamerican}. 
    \end{description}

    \section{Convergence and stability analysis for the semidiscretisation approach}
   
    \label{theory}
     
    This section deals with the convergence and stability analysis of some of the numerical schemes above. In Section \ref{5.1},  the convergence in the $\ell_{\infty}$-norm of the CNFI scheme  \eqref{CNfpi}
    to the Crank--Nicolson scheme is studied and similarly the DIRKFI scheme. Then, in Section \ref{5.2}, we study the stability and convergence in the $\ell_2$-norm of the Crank--Nicolson and DIRK schemes under Dirichlet boundary conditions. Throughout this section,
   we consider a uniform grid with mesh width $\Delta x$ in the $x$-direction and $\Delta y$ in the $y$-direction.

For the semidiscretisation of the convection terms in \eqref{Eq:1} we consider the following general finite difference scheme

    \begin{equation}\label{eqFVij}
    \begin{split}
        a_i\frac{\partial v}{\partial x}(x_i,y_j,t)
        & \approx a^{+}_{i}\frac{w_2V_{i+2,j}+w_1V_{i+1,j}+ w_0V_{i,j} + w_{-1}V_{i-1,j}}{\Delta x}\\
       & \quad + a^{-}_{i}\frac{-w_{-1}V_{i+1,j}-w_0V_{i,j}-w_1V_{i-1,j}-w_2V_{i-2,j}}{\Delta
        x},\\[4mm]
        b_j\frac{\partial v}{\partial y}(x_i,y_j,t)
        &\approx b^{+}_{j}\frac{w_2V_{i,j+2}+w_1V_{i,j+1}+ w_0V_{i,j} + w_{-1}V_{i,j-1}}{\Delta y}\\
        & \quad + b^{-}_{j}\frac{-w_{-1}V_{i,j+1}-w_0V_{i,j}-w_1V_{i,j-1}-w_2V_{i,j-2}}{\Delta
        y},
        \end{split}
    \end{equation}
    for $0\leq i\leq m_1$ and $0\leq j\leq m_2$. Here,  $a_i =\alpha(\mu-x_i) $ and $b_j =-\beta y_j $ and  for any real number $c$ we denote $c^{+}=\max(c,0)$ and $c^{-}=\min(c,0)$. The coefficients $w_{-1},w_0,w_1,w_2$ satisfy the following conditions
      \begin{equation}
      \label{W3} 
        \sum_{k=-1}^2w_k=0,\quad \sum_{k=-1}^2kw_k=1,\quad \sum_{k=-1}^2k^{2}w_k=0,\quad
            w_2\leq0.
            \end{equation}
       The three equalities in \eqref{W3}, which can be derived using Taylor expansion, are sufficient and necessary conditions for the finite difference scheme to be at least of second-order. Schemes of interest that belong to the above family of schemes \eqref{eqFVij}-\eqref{W3} are:
        \begin{itemize}
            \item The second-order upwind scheme \eqref{eqSUSx}-\eqref{eqSUSy} with $w_2=-\frac{1}{2}$, $w_1 = 2, w_0 = -\frac{3}{2}$ and $w_{-1}=0$.
            \item The second-order QUICK scheme \eqref{quickx}-\eqref{quicky} with $w_2=-\frac{1}{8}$, $w_1 = \frac{7}{8}$ and $w_0=w_{-1}=-\frac{3}{8}$.
            \item The third-order upwind scheme with $w_2=-\frac{1}{6}$, $w_1 = 1$, $w_0=-\frac{3}{6}$ and $w_{-1}=-\frac{2}{6}$.
            \item The second-order central scheme with $w_2=0$, $w_1=\frac{1}{2}$, $w_{0}=0$ and $w_{-1}=-\frac{1}{2}$.
        \end{itemize}
        Without loss of generality, we assume that the functions $a$ and $b$ with $a(x)= \alpha(\mu-x)$ and $b(y)=-\beta y$ change sign within the truncated domain. 
        Near the boundary, if a numerical stencil extends outside the domain -- e.g., requiring values such as $V_{-1,j}$ -- linear extrapolation is employed, see Section \ref{sec3.3.3}. This extrapolation is consistent with the imposed linear boundary conditions. We note that, in Section \ref{5.2}, Dirichlet boundary conditions are prescribed, eliminating the need to address extrapolation at the boundaries in that context, in particular that the convection coefficients $a$ and $b$ are positive near $x_{\min}, y_{\min}$ and negative near $x_{\max}, y_{\max}$.
 
    \subsection{Convergence of CNFI and DIRKFI}\label{5.1}
    The theorem below deals with the convergence in the $\ell_{\infty}$-norm of the CNFI scheme \eqref{CNfpi} to the Crank--Nicolson scheme:
    \begin{equation}
        (I-\tfrac{1}{2}\Delta t A)V^{n + 1}=(I+\tfrac{1}{2}\Delta t A)V^{n}+\tfrac{1}{2}\Delta
        t(BV^{n}+BV^{n
        + 1}). \label{CNFII}
    \end{equation}
    \begin{theorem}\label{thcnfi} 
        Let $\varepsilon_{\ell} = V^{n+1}-Y_{\ell}$ where $V^{n+1}$ is given by \eqref{CNFII} and $Y_{\ell}$ is given by \eqref{CNfpi}. Let
        \[
        \kappa_x = (|w_2|+|w_1|+|w_{-1}|+w_0)\max_i | a_{i}|\quad \text{and} \quad \kappa_y = (|w_2|+|w_1|+|w_{-1}| + w_0)\max_j |b_{j}|,
        \]
        where the weights $w_{-1},w_0,w_1,w_2$ satisfy \eqref{W3}. If $\kappa_x\frac{\Delta t}{2\Delta x}+\kappa_y\frac{\Delta t}{2\Delta y}< 1 + \frac{\Delta t}{2}r$, then the CNFI scheme \eqref{CNfpi} converges to the Crank--Nicolson scheme \eqref{CNFII} in the $\ell_{\infty}$-norm and $$|| \varepsilon_{\ell}||_{\infty}\leq
        \Theta\, || \varepsilon_{\ell-1}||_{\infty} \quad \textrm{with} ~~\Theta = \dfrac{\frac{\Delta t}{2}\lambda}{1 + \frac{\Delta t}{2}(r+\lambda)-(\kappa_x\frac{\Delta t}{2\Delta x}+\kappa_y\frac{\Delta t}{2\Delta y})}
        <1.$$
    \end{theorem}
    \begin{proof}
        From \eqref{CNfpi} and \eqref{CNFII}, we obtain:
        \begin{align*}
            || \varepsilon_{\ell}||_{\infty}=||V^{n+1}-Y_{\ell}||_{\infty} & = \tfrac{1}{2}\Delta t||(I-\tfrac{1}{2}\Delta t A)^{-1}B(V^{n+1}-Y_{\ell-1})||_{\infty} \\
                                       & \leq \tfrac{1}{2}\Delta t||(I-\tfrac{1}{2}\Delta t A)^{-1}||_{\infty}||B\varepsilon_{\ell-1}||_{\infty} .
        \end{align*}
         If the matrix $I-\frac{1}{2}\Delta t A$ is diagonally dominant, then one has the following bound by \cite{VARAH19753}:
        \begin{equation*}
            ||(I-\tfrac{1}{2}\Delta t A)^{-1}||_{\infty}\leq \frac{1}{\min_{i} \{|c_{ii}|-\sum_{j\neq i}|c_{ij}| \}},
        \end{equation*}
        where $c_{i,j}$ denote the elements of $I-\tfrac{1}{2}\Delta t A$.
        
        For any given  $i\in \{0,1,\ldots, (m_1+1)\cdot (m_2+1)-1\}$, there exists a tuple $(k_{i},l_{i})\in \{0,1,\ldots,m_1\}\times\{0,1,\ldots,m_2\}$ such that
        \begin{align*}
        |c_{i,i}| = \Bigl|1 + \tfrac{1}{2}\Delta t\Bigl(\delta_i\frac{\sigma^2}{\Delta x^2} +  (r+\lambda) -\frac{|a_{k_{i}}|}{\Delta x}w_0 -\frac{|b_{l_{i}}|}{\Delta y}w_0\Bigl) \Bigl|,
        \end{align*}
        and
        \begin{align*}
            \sum_{j\neq i}|c_{ij}|\leq \tfrac{1}{2}\Delta t\Bigl(\delta_i\frac{\sigma^2}{\Delta x^2} + \frac{|a_{k_{i}}|}{\Delta x}(|w_2|+|w_1|+|w_{-1}|)  + \frac{|b_{l_{i}}|}{\Delta y}(|w_2|+|w_1|+|w_{-1}|) \Bigl),
        \end{align*}
        where $\delta_i=1$ if $i\in\{m_2+1,\ldots,m_1(m_2+1)-1\}$ and $\delta_i=0$ otherwise, which corresponds to the linear boundary condition.
        
        Noting that $|w_2|+|w_1|+|w_{-1}|+w_0\geq \sum_{k=-1}^2w_k=0$ and $|c_{i,i}|\geq c_{i,i}$, it follows that
        \begin{align*}
            |c_{i,i}| -  \sum_{j\neq i}|c_{ij}|&\geq 1 + \tfrac{1}{2}\Delta t(r+\lambda) -  \tfrac{1}{2}\Delta t\Bigl(\frac{|a_{k_{i}}|}{\Delta x}(|w_2|+|w_1|+|w_{-1}|)+ \frac{|a_{k_{i}}|}{\Delta x}w_0\Bigl)\\
            &\ \qquad \qquad - \tfrac{1}{2}\Delta t\Bigl(\frac{|b_{l_{i}}|}{\Delta y}(|w_2|+|w_1|+|w_{-1}|)+\frac{|b_{l_{i}}|}{\Delta y}w_0\Bigl)\\
            &\geq 1 + \tfrac{1}{2}\Delta t(r+\lambda)-\Bigl(\kappa_x\frac{\Delta t}{2\Delta x}+\kappa_y\frac{\Delta t}{2\Delta y}\Bigl).
        \end{align*}
        If $\kappa_x\frac{\Delta t}{2\Delta x}+\kappa_y\frac{\Delta t}{2\Delta y}< 1 + \frac{\Delta t}{2} r$, the matrix $I-\frac{1}{2}\Delta t A$ is diagonally dominant and thus 
        \begin{equation*}
            ||(I-\tfrac{1}{2}\Delta t A)^{-1}||_{\infty}\leq \frac{1}{1 + \tfrac{1}{2}\Delta t(r+\lambda)-(\kappa_x\frac{\Delta t}{2\Delta x}+\kappa_y\frac{\Delta t}{2\Delta y})}.
        \end{equation*}
        By \eqref{eqAtilde} we have
        \begin{align*}
            \sum_{\ell =0}^{m_2}|\Tilde{B}_{j,\ell}|
            &= \sum_{\ell = 0}^{m_2-1} \frac{y_{\ell+1}-y_{\ell}}{\Delta y}\int_{y_{\ell}}^{y_{\ell+1}}f(\xi-y_{j})d\xi \\
            & =\sum_{\ell = 0}^{m_2-1}\int_{y_{\ell}}^{y_{\ell+1}}f(\xi-y_{j})d\xi \\
            &\leq \int_{-\infty}^{\infty}f(\xi)d\xi = 1,
        \end{align*}
        whenever $0\leq j\leq m_2$, which implies that $||B||_{\infty}\leq \lambda$ and, hence, $||B\varepsilon_{\ell-1}||_{\infty}
        \leq \lambda ||\varepsilon_{\ell-1}||_{\infty}$.
        \newline
        Combining the above inequalities, we get the stated bound on the error $||\varepsilon_{\ell}||_{\infty}$.
    \end{proof}

   In Theorem \ref{thcnfi}, the sufficient condition for the convergence of the fixed-point iteration takes a CFL-like form, imposing constraints on the time step $\Delta t $ and the spatial mesh widths $ \Delta x $ and $ \Delta y$. However, in our numerical experiments in Section \ref{sec6}, no restriction on $\Delta t$ was observed for the convergence of the fixed-point iteration.

   We remark that for the convergence of the DIRKFI scheme \eqref{DIRK} to the DIRK scheme a completely similar result is obtained. For the sake of brevity, we omit the details.

    \subsection{Stability and convergence study}\label{5.2}

    For the purpose of this theoretical investigation, we impose on the PIDE \eqref{Eq:1} Dirichlet boundary conditions in both directions, thus for some given functions $u_1$, $u_2$ and $v_1$, $v_2$:
    \begin{align}
        v(x_{-1},y,t) &=        u_{1}(y,t),\quad v(x_{m_1+1},y,t) = u_{2}(y,t)  ,    \notag\\                             
        v(x,y_{-1},t) &=        v_{1}(x,t),\quad v(x,y_{m_2+1},t) = v_{2}(x,t)  ,\label{DirciheltCondition}
    \end{align}
    where $x_{-1}, x_{m_1+1}, y_{-1},y_{m_2+1}$ are introduced to be on the boundary of the domain such that $x_{\min}=x_{-1}<x_0<\cdots<x_{m_1}<x_{m_1+1}=x_{\max}$ and $y_{\min}=y_{-1}<y_0<\cdots<y_{m_2}<y_{m_2+1}=y_{\max}$ are uniform meshes in the $x$- and $y$- direction respectively.
    
        Now, the semidiscrete system takes the form
        \begin{align}\label{SD}
        \frac{d V}{dt}(t)&=(A + B)V(t) + g(t),\quad t>0\\
        V(0) &=\max(\hat{x}\otimes e_{y}+ e_{x}\otimes \hat{y} - K\cdot e_{x}\otimes e_{y},0),
    \end{align}
    where $\hat{x},\hat{y},e_{x},e_{y}$ are defined in Section \ref{sec3.3.3}  and $A$ and $B$ are the matrices defined in \eqref{A} and \eqref{eqAtilde} respectively but taking into account the Dirichlet boundary conditions instead of the linear boundary conditions. The vector $g(t)$ contains the contribution of the Dirichlet boundary.
\newline
For the stability study, consider the scaled Euclidean product of two vectors $U$ and $V$ of equal size defined by $\langle
    U,V\rangle =\Delta x\Delta y \sum U_{k} V_{k}$ 
    with corresponding $\ell_{2}$-norm $||V||_{2}=\sqrt{\langle V,V\rangle}$ and recall the formula of the logarithmic norm for an $m\times m$-matrix $B$ induced by the $\ell_2$-norm:
    \[
    \mu_2[B] = \max\Bigl\{\frac{\langle BV,V\rangle}{\langle V,V\rangle}\mid V\in \mathbb{R}^m,\; V\neq 0\Bigl\}.
    \]
    First, three lemmas are stated before the main results of stability and convergence are derived.
    \begin{lemma}
        For the jump matrix $B$, the bound $|| B||_2 \leq\lambda\sqrt{ L_{y}
        ||f ||_{\infty}}$ holds with $L_{y}= y_{\max}-y_{\min}$, implying that $\mu_2[B]\leq \lambda\sqrt{ L_{y}
        ||f ||_{\infty}}$ where $ ||f ||_{\infty} = \sup_{\xi\in \mathbb{R}}|f(\xi)|<\infty$.\label{bnd}
    \end{lemma}
    \begin{proof}
        We recall $B= I_{1}\otimes\lambda \tilde{B}$ with $\tilde{A}$ defined in \eqref{eqAtilde}. Using that $||\tilde{B}||_{\infty}\leq1$ (see the proof of Theorem \ref{thcnfi}), there follows
        \begin{equation}\label{eaqAJbound}
         ||B||^2_{2}=  \rho((B)^{\top}B)                      \\
                            = \lambda^2 \rho((\tilde{B})^{\top}\tilde{B})              \\
                          \leq\lambda^2 ||\tilde{B}||_{1}||\tilde{B}||_{\infty} \\
                        \leq \lambda^2  ||\tilde{B}||_{1}.
        \end{equation}
        For  $0\leq j\leq m_{2}$, $1\leq \ell\leq m_{2}-1$, we have
        \begin{align*}
            |\Tilde{B}_{j,\ell}| & = \frac{1}{\Delta y}\int_{y_{\ell-1}}^{y_\ell}(\xi-y_{\ell-1})f(\xi-y_{j})d\xi+ \frac{1}{\Delta y}  \int_{y_{\ell}}^{y_{\ell+1}}(y_{\ell+1}-\xi)f(\xi-y_{j})d\xi                                   \\
            &\leq 2||f||_{\infty}\Delta y\int_{0}^{1}\xi d\xi\\
            &= \Delta y||f||_{\infty}.
            \end{align*}
        Similarly, we find $|\Tilde{B}_{j,0}|\leq \frac{1}{ 2}\Delta y||f||_{\infty}
        $ and $|\Tilde{B}_{j,m_2}|\leq \frac{1}{ 2}\Delta y||f||_{\infty}$, leading to 
        $ \sum_{j=0}^{m_2}|\Tilde{B}_{j,\ell}|\leq L_{y}
        ||f ||_{\infty}$. Taking the maximum over $\ell$ yields $||\tilde{B}||_1\leq L_{y}
        ||f ||_{\infty}$. Substitution of the latter bound in the inequality \eqref{eaqAJbound} gives the stated result, noticing that $\mu_2[B]\leq ||B||_2$.
    \end{proof}
    \begin{lemma}\label{bound}
        Let $\tilde{a},\tilde{b}$ be any given real numbers. Let $D_x$ be the diagonal matrix given by $D_{x,i,i}=\tilde{a} -\tilde{b}x_i$. Consider any given finite difference scheme of the form \eqref{eqFVij}-\eqref{W3} and let $\tilde{A}_x$ denote the corresponding differentiation matrix for the $x$-direction, defined analogously to before. Then, the following bound holds:
        \[
        \mu_2[D_x\tilde{A}_x]\leq (|w_{-1}|-10w_2)\tilde{b}.
        \]
        The same bound applies in the case of the $y$-direction.
    \end{lemma}
    \begin{proof}
    The technical proof of this lemma can be found in Appendix \ref{AppA}. 
    \end{proof}
    
  Denote 
  \begin{equation}\label{constants_C}
  \widetilde{C} = (|w_{-1}|-10w_2)(\alpha+\beta)- (r + \lambda)\quad \text{and}\quad \widehat{C} =\widetilde{C} +\lambda\sqrt{ L_{y}
        ||f ||_{\infty}}.
        \end{equation}
    \begin{lemma}\label{logIneq}
    The logarithmic $\ell_2$-norm of the matrix $A$ in \eqref{A} satisfies the following bound
        \begin{equation}
        \mu_2[A] \leq \widetilde{C}.
        \end{equation}    
    \end{lemma}
    \begin{proof}
           From the formulas \eqref{eqA}, \eqref{Ax}, \eqref{Ay}, \eqref{A} and properties of the logarithmic norm we obtain
        \begin{align*}
           \mu_2[A] &\leq \mu_2[\Bigl(\frac{\sigma^{2}}{2}D_{2}-(r+\lambda)I_{1}\Bigl)\otimes I_{2}]+\mu_2[D_x\tilde{A}_x\otimes I_2] + \mu_2[I_1\otimes D_y\tilde{A}_y] \\
           &\leq \frac{\sigma^{2}}{2}\mu_2[D_{2}] -(r+\lambda) + \mu_2[D_x\tilde{A}_x] + \mu_2[D_y\tilde{A}_y].
        \end{align*}
        It is easily seen that $\mu_2[D_2]\leq 0$. Thus, by invoking Lemma \ref{bound},
        \begin{equation*}
        \mu_2[A] \leq (|w_{-1}|-10w_2)(\alpha+\beta)-(r+\lambda ) = \widetilde{C}.      
        \end{equation*}     
    \end{proof}
    The CNFI scheme \eqref{CNfpi} adapted to the case of  Dirichlet boundary conditions reads
    \begin{equation}\label{CNFIDirichlet}
        \big(I-\tfrac{1}{2}\Delta t A\big)Y_{\ell}=(I+\tfrac{1}{2}\Delta t A)V^{n} + \tfrac{1}{2}\Delta t(BV^n+BY_{\ell-1})+ \tfrac{1}{2}\Delta t (g_n + g_{n+1}),
       \end{equation}
    where $g_n=g(n\Delta t)$. 
    The Crank--Nicolson scheme is given in this case by
    \begin{equation}\label{CNDirichlet}
        \big(I-\tfrac{1}{2}\Delta t A\big)V^{n + 1}=(I+\tfrac{1}{2}\Delta t A)V^{n} + \tfrac{1}{2}\Delta t(BV^n+BV^{n+1})+ \tfrac{1}{2}\Delta t (g_n + g_{n+1}).
    \end{equation}

\begin{theorem}
        Let $\varepsilon_{\ell} = V^{n+1}-Y_{\ell}$ where $V^{n+1}$ is given by \eqref{CNDirichlet} and $Y_{\ell}$ is given by \eqref{CNFIDirichlet}. If $\widehat{C}\Delta t< 2$, then the CNFI scheme \eqref{CNFIDirichlet} converges to the Crank--Nicolson scheme \eqref{CNDirichlet} in the $\ell_{2}$-norm and
    $$|| \varepsilon_{\ell}||_{2}\leq 
        \Theta\, || \varepsilon_{\ell-1}||_{2} \quad \textrm{with} ~~\Theta = \frac{\tfrac{1}{2}\Delta t\lambda\sqrt{L_y||f||_{\infty}}}{1 - \tfrac{1}{2}\Delta t\widetilde{C}}<1.$$
\end{theorem}
\begin{proof}
Subtracting \eqref{CNFIDirichlet} from \eqref{CNDirichlet} yields
\begin{align*}
            || \varepsilon_{\ell}||_{2}=||V^{n+1}-Y_{\ell}||_{2} & = \tfrac{1}{2}\Delta t||(I-\tfrac{1}{2}\Delta t A)^{-1}B(V^{n+1}-Y_{\ell-1})||_{2} \\
                                       &\leq \tfrac{1}{2}\Delta t||(I-\tfrac{1}{2}\Delta t A)^{-1}||_{2}||B||_{2} ||\varepsilon_{\ell-1}||_2.
        \end{align*}  
        By Lemma \ref{logIneq}, we have $\frac{\Delta t}{2}\mu_2[A]\leq  \frac{\Delta t}{2}\widetilde{C}<1$. Hence, by von Neumann theorem, see \citealp[Section IV.11]{HW},
 \[||(I-\tfrac{1}{2}\Delta t A)^{-1}||_{2}\leq \frac{1}{1-\tfrac{1}{2}\Delta t\mu_2[A]}\leq \frac{1}{1-\tfrac{1}{2}\Delta t\widetilde{C}}.
        \]
Combining this with the bound of Lemma \ref{bnd} for $||B||_2$ and using \eqref{constants_C}, we arrive at the stated result.
\end{proof}

The theorem above addresses the convergence of the fixed-point iteration in the $\ell_2$-norm and differs from Theorem \ref{thcnfi}, which deals with its convergence in the $\ell_{\infty}$-norm. 
Since convergence in the $\ell_{\infty}$-norm is a stronger requirement, it leads to a more restrictive CFL-like condition involving both the time step $\Delta t$ and spatial mesh widths $\Delta x$ and $\Delta y$. On the other hand, the $\ell_{2}$-norm analysis guarantees convergence under a milder condition, imposing a restriction only on the time step $\Delta t$.
We remark again that for the convergence of the DIRKFI scheme \eqref{DIRK} to the DIRK scheme a completely similar result is obtained.

We focus now on the stability and convergence of the Crank--Nicolson scheme \eqref{CNDirichlet}. The stability function of the Crank--Nicolson scheme is
    \begin{equation}
        R(z)=\frac{1+\frac{1}{2}z}{1-\frac{1}{2}z} \quad (z\in\mathbb{C}).
    \end{equation}
    Let $G(x) = \sup_{Re(z)\leq x} |R(z)|$ be the so-called error growth function. It is known that, see \cite{HW},
     \begin{equation*}
    G(x)\leq 
        \begin{cases}
            1 & \text{if } x\leq 0\\
            1 + 2x & \text{if } 0\leq x\leq 1.
        \end{cases} 
    \end{equation*}
    
    \begin{theorem}\label{Theo5.5}
        Let $\widehat{C}^+=\max(\widehat{C},0)$. The Crank--Nicolson scheme in  \eqref{CNDirichlet} is unconditionally stable in the $\ell_2$-norm: 
        \[||R(\Delta t(A+B))||^n_{2}\leq e^{2T\widehat{C}^+} \; \text{whenever}\quad n=0,1,2,\ldots, \;\text{with}\; n\Delta t\leq T,\;\; \hat{C}\Delta t\leq 1.
        \]
    \end{theorem}
    \begin{proof}
        From Lemmas \ref{bnd} and \ref{logIneq}, we get $\mu_2[\Delta t(A+B)]\leq \Delta t\widehat{C}$. Then, by von Neumann theorem, see \citealp[Section IV.11]{HW}, \[||R(\Delta t(A+B))||_{2}\leq G(\hat{C}\Delta t)\leq G(\hat{C}^+\Delta t).\]
        Thus, 
         \[||R(\Delta t(A+B))||^n_{2}\leq (1+2\Delta t\widehat{C}^+)^n\leq e^{2T\widehat{C}^+}.\]
    \end{proof}
    Theorem \ref{Theo5.5} establishes the stability of the Crank--Nicolson scheme in the $\ell_2$-norm. In the following, we turn to the study of its convergence under the assumption of sufficient smoothness of the exact solution. The proof of the subsequent theorem follows along the lines of \cite{HV}. A convergence analysis without the smoothness assumption typically requires monotonicity of the scheme and relies upon the framework of viscosity solutions, see \cite{Viscosity}. This aspect will be addressed in future research. 

Denote \[
        v_{\Delta x, \Delta y}(t)=(v(x_0,y_0,t),v(x_0,y_1,t),\ldots,v(x_{m_1},y_{m_2-1},t),v(x_{m_1},y_{m_2},t))\in \mathbb{R}
        ^{(m_1+1)(m_2+1)}.
    \]    
    \begin{theorem}\label{Theo5.7}
        Under sufficient smoothness of $v_{\Delta x, \Delta y}$ and second-order consistency of the semidiscretisation, the Crank--Nicolson scheme \eqref{CNDirichlet} applied to the semidiscrete system \eqref{SD} is convergent in the $\ell_2$-norm. Moreover, the global spatial-temporal error $\hat{\epsilon}_n = v_{\Delta x, \Delta y}(t_n) - V^n$ satisfies the following bound for some positive constant $C$: 
        \[
        ||\hat{\epsilon}_n||_2\leq C(\Delta t^2 + \Delta x^2 + \Delta y^2) \; \text{whenever}\quad n=0,1,2,\ldots, \;\text{with}\; n\Delta t\leq T,\;\; \hat{C}\Delta t\leq 1.
        \]
    \end{theorem}
    \begin{proof}
    The Crank--Nicolson scheme \eqref{CNDirichlet} can be written as:
    \[V^{n+1} = R(\Delta t (A+B))V^{n} + \tfrac{1}{2}\Delta t(I-\tfrac{1}{2}\Delta t(A+B))^{-1}(g_{n}+g_{n+1}).
    \]
    Let the local spatial-temporal error $\hat{\delta}_{n+1}$ be defined by
    \[
    v_{\Delta x, \Delta y}(t_{n+1}) = R(\Delta t (A+B))v_{\Delta x, \Delta y}(t_{n}) + \tfrac{1}{2}\Delta t(I-\tfrac{1}{2}\Delta t(A+B))^{-1}(g_{n}+g_{n+1}) + \hat{\delta}_{n+1}.
    \]
    Subtracting the two equalities and noting that $\hat{\epsilon}_{0}=0$ leads to
    \[
    \hat{\epsilon}_{n+1} = R(\Delta t (A+B))\hat{\epsilon}_{n} + \hat{\delta}_{n+1} = \cdots=\sum_{i=1}^{n+1}R(\Delta t (A+B))^{n+1-i}\hat{\delta}_{i}.
    \]
    Then, by virtue of Theorem \ref{Theo5.5},
    \begin{equation}\label{epsiBond}
            ||\hat{\epsilon}_{n+1}||_2\leq e^{2T\widehat{C}^+}(n+1)\max_{1\leq i\leq n+1}||\hat{\delta}_i||_2.
    \end{equation}
    For any integer $i\geq 0$ with $(i+1)\Delta t\leq T$, it holds that
    \begin{align*}
        \hat{\delta}_{i+1} &= \tfrac{1}{2}\Delta t(I-\tfrac{1}{2}\Delta t(A+B))^{-1}( \delta(t_{i}) + \delta(t_{i+1}))\\
        &\quad+ (I-\tfrac{1}{2}\Delta t(A+B))^{-1}\Bigl(v_{\Delta x, \Delta y}(t_{i+1})-v_{\Delta x, \Delta y}(t_{i}) - \tfrac{1}{2}\Delta t\Bigl(\frac{dv_{\Delta x, \Delta y}}{dt}(t_{i+1})+\frac{dv_{\Delta x, \Delta y}}{dt}(t_{i})\Bigl)\Bigl),
    \end{align*}
    where \[\delta(t) = \frac{dv_{\Delta x, \Delta y}}{dt}(t) - (A+B)v_{\Delta x, \Delta y}(t)-g(t)\] denotes the local spatial error.
    By assumption, the semidiscretisation is consistent of order two, i.e., there exists a positive constant $C_1$ (independent of $t$, $\Delta x$, $\Delta y$) such that \[||\delta(t)||_2 \le C_2(\Delta x^2+\Delta y^2).\] 
    Next, by the smoothness assumption on $v_{\Delta x, \Delta y}$, Taylor expansion yields for some positive constant $C_2$ (independent of $i$, $\Delta t$, $\Delta x$, $\Delta y$) that 
    \[
    \| v_{\Delta x, \Delta y}(t_{i+1})-v_{\Delta x, \Delta y}(t_{i}) - \tfrac{1}{2}\Delta t\Bigl(\frac{dv_{\Delta x, \Delta y}}{dt}(t_{i+1})+\frac{dv_{\Delta x, \Delta y}}{dt}(t_{i})\Bigr)\|_2 \leq C_2 \Delta t^3.
    \]
    It follows that there exists a positive constant $C_3$ (independent of $i$, $\Delta t$, $\Delta x$, $\Delta y$) such that
    \begin{equation}\label{deltabnd}
        || \hat{\delta}_{i+1}||_2\leq C_3\Delta t||(I-\tfrac{1}{2}\Delta t(A+B))^{-1}||_2(\Delta t^2 + \Delta x^2+\Delta y^2).
    \end{equation}
    From Lemmas \ref{bnd} and \ref{logIneq}, we have $\mu_2[\tfrac{1}{2}\Delta t(A+B)]\leq \tfrac{1}{2}\Delta t\widehat{C}$. Thus, if $\widehat{C}\Delta t\leq 1$, then
    \begin{equation}\label{neumannbnd}
    ||(I-\tfrac{1}{2}\Delta t(A+B))^{-1}||_2\leq \frac{1}{1-\frac{1}{2}\Delta t\widehat{C}}\leq 2.        
    \end{equation}
    Combining the bounds \eqref{epsiBond}, \eqref{deltabnd} and \eqref{neumannbnd} gives
    \[
    ||\hat{\epsilon}_n||_2\leq 2Te^{2T\widehat{C}^+}C_3(\Delta t^2 + \Delta x^2+\Delta y^2).
    \]
    \end{proof}
Our next step is the study of the stability of the DIRK scheme:
    \begin{equation}\label{DIRKDirichlet}
    \begin{cases}
        W_1 = V^n + (1-\theta)\Delta t(AV^n+BV^n),\\
        (I - \theta\Delta tA)\hat{V}^{n+1} = W_1 + \theta\Delta tB\hat{V}^{n+1} + \Delta t((1-\theta)g_{n} + \theta g_{n+1}),\\
        W_2 = V^n + \tfrac{1}{2}\Delta t(AV^n+BV^n)+ (\tfrac{1}{2}-\theta)\Delta t(A\hat{V}^{n+1}+B\hat{V}^{n+1}),\\
        (I - \theta\Delta tA)V^{n+1} = W_2  + \theta\Delta tBV^{n+1} + \tfrac{1}{2}\Delta t(g_{n} + g_{n+1}).
    \end{cases}        
    \end{equation}
The stability function of the DIRK scheme is
\[R_{\theta}(z)=\frac{1+(1-2\theta)z+(\frac 12 -2\theta+\theta^2)z^2}{(1-\theta z)^2} \quad (z\in\mathbb{C}),\] which is $A$-stable whenever $\theta\geq \frac 14$ (see, e.g., \cite{CashDIRK}).
The following lemma is key to the proof of the stability of the scheme.
Define the error growth function of the DIRK scheme by $G_{\theta}(x) = \sup_{Re(z)\leq x}|R_{\theta}(z)|$.
\begin{lemma}\label{egf}
    Let $\theta\in [\frac 14,\frac 12]$ and $\nu \in \, ]0,\frac{1}{\theta}[$. Then, 
    \begin{equation*}
        \begin{cases}
            G_{\theta}(x) \leq 1 & \text{if }\, x\leq 0,\\
            G_{\theta}(x) = R_{\theta}(x) & \text{if }\, 0\leq x< \frac{1}{\theta}.
        \end{cases}
    \end{equation*}
    Moreover,
    \[
    G_{\theta}(x)\leq  1 + \frac{R_{\theta}(\nu)-1}{\nu}\,x   \quad \text{ for }\, 0\leq x\leq \nu <\frac{1}{\theta}.
    \]
\end{lemma}
 \begin{proof}
    The technical proof of this lemma is given in Appendix \ref{AppB}. 
    \end{proof}

Now, we can derive the following stability result for the DIRK scheme.
\begin{theorem}
    Let $\theta\in [\frac 14,\frac 12]$ and $\nu \in \, ]0,\frac{1}{\theta}[$. The DIRK scheme \eqref{DIRKDirichlet} is unconditionally stable in the $\ell_2$-norm: 
    \[
||R_{\theta}(\Delta t(A+B))||_{2}^n\leq e^{\gamma T\widehat{C}^+}\; \text{whenever}\quad n=0,1,2,\ldots, \;\text{with}\;  n\Delta t\leq T,\;\;  \widehat{C}\Delta t\leq \nu,
\]
where $\gamma$ is a constant independent of $\Delta t$, $\Delta x$ and $\Delta y$.
\end{theorem}
\begin{proof}
There holds \[||R_{\theta}(\Delta t(A+B))||_{2}\leq G_{\theta}(\widehat{C}\Delta t)\leq G_{\theta}(\widehat{C}^+\Delta t).
\]
Let $\gamma = \frac{R_{\theta}(\nu)-1}{\nu}$. Applying Lemma \ref{egf} yields $||R_{\theta}(\Delta t(A+B))||_{2}^n\leq (1+\gamma\widehat{C}^+\Delta t)^n\leq e^{\gamma T\widehat{C}^+}$.
\end{proof}

For the DIRK scheme \eqref{DIRKDirichlet} with $\theta\in [\frac 14,\frac 12]$ the same second-order convergence result holds as Theorem \ref{Theo5.7} for the Crank--Nicolson scheme. A direct adaptation of the proof of Theorem \ref{Theo5.7} to the case of the DIRK scheme yields only first-order convergence, however (when $\theta \not= \frac 12$). Still, second-order convergence can be proved by invoking a key condition of \cite[Lemma 5.1]{Hundsdorfer92}, see also \cite[Sect. II.2.3]{HV}, which employs a special structure of the local error. For the sake of brevity, we omit the details of the proof of second-order convergence for the DIRK scheme in this paper. We mention that similar convergence proofs, using Hundsdorfer's Lemma, have been given in, e.g., \cite{Hundsdorfer02} and \cite{INTHOUTWyns} for closely related (splitting) schemes.

    \section{Numerical experiments} 
    \label{sec6} 
    In this section, we present a series of numerical experiments for European call options and swing options. 
    A main objective is to experimentally validate the proposed numerical schemes. For the European call option, we investigate the convergence behaviour of the total and temporal errors, see Section \ref{sec6.2}, which gives us an indication of the convergence behaviour for swing options. 
    The numerical valuation algorithm for swing options is described in Section \ref{sec6.3}.
    The subsequent study for swing options deals with the convergence behaviour of the total error, the Delta Greeks and the optimal exercise policy, see Sections \ref{sec6.4}, \ref{sec6.5} and \ref{sec6.6} respectively. 

     In the following, we apply cell averaging in the definition of the initial vector $V(0)=V^0$ corresponding to the option payoff, because relying fully on its pointwise evaluation can lead to a deteriorated spatial convergence behaviour. Let 
    \begin{align*}
        x_{i+1/2} = \frac{1}{2}(x_i+x_{i+1}) & \quad \text{and} \quad y_{j+1/2} = \frac{1}{2}(y_j+y_{j+1})\\
        \Delta x_{i+1/2} = x_{i+1/2} - x_{i-1/2} & \quad \text{and} \quad \Delta y_{j+1/2} = x_{j+1/2} - x_{j-1/2},
    \end{align*}
    with $x_{-1/2} = 2x_0-x_{1/2}$, $x_{m_1+1/2} = x_{m_1}$ and $y_{-1/2} = 2y_0-y_{1/2}$, $y_{m_2+1/2} = y_{m_2}$. Then, we define
    $$V_{i,j}(0) = \frac{1}{\Delta x_{i+1/2}\Delta y_{j+1/2}}\int_{x_{i-1/2}}^{x_{i+1/2}}\int_{y_{j-1/2}}^{y_{j+1/2}}\max\{x+y-K,0\}dxdy$$
    whenever the cell $[x_{i-1/2},x_{i+1/2}[\times[y_{j-1/2},y_{j+1/2}[$ has a non-empty intersection with the line segment $x+y=K$.

    To solve the linear systems arising in each time step of the temporal discretisation schemes, we adopt different strategies, depending on the approach. In the semi-Lagrangian approach, the resulting linear system involves the simple tridiagonal matrix $ I - \tfrac{1}{2} \Delta tA^D $, which allows for an efficient direct solution via LU factorisation. In contrast, in the semidiscretisation approach, we use the BiCGSTAB iterative method to solve the pertinent linear systems. To enhance its convergence, an incomplete LU threshold pivoting (ILUTP) preconditioner is applied. The initial guess for the BiCGSTAB iteration corresponding to the solution at time level $n$ is taken as $ V^{n-1} $ from the previous time level. 
    All the computations have been made using Matlab version R2024b, on an Intel 13th Gen Intel(R) Core(TM) i7-1370P 1.90 GHz with 16 GB memory.

    \subsection{Financial parameter values}\label{6.1}
    For the numerical experiments, the parameter values in the PIDEs \eqref{Eq:2} and \eqref{Eq:1} are chosen to reflect reasonable electricity price behaviour. Every parameter value corresponds to a yearly time scale.
    
In this section, the mean reversion level $\mu$ is assumed to be constant and equal to $80$. This value is close to the average electricity price over the year 2024 in the Belgian market (EUR/MWh). 
The half-life of the processes $X-\mu$ and $Y$, defined as the time required for them to revert to half of their values, can reasonably be estimated as 30/365 years and 2/365 years, respectively. Using the half-life formula $H_{\alpha} = \frac{\ln(2)}{\alpha}$ and $H_{\beta} = \frac{\ln(2)}{\beta}$ leads to $\alpha \approx 8$ and $\beta \approx 126$.  
For the volatility $\sigma$, we assume that the stationary variance of the process $X$, given by $\frac{\sigma^2}{2\alpha}(1-e^{-2\alpha})\approx \frac{\sigma^2}{2\alpha}$, is approximately $10\%$ of the average price $\mu$. Hence, $\sigma = \sqrt{\frac{\mu\alpha}{5}}\approx 11$.
Next, assuming an average of one jump per week, we set $\lambda = 52$.  

To test the numerical robustness of the schemes and to check that they perform well also for more extreme market values, we consider additional sets of parameters where there is a higher volatility $\sigma$ and more frequent occurrences of jumps (larger $\lambda$) as well as lower volatility combined with fewer occurrences of jumps (smaller $\lambda$). 
We deal with two finite activity jump models characterised by the following jump density functions:

\begin{itemize}
    \item \textbf{Merton-type jump}, with a normally distributed jump size:
    $$
    f(y) = \frac{1}{\sigma_J \sqrt{2\pi}} \exp\left(-\frac{(y - \mu_J)^2}{2\sigma_J^2}\right) \quad (y\in \mathbb{R}),
    $$
    where $\mu_J$ and $\sigma_J$ denote the mean and standard deviation of the jump sizes, respectively.
    
    \item \textbf{Kou-type jump}, with a double-exponential distribution:
    $$
    f(y) = p \eta_1 e^{-\eta_1 y} \mathbf{1}_{\{y \geq 0\}} + (1 - p)\eta_2 e^{\eta_2 y} \mathbf{1}_{\{y < 0\}}\quad (y\in \mathbb{R}),
    $$
    where $p \in [0,1]$ is the probability of an upward jump, and $\eta_1, \eta_2 > 0$ control the decay rates of the jump sizes in the positive and negative directions, respectively.
\end{itemize}

The jump parameter values are selected to allow for significant spikes in the asset price. Accordingly, the truncation of the domain in the $y$-direction is taken to be sufficiently large to accurately capture the influence of such large jumps.

    \begin{table}[!h]
        \caption{Parameter sets for the Merton-type jump case. The time is measured in years.}
        \centering
        \begin{tabular}{cccccccccccccc}
            parameters & $\mu$ & $\alpha$ & $\beta$ & $\sigma$& $r$ & $\lambda$ & $\mu_J$  & $\sigma_J$ & $K$ & $x_{\min}$& $x_{\max}$& $y_{\min}$&$y_{\max}$\\
           Set 1  & 80 & 8 & 126 & 11 &0.03& 52 & 20 & 60 & 50 & -2$K$&5$K$&-15$K$&15$K$\\
           Set 2  & 80 & 8 & 126 & 20 &0.03&  100&  20& 60&50& -2$K$&5$K$&-15$K$&15$K$\\
           Set 3  & 80 & 8 & 126 & 2 &0.03&  10&  20& 60&50& -2$K$&5$K$&-15$K$&15$K$
        \end{tabular}
        \label{paramSets}
    \end{table}
    \begin{table}[!h]
        \caption{Parameter sets for Kou-type jump case. The time is measured in years.}
        \centering
        \begin{tabular}{ccccccccccccccc}
            parameters &$\mu$& $\alpha$ & $\beta$ & $\sigma$& $r$ & $\lambda$ & $p$& $\eta_1$& $\eta_2$  & $K$ & $x_{\min}$& $x_{\max}$& $y_{\min}$&$y_{\max}$\\
           Set 4  &80& 8 & 126 & 11 &0.03& 52 & 0.6 & 0.01 & 0.02& 50 & -2$K$&5$K$&-20$K$&20$K$\\
           Set 5  &80& 8 & 126 & 20 &0.03& 100 & 0.6 & 0.01 & 0.02& 50 & -2$K$&5$K$&-20$K$&20$K$\\
           Set 6  &80& 8 & 126 & 2 &0.03& 10 & 0.6 & 0.01 & 0.02& 50 & -2$K$&5$K$&-20$K$&20$K$\\
        \end{tabular}
        \label{paramSetsKou-type}
\end{table}

\subsection{Convergence behaviour: European call option}\label{sec6.2}
    In this section, we numerically examine the convergence behaviour of the three schemes formulated in Section \ref{TempScheme} in the case of a European call option for the six different parameter sets given by Tables \ref{paramSets} and \ref{paramSetsKou-type}. We take the number of the spatial grid points $m_1=m_2=m$ and consider two types of discretisation errors: 
    \begin{itemize}
        \item The \emph{total discretisation error} on the region of interest defined by \[E_T(N,m) = \max\{|V^{N}_{i,j} - v(x_i,y_j,T)|\mid (x_i,y_j)\in [-\half K,\frac{3}{2}K]\times[-K,K]\}.\] We will study this error for a sequence of values $N$ and $m$ that are directly proportional to each other. More precisely, we take $N= \left \lceil{\frac{m}{2}}\right \rceil$ and consider the total error for $20$ different values of $m$ between $50$ and $500$. The reference solution for $v(\cdot,\cdot,T)$ is computed by applying the CNFI scheme \eqref{CNfpi} with $N=m=1500$.
        \item The \emph{temporal discretisation error} defined by \[E(N,m) = \max\{|V^{N}_{i,j} - V_{i,j}(T)| \mid 0\leq i,j\leq m \}.\] For this discretisation error, we consider only the semidiscretisation approach, excluding the semi-Lagrangian approach, as for the latter the temporal error is not clearly defined. A reference solution for $V(T)$ is computed by applying the CNFI scheme with $N=6000$ time steps. The temporal error is considered for $20$ different values of $N$ between $100$ and $1000$. For the number of spatial grid points, $m = 200$ is taken.
    \end{itemize}
The maturity time for the European call option is set to $T = \frac{1}{10}$ in agreement with the small interval between two consecutive action times in the case of swing options.
 
The results for the total and temporal errors are displayed in Figures \ref{TotalError1} and \ref{TemporalError1} for the Merton-type jump model and in Figures \ref{TotalError2} and \ref{TemporalError2} for the Kou-type jump model. The numerical schemes considered are SLCNFI \eqref{slcnfi}, CNFI \eqref{CNfpi}, and DIRKFI \eqref{DIRK} with $\theta = 1 - \frac{\sqrt{2}}{2}$. In the semidiscretisation approach, we choose the QUICK scheme \eqref{quickx}-\eqref{quicky} for the convection terms.

Figures \ref{TotalError1} and \ref{TotalError2} show that, for all schemes and all parameters sets, the {\it total error} decreases monotonically as $N$ and $m$ increase in a directly proportional way. 
For the parameter sets 1, 2, 4, 5, a favourable second-order convergence behaviour is observed. 
Further, for each of these four sets, CNFI and DIRKFI are seen to have about the same error constants, which are always smaller than that for SLCNFI, and hence, CNFI and DIRKFI are to be preferred over SLCNFI.

The sets 3 and 6 represent highly convection-dominated problems as the volatility $\sigma$ is small and the jump intensity $\lambda$ is low. These characteristics lead to the emergence of a region of nonsmoothness where the solution $v$ has steep gradients. In this situation, the convergence order for the total error of CNFI and DIRKFI reduces: it is (asymptotically) equal to 1.6. 
On the other hand, for SLCNFI the convergence order remains (asymptotically) equal to two. Further, its error constant is smaller than that for CNFI and DIRKFI.
Hence, for these two sets, SLCNFI is preferred.
We remark, however, that sets 3 and 6 are less representative of realistic market situations, since electricity prices typically experience significant fluctuations and frequent jumps, which are not captured well by a small volatility and low jump intensity.

Figures \ref{TemporalError1} and \ref{TemporalError2} display the {\it temporal errors} for the CNFI and DIRKFI schemes for, respectively, the Merton- and Kou-type jump models. The favourable result is observed that, for all six parameter sets, second-order convergence holds.
Additional experiments have been carried out with larger numbers of spatial grid points ($m = 300, 400$) and the obtained temporal errors are found to be essentially unaffected.
This is a desirable property of the temporal error and is often referred to in the literature as convergence in the stiff sense.

Concerning the temporal error constant, this is seen to be noticeably smaller for DIRKFI than for CNFI. We note here that DIRKFI involves two fixed-point iteration processes per time step, thus requiring approximately twice the computational effort of CNFI per time step. 

Additional experiments reveal that CNFI may show, however, unstable behaviour for larger time steps, even when Rannacher time-stepping (backward Euler damping) is applied. 
Such unstable behaviour has not been observed in our experiments with SLCNFI and DIRKFI, which forms a favourable property of the latter schemes.
\vfill\clearpage

\begin{figure}[!h]
    \centering
    \includegraphics[trim=70 0 0 0, clip,width=0.7\linewidth]{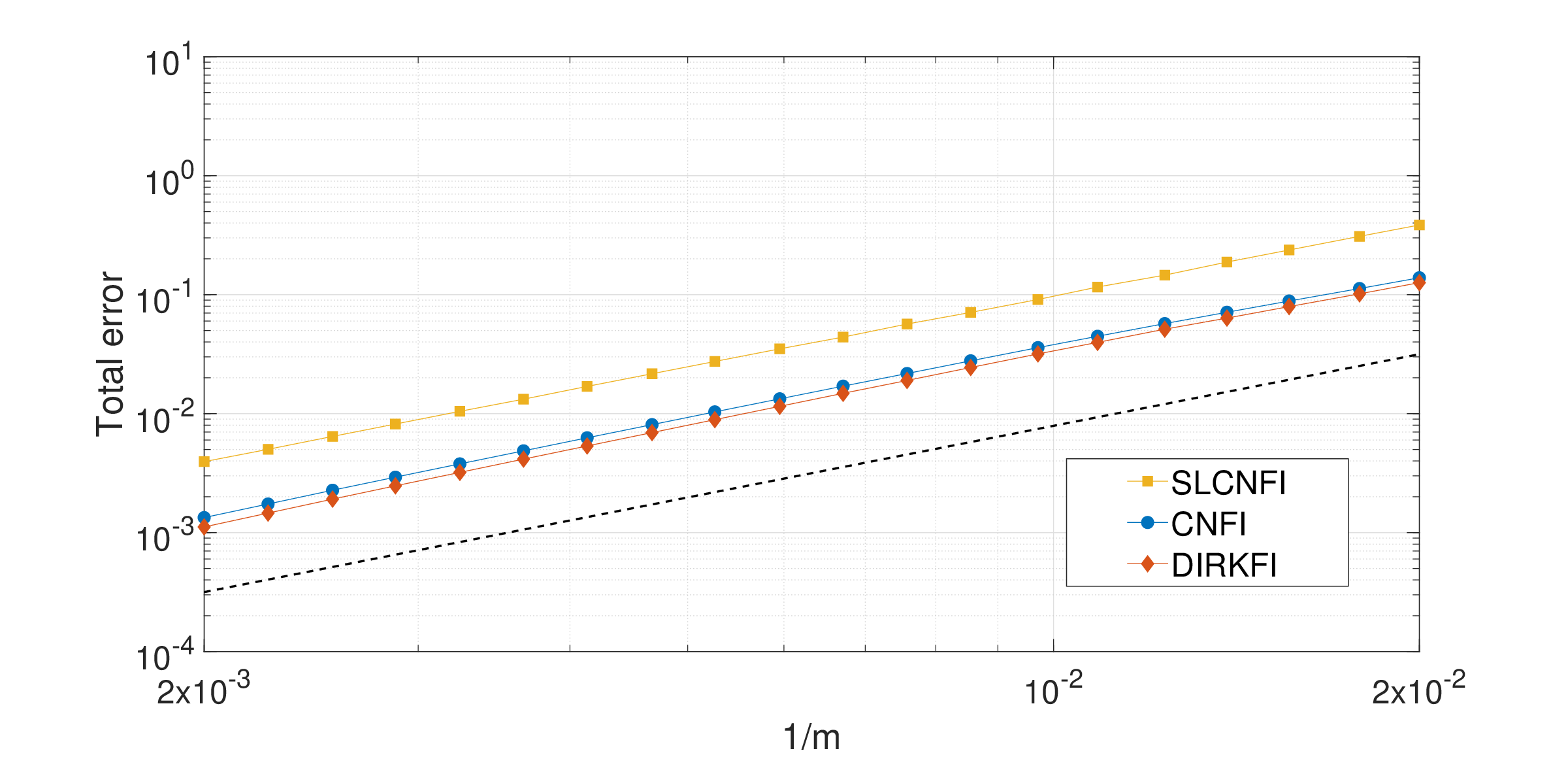}
    \includegraphics[trim=70 0 0 0, clip,width=0.7\linewidth]{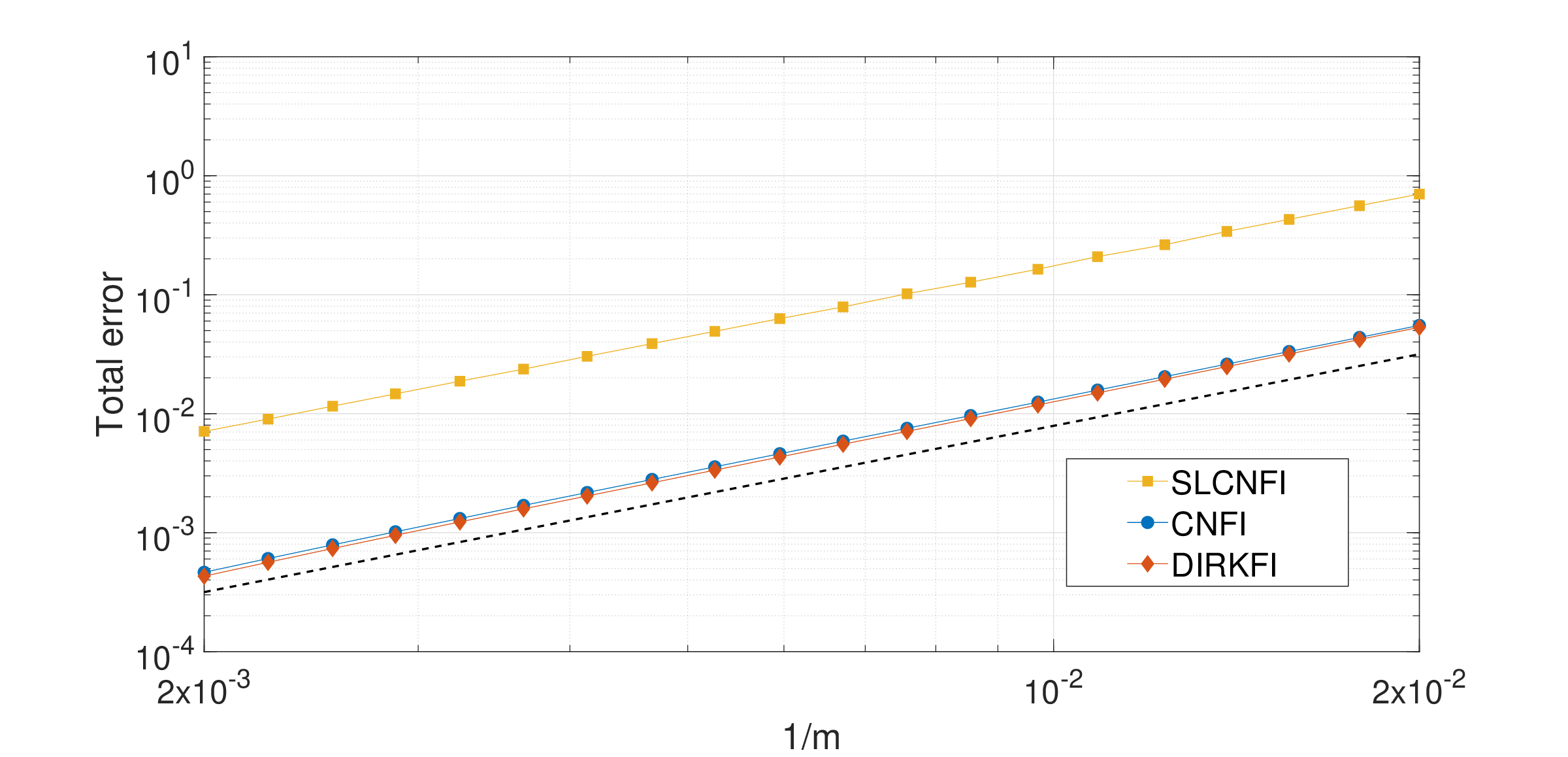}
    \includegraphics[trim=70 0 0 0, clip,width=0.7\linewidth]{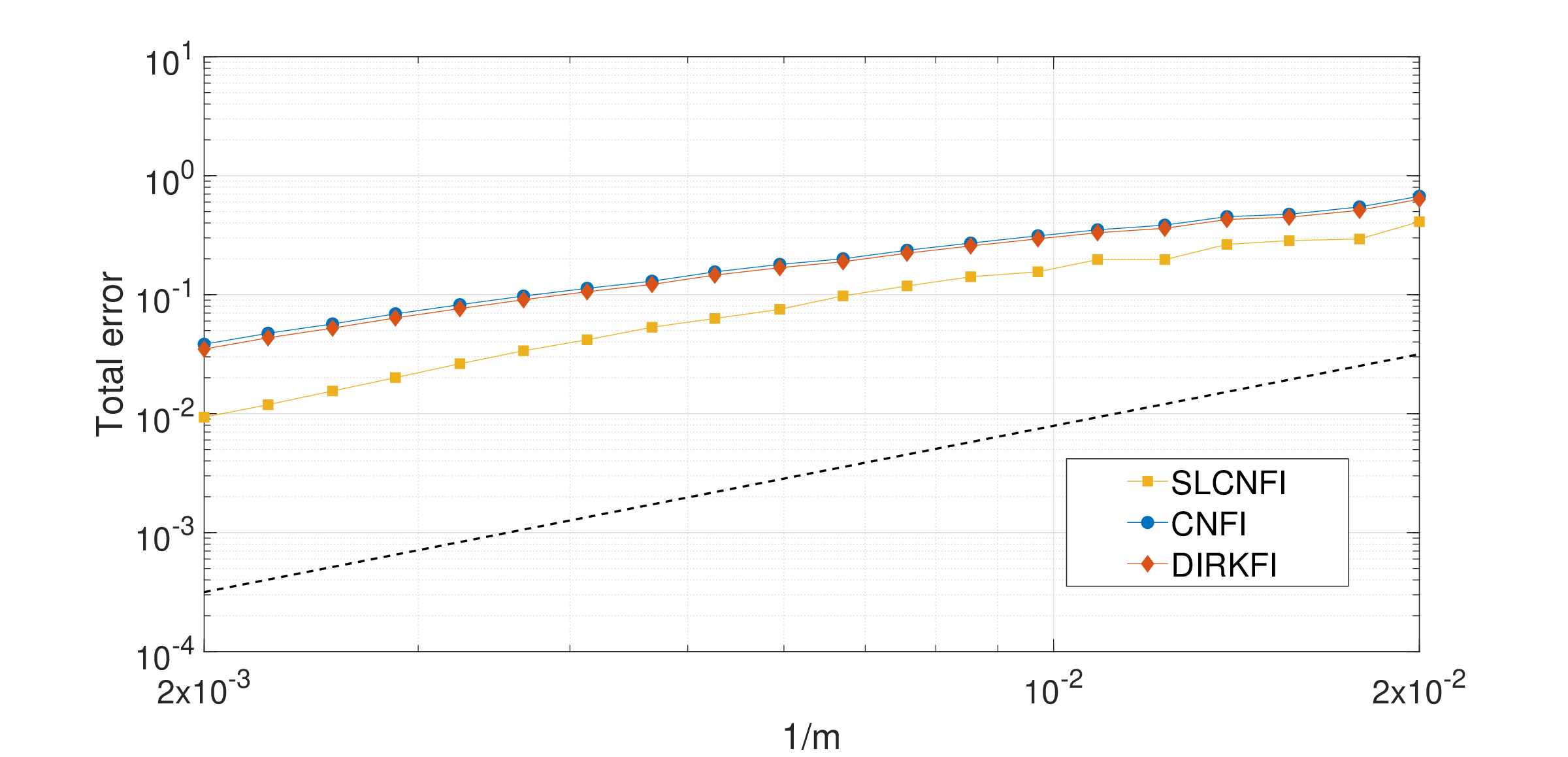}
    \caption{European call option in the Merton-type jump model. Total discretisation errors of the SLCNFI, CNFI and DIRKFI schemes for $N= \left \lceil{\frac{m}{2}}\right \rceil$ and set 1 (top), set 2 (middle), set 3 (bottom). Added: dashed reference line for convergence order 2.}
    \label{TotalError1}
\end{figure}
\vfill\clearpage
\begin{figure}[!h]
    \centering
    \includegraphics[trim=70 0 0 0, clip,width=0.7\linewidth]{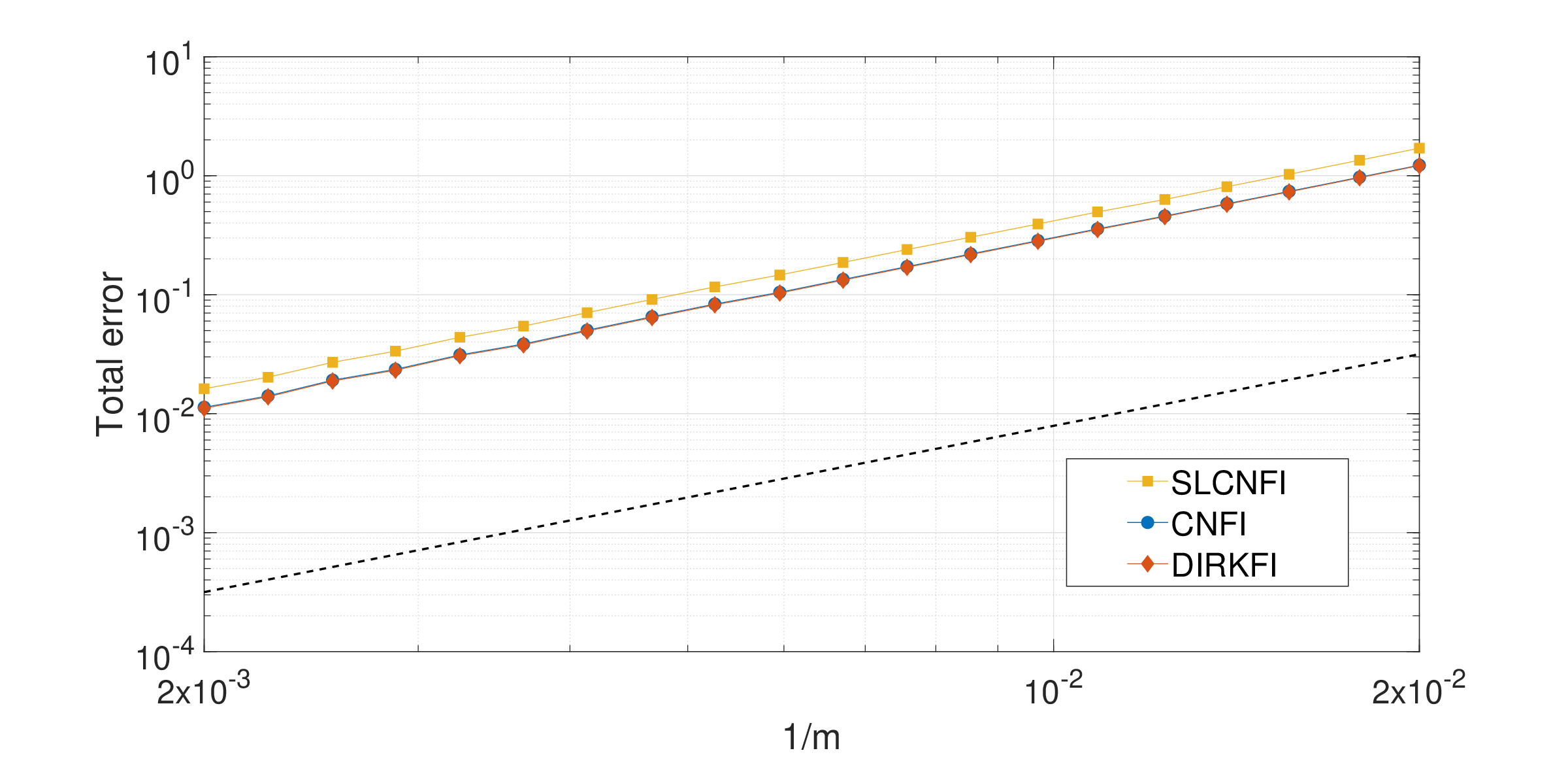}
    \includegraphics[trim=70 0 0 0, clip,width=0.7\linewidth]{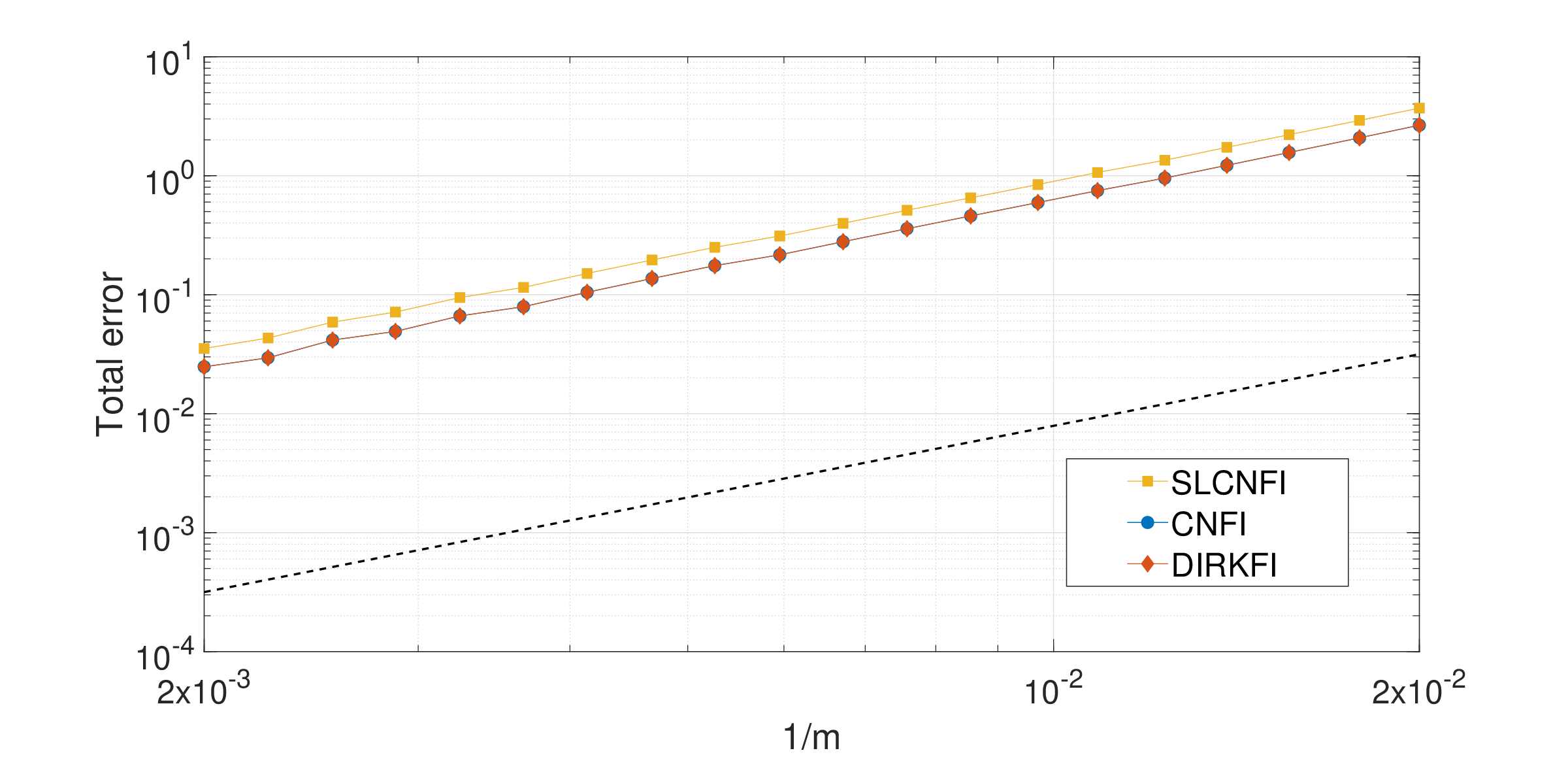}
    \includegraphics[trim=70 0 0 0, clip,width=0.7\linewidth]{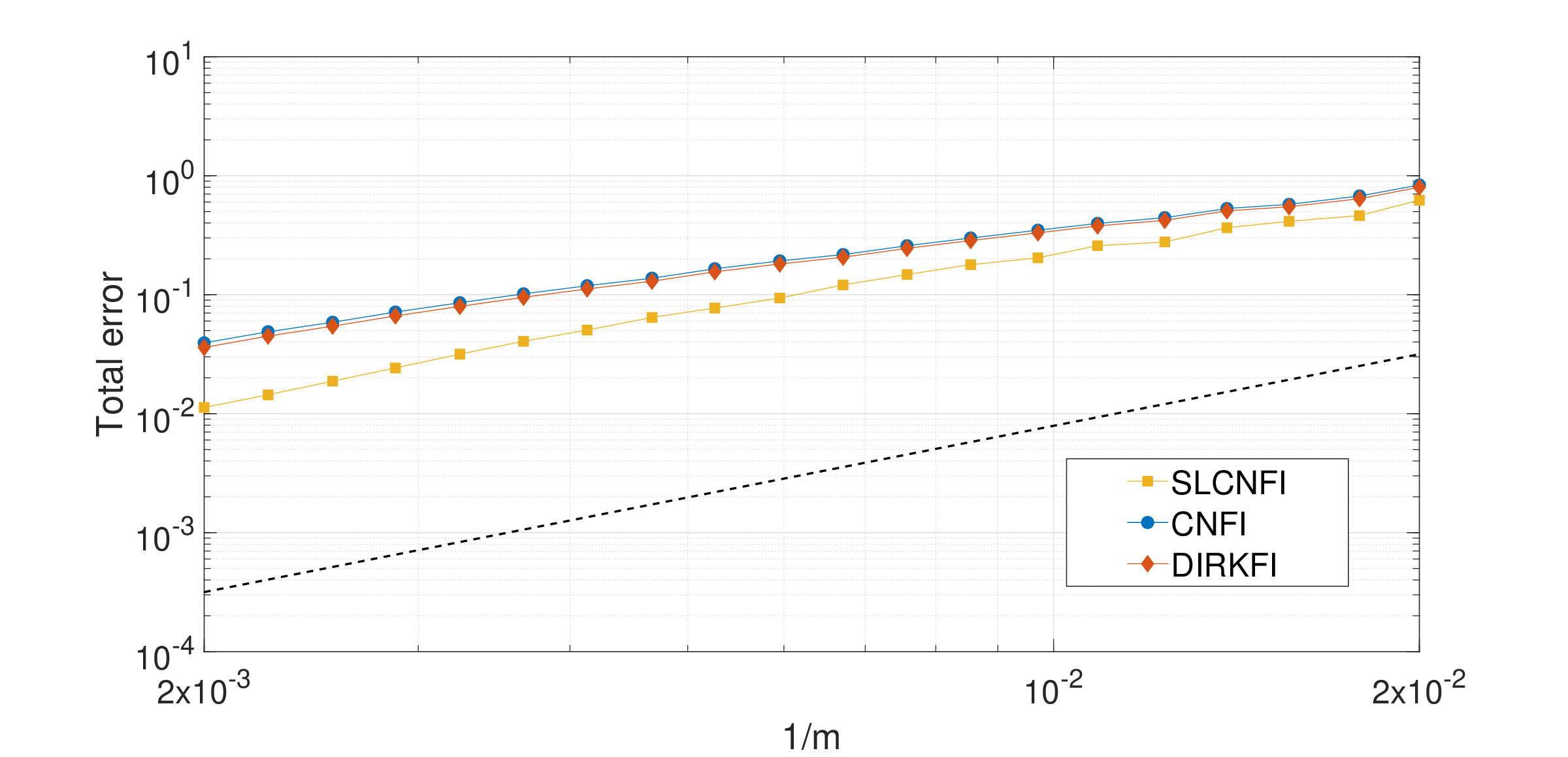}
    \caption{European call option in the Kou-type jump model. Total discretisation errors of the SLCNFI, CNFI and DIRKFI schemes for $N= \left \lceil{\frac{m}{2}}\right \rceil$ and set 4 (top), set 5 (middle), set 6 (bottom). Added: dashed reference line for convergence order 2.}
    \label{TotalError2}
\end{figure}
\vfill\clearpage
\begin{figure}[!h]
    \centering
    \includegraphics[trim=70 0 0 0, clip,width=0.7\linewidth]{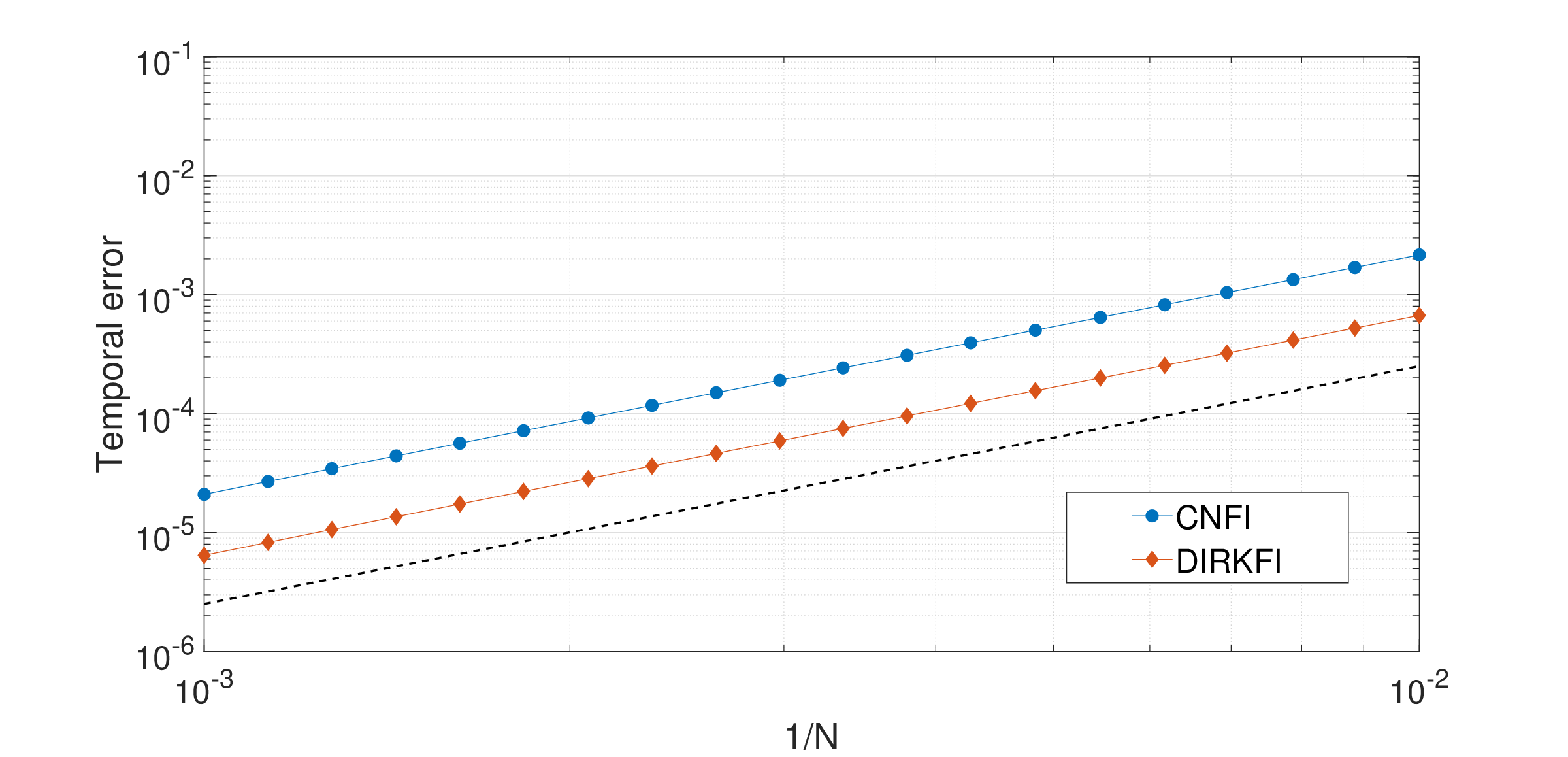}
    \includegraphics[trim=70 0 0 0, clip,width=0.7\linewidth]{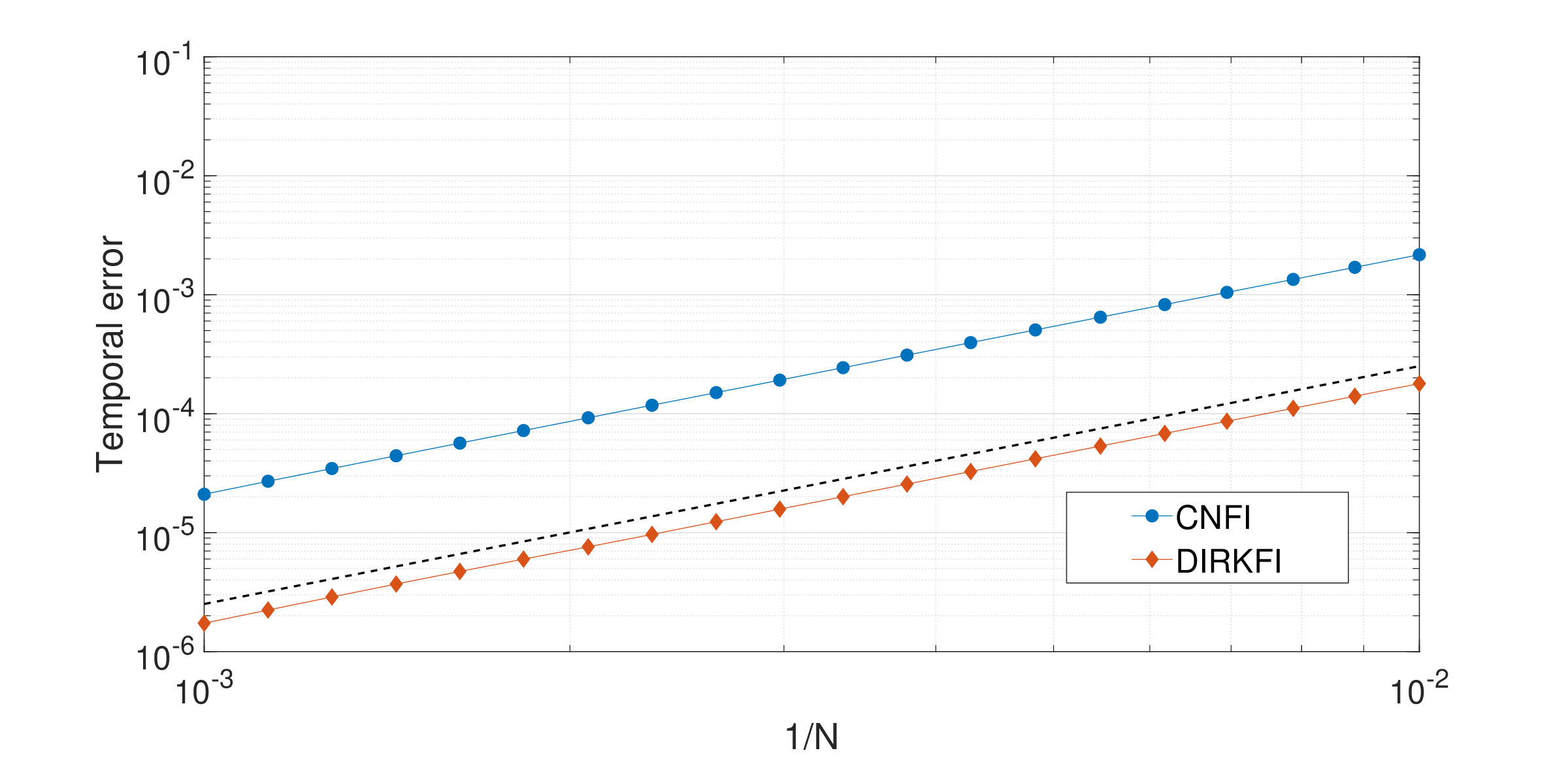}
    \includegraphics[trim=70 0 0 0, clip,width=0.7\linewidth]{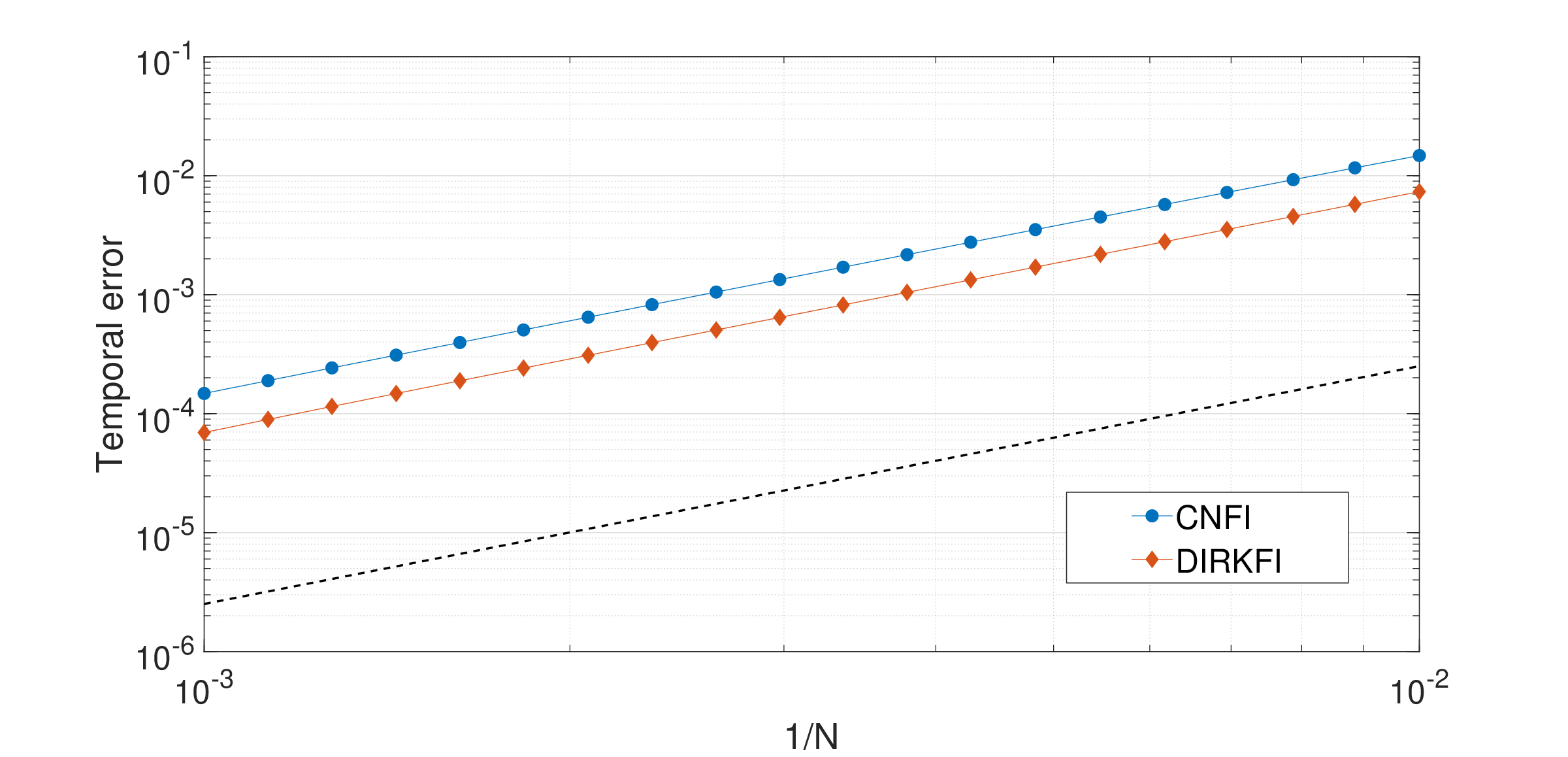}
    \caption{European call option in the Merton-type jump model. Temporal discretisation errors of the CNFI and DIRKFI schemes for $m=200$ and set 1 (top), set 2 (middle), set 3 (bottom). Added: dashed reference line for convergence order 2.}
    \label{TemporalError1}
\end{figure}
\vfill\clearpage
\begin{figure}[!h]
    \centering
    \includegraphics[trim=70 0 0 0, clip,width=0.7\linewidth]{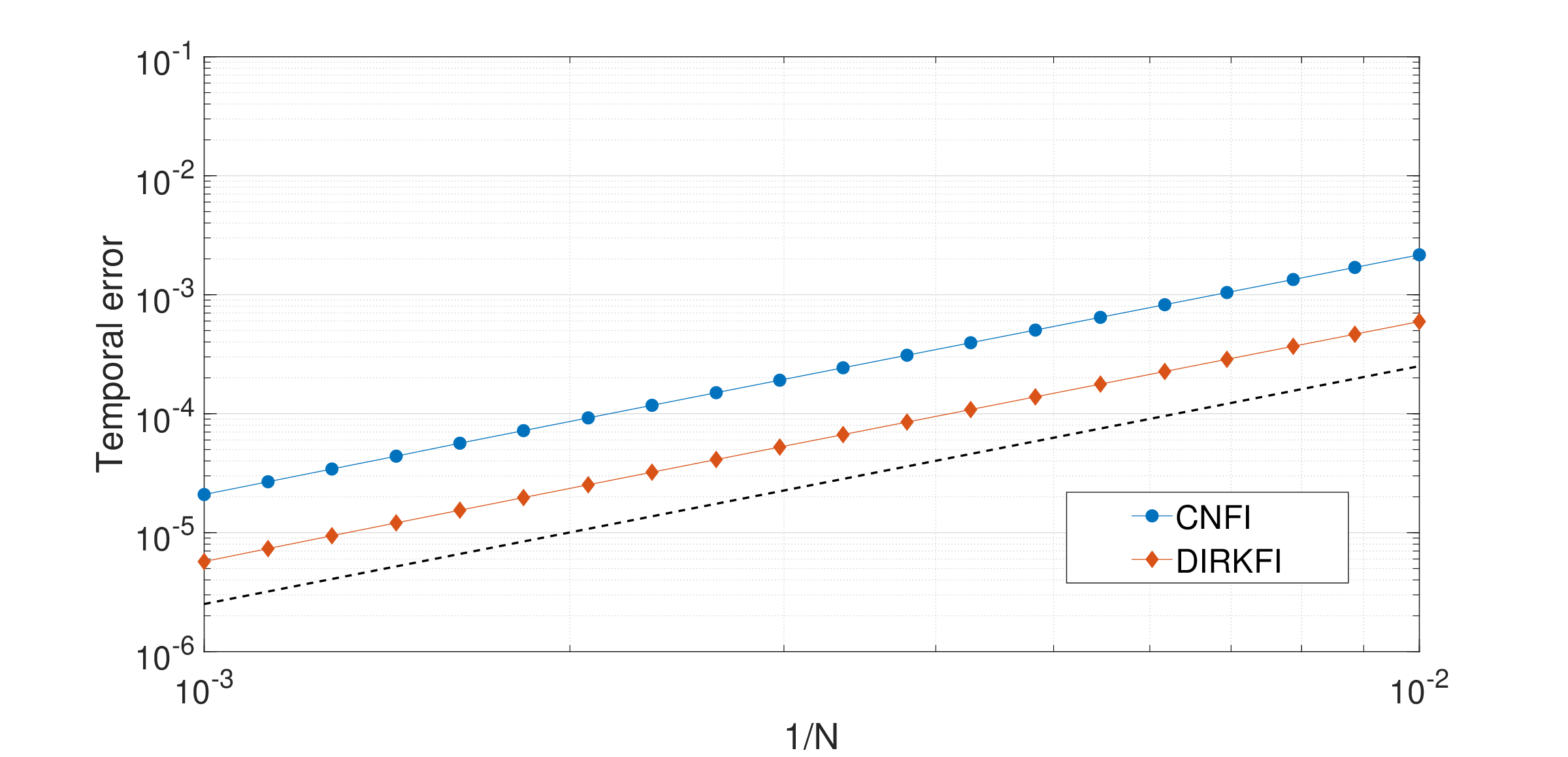}
    \includegraphics[trim=70 0 0 0, clip,width=0.7\linewidth]{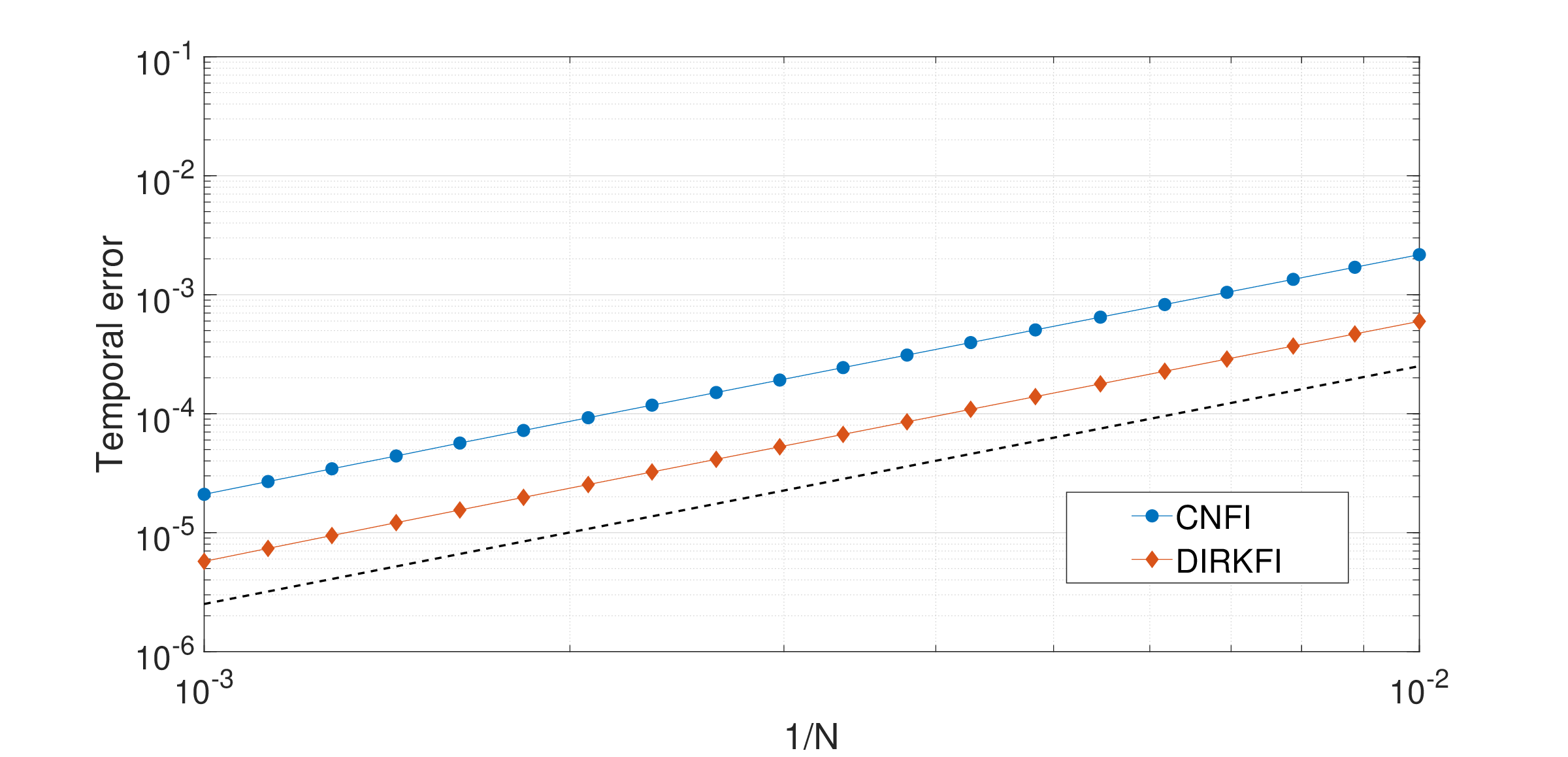}
    \includegraphics[trim=70 0 0 0, clip,width=0.7\linewidth]{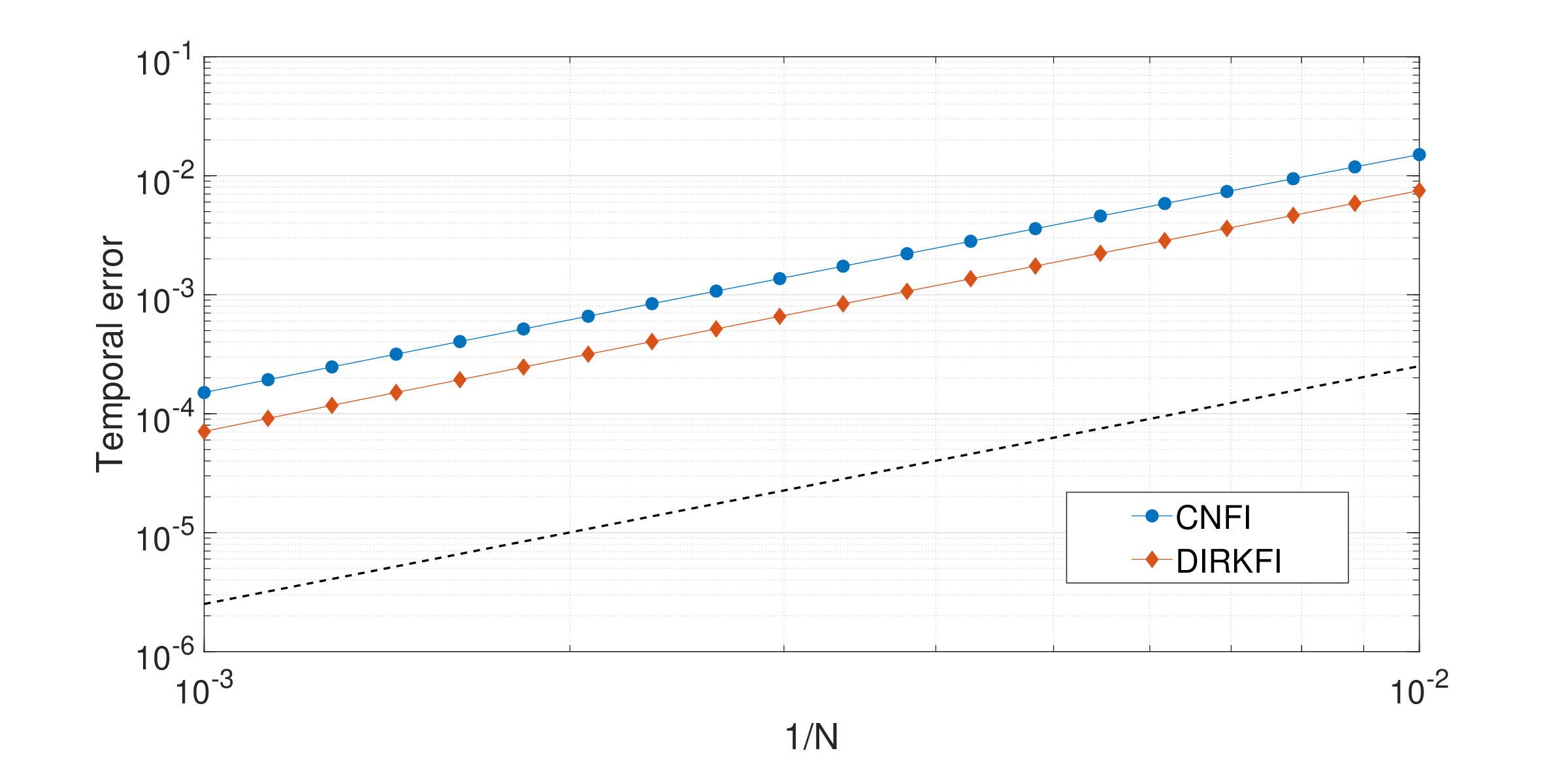}
    \caption{European call option in the Kou-type jump model. Temporal discretisation errors of the CNFI and DIRKFI schemes for $m=200$ and set 4 (top), set 5 (middle), set 6 (bottom). Added: dashed reference line for convergence order 2.}
    \label{TemporalError2}
\end{figure}
\vfill\clearpage

\subsection{Numerical valuation of swing options}\label{sec6.3}
For the numerical valuation of swing options, we combine the numerical schemes proposed for the European call option case with the dynamic programming principle to find the optimal exercise at each action time. As before, we reverse the time to obtain an initial condition instead of a terminal condition. The action time $\tau_n$ in reversed time is $T - T_{N_a-n + 1}$ ($n = 1,2,\ldots, N_a$) and the sequence of PIDEs \eqref{Eq:2} then becomes
    \begin{equation}
        \begin{cases}
            \displaystyle \partial_{t}v(x,y,z,t) = \frac{\sigma^2}{2}\partial_{xx}v(x,y,z,t) + \alpha(\mu - x) \partial_{x}v(x,y,z,t) -\beta y\partial_{y}v(x,y,z,t) - (r+\lambda)v(x,y,z,t) \\[3mm]
            \displaystyle \qquad \qquad \qquad \qquad \qquad+\lambda\int_{\mathbb{R}}v(x,y +\xi,z,t)f(\xi)d\xi, \qquad \tau_{n}< t < \tau_{n+1}, \\[3mm]
            \displaystyle v(x,y,z,\tau_{n}) = \max \big\{b_n(x+y-K) + v(x,y,z+b_n,\tau_{n}^-)\mid  b_n\in\{0,1,\ldots,L\}, z + b_n\leq M \big\}                                 \\
        \end{cases}
        \label{Eq:3}
    \end{equation}
   for $(x,y,z)\in \mathbb{R}\times\mathbb{R}\times\{0,1,\ldots,M-1\}$ and $1\leq n\leq N_a$. Here $v(\cdot,\cdot,\cdot,\tau_{1}^-) = 0$, $v(\cdot,\cdot,M, \cdot) = 0$ and $\tau_{N_a+1} = T$. 
    
The valuation procedure is outlined in Algorithm \ref{DPP}.
At the first action time, $\tau_{1} = 0$, the holder buys the maximal allowable amount.  
Given the option value at $\tau_1$, the PIDE is solved to obtain the option value up to the second action time, $\tau_{2}$. It is then considered whether it is optimal to exercise or not by maximising the option value at $\tau_{2}$. This process is repeated for each subsequent action time up to and including $\tau_{N_a}$. At each action time, the option value is determined by optimising over all feasible exercise amounts $b\in \{0,1,\ldots,L\}$, while ensuring that the cumulative purchased amount does not exceed the global constraint $M$. Finally, the PIDE is solved over the interval $(\tau_{N_a},\tau_{N_a+1}]$ to arrive at the desired swing option value function $v(\cdot,\cdot,0,T)$.

\begin{algorithm}[!h]
\caption{Dynamic Programming for Swing Option Valuation}
\label{DPP}
\begin{algorithmic}[1]
\State \textbf{Input:} $N_a$ (number of action times), $L$ (local constraint), $M$ (global constraint), $N$ (number of time steps between two successive action times), strike $K$
\State \textbf{Initialise:} $V \gets 0$, optimal policy $b^*_{n,l,i,j} \gets 0$
\For{$n = 1$ to $N_a$} \Comment{$n$-th swing action time}
    \For{$l = 0$ to $M - 1$} \Comment{$l$: cumulative purchased amount}
        \For{$b = 0$ to $\min(L, M - l)$}
            \State $V^{(n-1)N}_{l,i,j} \gets \max\left( b(x_i + y_j - K) + V^{(n-1)N}_{l + b, i, j}, \; V^{(n-1)N}_{l,i,j} \right)$
        \EndFor
        \State Store optimal $b^*_{n,l,i,j}$ that yields the maximum
        \For{$k = 1$ to $N$} 
            \State Compute $V^{(n-1)N + k}_{l}$ by the time-stepping scheme
        \EndFor
    \EndFor
\EndFor
\end{algorithmic}
\end{algorithm}
\begin{algorithm}[!h]
\caption{Tracking back the cumulative optimal exercise path}
\label{optPath}
\begin{algorithmic}[1]
\State \textbf{Input:} Number of swing times $N_a$, global constraint $M$, policy array $b^*$
\State \textbf{Initialize:} $\texttt{path} \gets 0$ \Comment{cumulative optimal path}
\For{$n = 1$ to $N_a$}
    \For{each grid point $(i, j)$}
        \If{$n = 1$}
            \State $l \gets 0$ \Comment{zero units of energy at the start}
        \Else
            \State $l \gets \texttt{path}[i,j, n-1]$
        \EndIf
        \If{$l < M$}
            \State $\delta \gets b^*_{N_a-n+1, l, i, j}$
            \State $\texttt{path}[i,j, n] \gets l + \delta$
        \Else
            \State $\texttt{path}[i,j, n] \gets l$
        \EndIf
    \EndFor
\EndFor
\State \textbf{Output:} $\texttt{path}$
\end{algorithmic}
\end{algorithm}

\subsection{Convergence behaviour: swing options}\label{sec6.4}
 In this section, we study the total discretisation error of the three schemes from Section~\ref{TempScheme} in the numerical valuation of swing options. 
Here the number of swing action times is set equal to $N_a=20$, with local constraint given by $L=1$ and global constraint by $M=10$. Both Merton- and Kou-type jump models are considered, for the parameter sets 1 and 4 given by the Tables \ref{paramSets} and \ref{paramSetsKou-type}, respectively, with maturity time $T=1$.
To assess the convergence rate of the schemes, we successively refine in this section the spatial and temporal grids, where at each refinement level the numbers of grid points in the $x$-, $y$- and $t$-directions are simultaneously doubled.
Denote by $\widetilde{V}(N)$ the approximated swing option value at a given point $(x,y,z,T)$ if $m_1=m_2=N$, where $N$ is the number of time steps between each two consecutive action times.
The numerical convergence rate is then defined by 
\[
\text{rate}  = \log_2 \Bigg(\frac{|\widetilde{V}(N/2)-\widetilde{V}(N/4)|}{|\widetilde{V}(N)-\widetilde{V}(N/2)|}\Bigg). 
\]
Table \ref{refrenceTable1} reports for the two parameter sets 1 and 4 the approximated swing option values at three points $(x,y,0,T)$ using the CNFI scheme with $N=50,100,200,400$. For each set and point, the corresponding numerical convergence rate is added. Clearly, a favourable second-order convergence behaviour is observed. Numerical experiments in the case of the SLCNFI and DIRKFI schemes reveal the same positive result.
\begin{table}
\caption{Swing option values obtained with the CNFI scheme at several points $(x,y,z,T)$ with $z=0$ and $T=1$ for sets 1 and 4 where $N_a=20$, $L=1$, $M=10$. Numerical convergence rate added.}
\label{refrenceTable1}
\begin{center}
\textbf{Set 1}
\end{center}
\begin{center}
\begin{tabular}{c c c c c c c}
\toprule
 $m_1=m_2=N$& $x = 40,y = 5$ & rate& $x = 60, y= -100$ & rate& $x = 80, y= 100$& rate \\
\midrule
$50$ &$501.0748$ & & $512.6952$ & & $528.1959$ & \\
$100$ & $500.8956$& & $512.5162$ &  &$527.9916$& \\
$200$ &  $500.8575$&$2.23$  &$512.4794$&$2.28$  &$527.9464$ &$2.17$ \\
$400$ &  $500.8479$& $1.98 $ &$512.4701 $& $1.98 $  &$527.9356$&$2.05$ \\
\bottomrule
\end{tabular}
\end{center}
\begin{center}
\textbf{Set 4}
\end{center}
\begin{center}
\begin{tabular}{c c c c c c c}
\toprule
 $m_1=m_2=N$& $x = 40,y = 5$ & rate& $x = 60, y= -100$ & rate& $x = 80, y= 100$& rate \\
\midrule
$50$ &$681.8451$ & & $694.4583$ & & $710.0696$ & \\
$100$ & $681.4723 $& & $694.1128$ &  &$709.7014$& \\
$200$ &  $681.3949$&$2.26$  &$694.0408$&$2.26$  &$709.6283$ &$2.33$ \\
$400$ &  $681.3740$& $1.89$ &$694.0210 $& $1.85$  &$709.6078$&$1.85$ \\
\bottomrule
\end{tabular}
\end{center}
\end{table}




\subsection{The Greeks and sensitivity analysis: swing options}\label{sec6.5}
The Greeks represent the sensitivities of the option value with respect to underlying variables and parameters, and are important for hedging purposes. They serve as an essential tool for quantifying the impact of changes in model inputs on the option value. In this section, we consider the Delta Greeks $\Delta_1 = \frac{\partial v}{\partial x} $ and $\Delta_2 = \frac{\partial v}{\partial y}$, which measure the sensitivity of the swing option value with respect to changes in the price components $x$ and $y$, respectively. Approximations to the Deltas are directly obtained at no additional computational cost, as they are naturally incorporated in the finite difference discretisation.
\begin{figure}[!h]
    \centering
    \begin{tabular}{c c}
        \textbf{Set 1} & \textbf{Set 4} \\[0.5em]

        \includegraphics[trim=55 0 0 0, clip,width=0.45\linewidth]{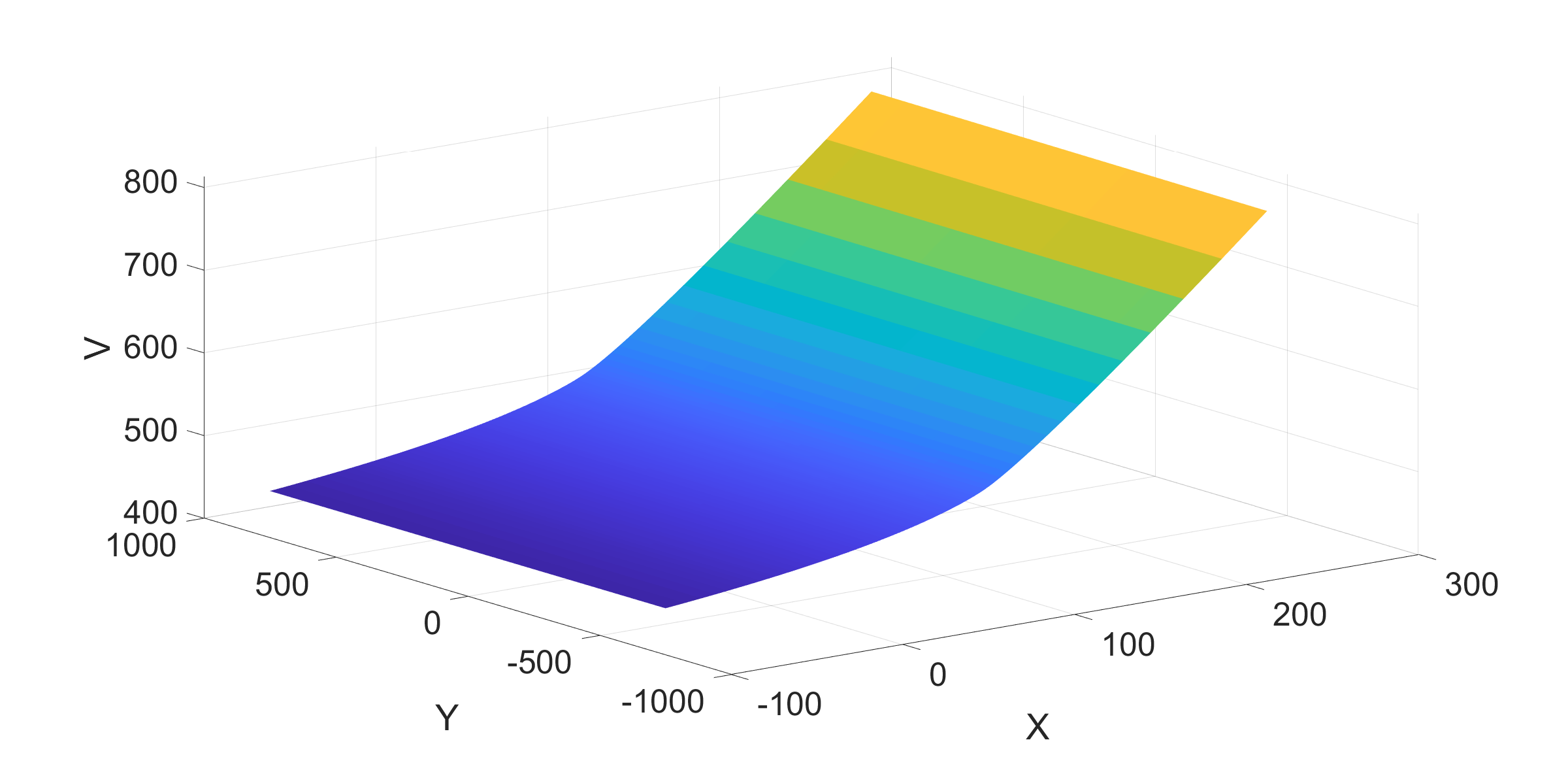} &
        \includegraphics[trim=55 0 0 0, clip,width=0.45\linewidth]{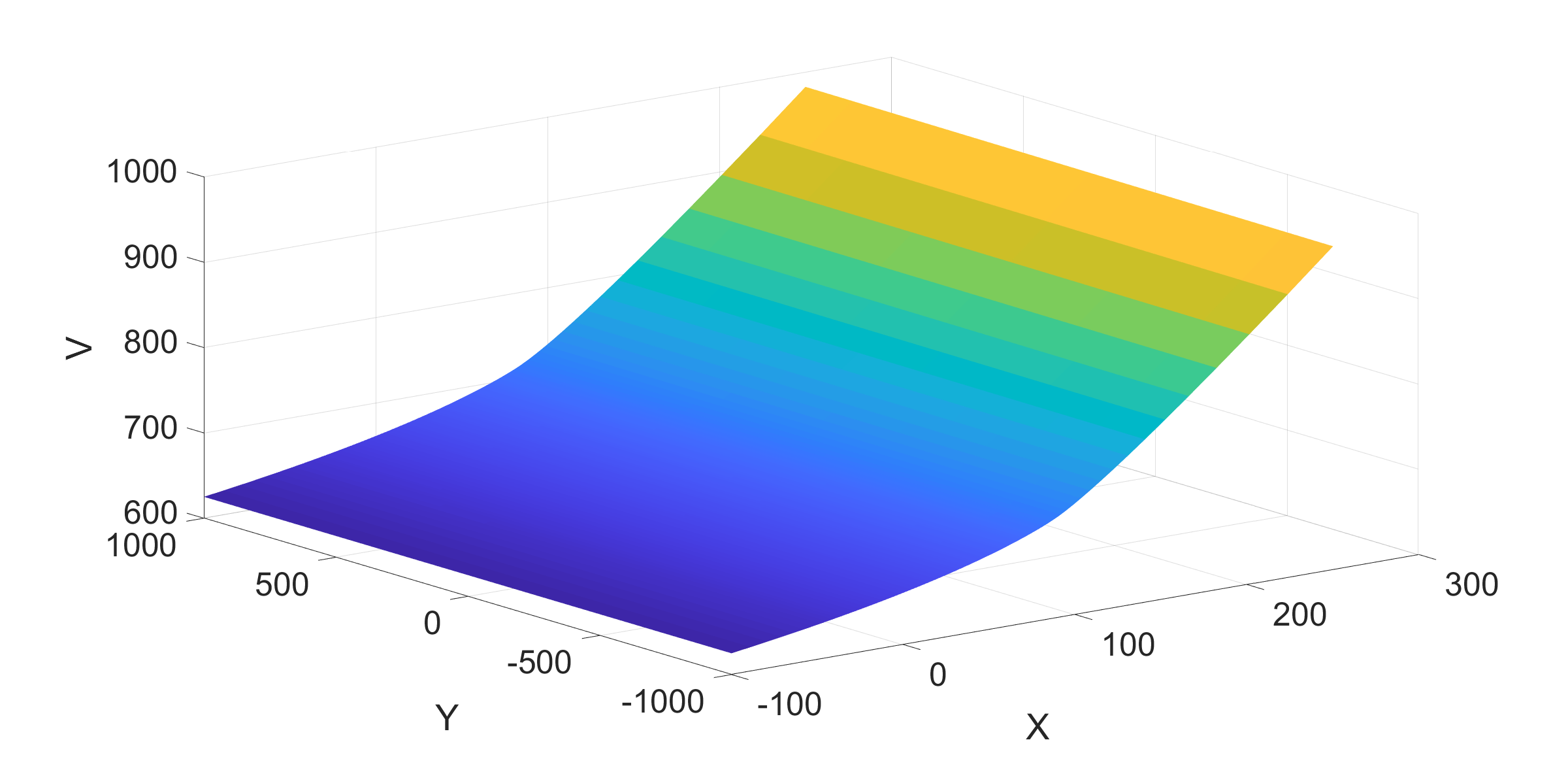}\\[1em]
        
        \includegraphics[trim=55 0 0 0, clip,width=0.45\linewidth]{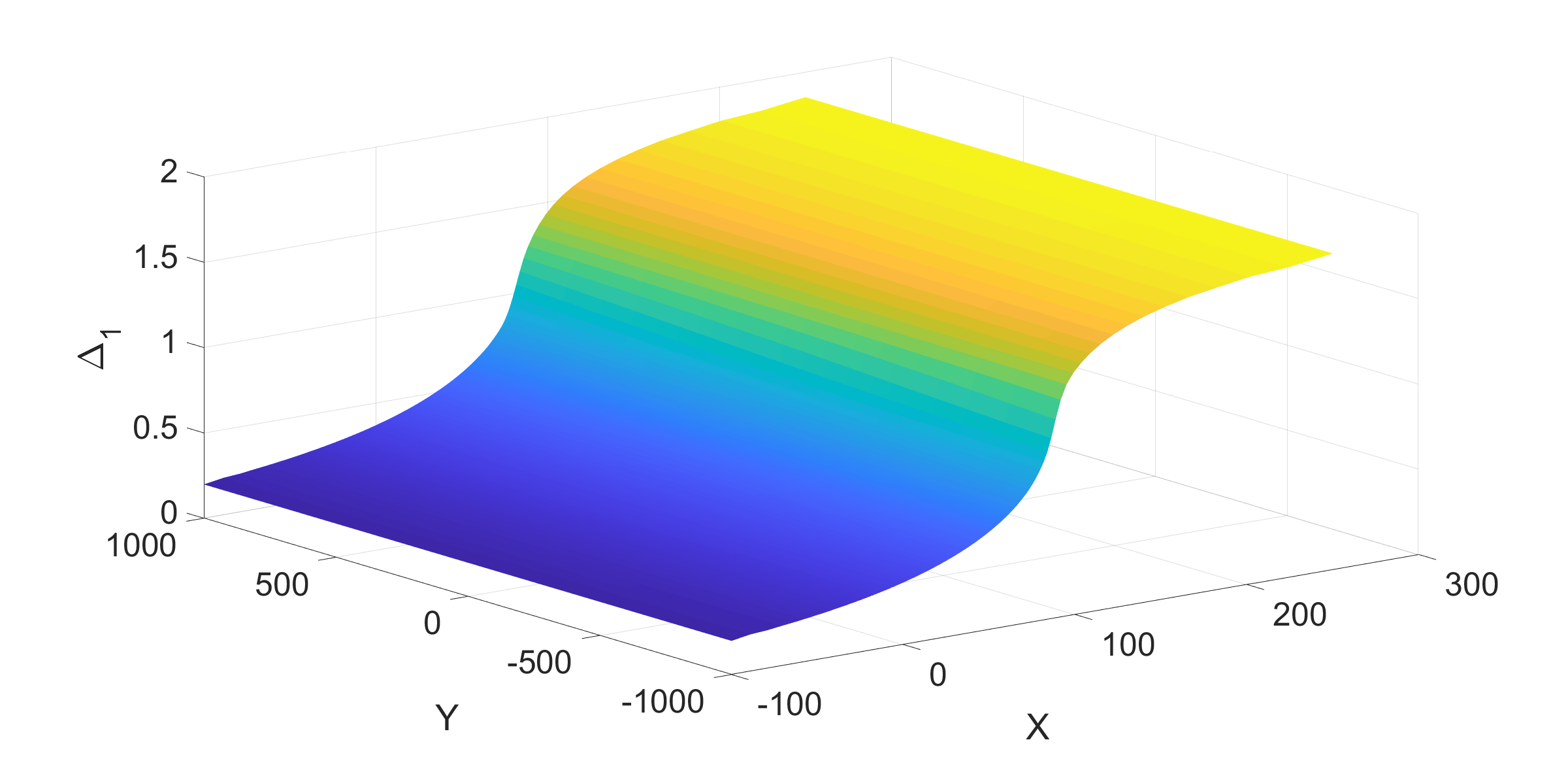} &
        \includegraphics[trim=55 0 0 0, clip,width=0.45\linewidth]{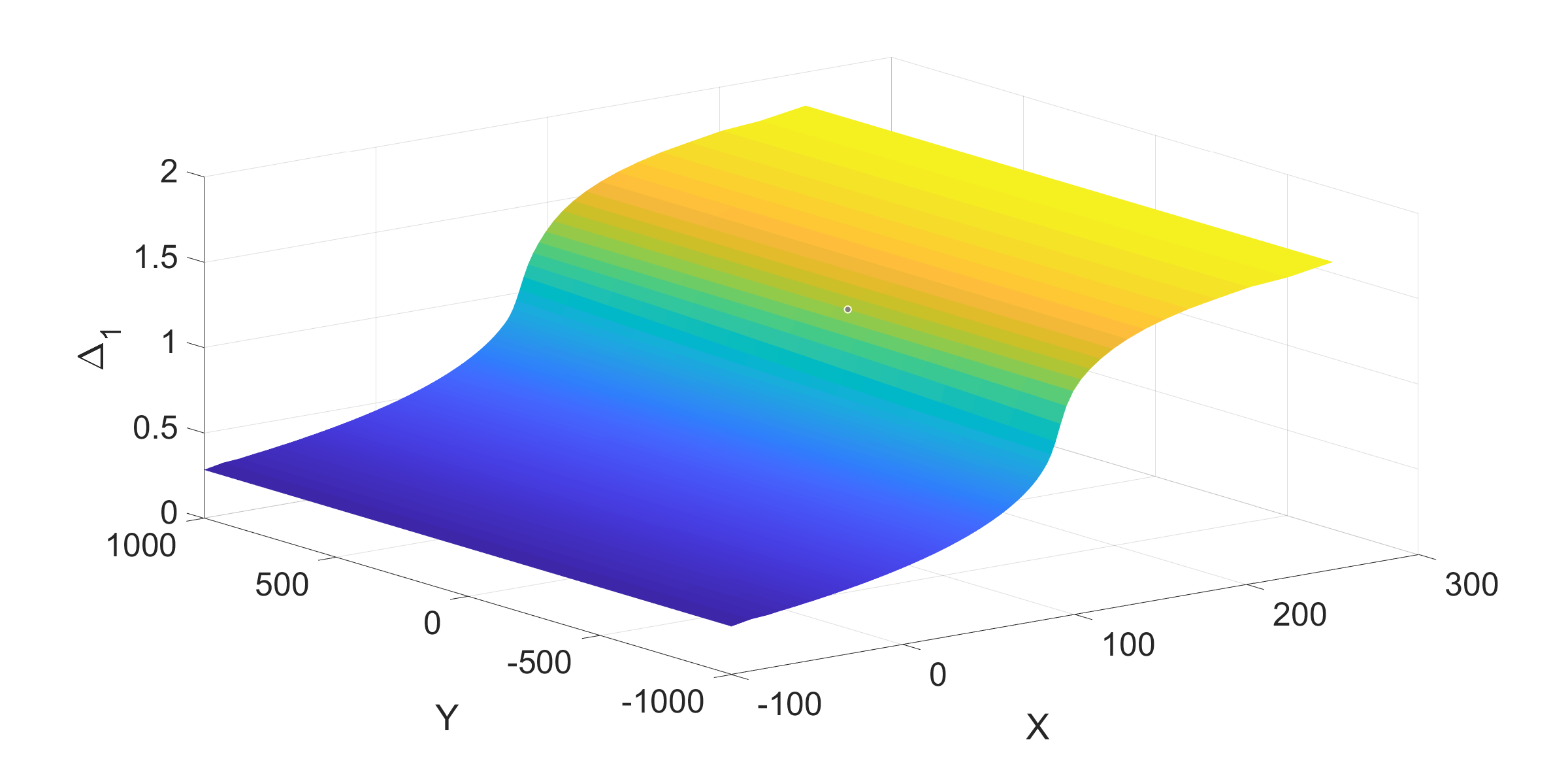} \\[1em]

        \includegraphics[trim=55 0 0 0, clip,width=0.45\linewidth]{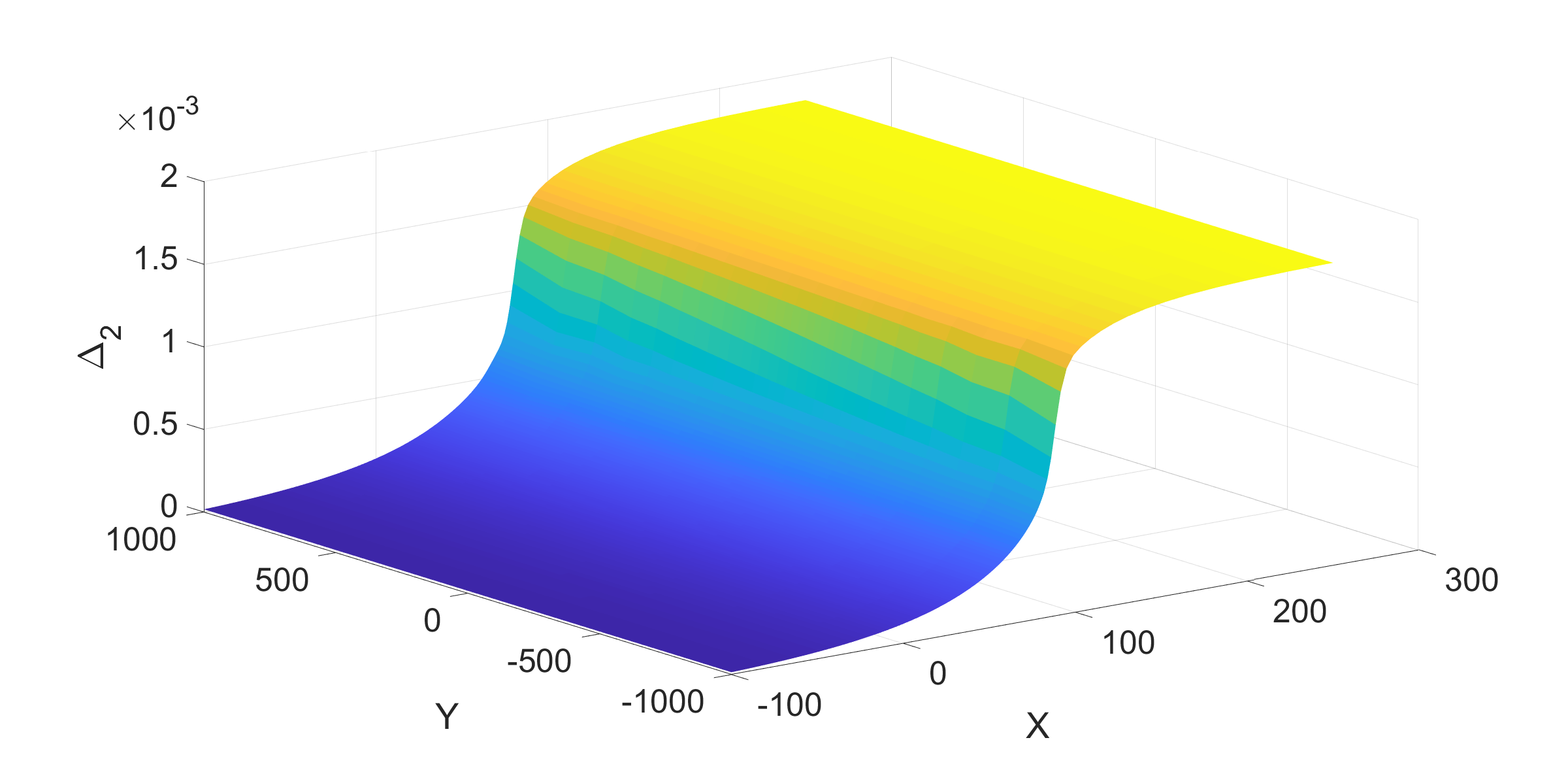} &
        \includegraphics[trim=55 0 0 0, clip,width=0.45\linewidth]{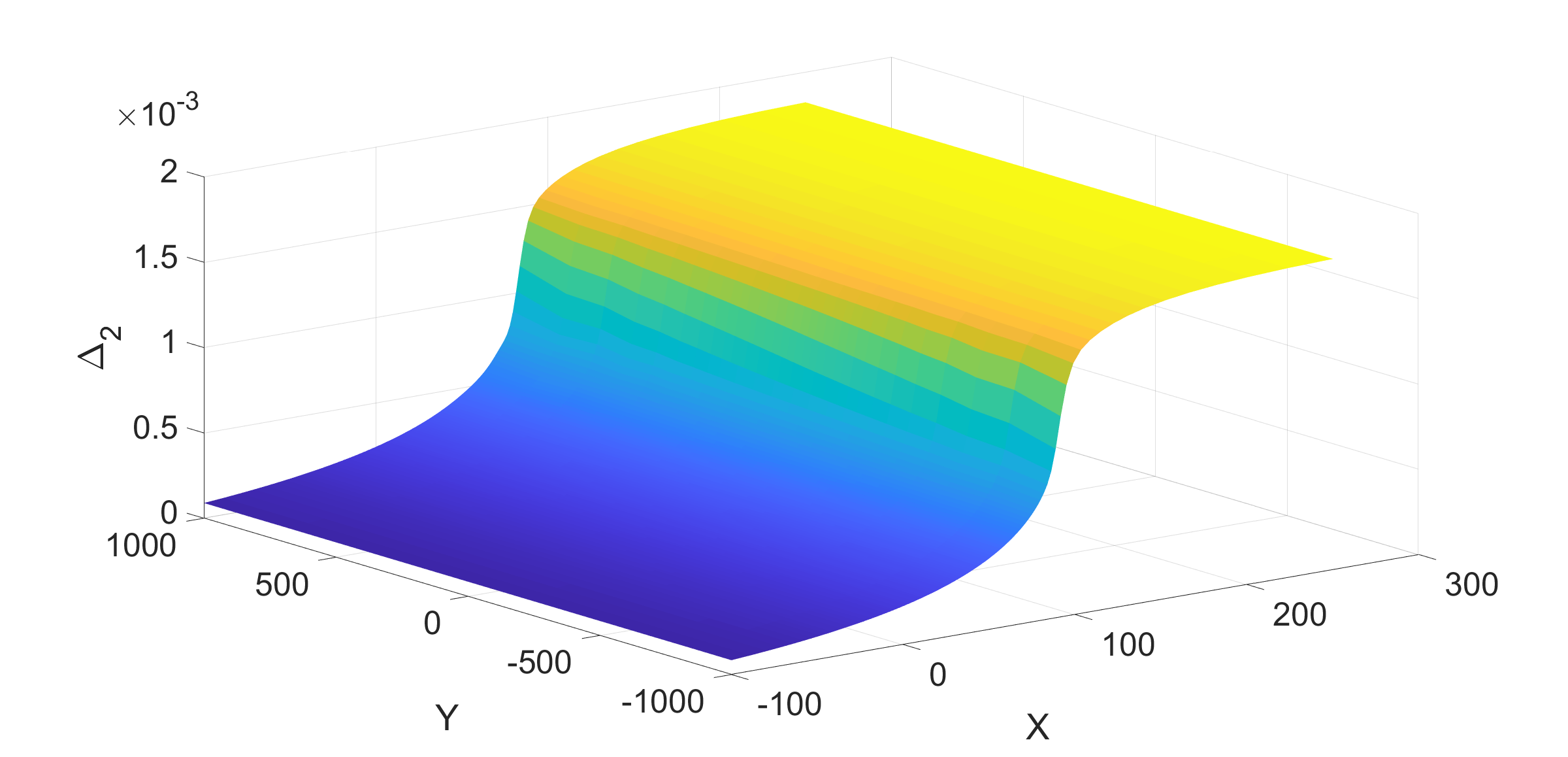}

    \end{tabular}

    \caption{Surfaces of the swing option value and its Delta Greeks if $z=0$ and $T=1$ for sets 1 and 4 where 
    $N_a=20$, $L=1$, $M=10$.
    The left column displays the option value, $\Delta_1$ and $\Delta_2$ for set 1.
    The right column displays the option value, $\Delta_1$ and $\Delta_2$ for set 4.
    }
    \label{Greeks}
\end{figure}

Figure \ref{Greeks} displays the surfaces of the approximated swing option value together with its Delta Greeks if $z=0$ and $t = T$ for parameter sets 1 and~4 with the swing parameters $N_a = 20$, $L = 1$, $M = 10$.
A perusal of $\Delta_1$ indicates that three common regions in the $x$-domain can be distinguished. In the deep out-of-the-money region, $\Delta_1$ is positive but relatively close to zero, i.e., small changes in the price component $x$ have minor impact on the option value. Next, there is a region where $\Delta_1$ exhibits a sharp increase as $x$ increases, and a small change in $x$ has a significant impact on the option value. Finally, in the deep in-the-money region, $\Delta_1$ reaches a plateau, meaning that the option value increases linearly with $x$.

A perusal of $\Delta_2$ shows the same three regions. Notice that the sensitivity with respect to the price component $y$ -- as captured by $\Delta_2$ -- is significantly lower than that with respect to $x$, with about a factor $10^{-3}$. This reduced influence is primarily attributed to the high mean reversion speed $\beta$, which rapidly undermines the impact of $y$ over time. Nevertheless, it is anticipated that the influence of the price component $y$ will become more pronounced as the number of swing action times increases, since the option holder may benefit more frequently from price jumps before their effects are diminished by the strong mean-reverting behaviour.

\subsection{Optimal exercise policy}\label{sec6.6}

In this subsection, we consider the optimal exercise policy for swing options and examine the influence of various model parameters,  given in Table \ref{paramSetsOptimal}, on this policy. We select the Kou-type jump-diffusion model in view of its flexibility in independently modelling upward and downward jumps. The optimal exercise policy is obtained by applying Algorithm \ref{optPath} after the approximation of the swing option values by the dynamic programming approach given in Algorithm \ref{DPP}.

Figure \ref{OptimalCumulative} presents a series of heat maps illustrating the cumulative amount of energy purchased up to forward time $t = 0.5$ under a swing option contract with parameters $K=50$, $T=1$, $r=0.03$, $N_a = 100$, $M = 50$. The choice $t = 0.5$ allows for a clear differentiation between three distinct cases in the optimal exercise policy. The first region (dark blue) represents no exercise, where no units have yet been bought up to time $t = 0.5$. The second region (dark red) corresponds to full exercise, where the maximum allowable number of units $M$ has already been bought. The third region (from light blue to light red) represents partial exercise, where the contract holder has already bought a certain amount of electricity but can still buy additional units in the future in anticipation of more favourable market conditions.

Each heat map in Figure~\ref{OptimalCumulative} corresponds to a specific parameter set listed in Table~\ref{paramSetsOptimal}.
Here set A is chosen as a reference set from which each of the other seven sets B, C,\ldots, H differs by just one parameter value, allowing a comparative analysis of how individual parameters affect the optimal exercise policy.
The $x$- and $y$-axes represent the two terms in the electricity price $x+y$ at initial time, i.e., today.

\begin{table}[!h]
        \caption{Parameter sets for the Kou-type jump case. The time is measured in years.}
        \centering
        \begin{tabular}{cccccccccc}
            parameters &$\mu$& $\alpha$ & $\beta$ & $\sigma$& $\lambda$ & $p$& $\eta_1$& $\eta_2$  & $L$ \\
           Set A  &80& 8 & 126 & 11 & 52 & 0.5 & 0.01 & 0.01& 1 \\
           Set B  &80& 8 & 126 & 11 & 52 & 0.9 & 0.01 & 0.01& 1 \\
           Set C  &80& 8 & 126 & 11 & 52 & 0.1 & 0.01 & 0.01& 1 \\
           Set D  &80& 8 & 126 & 11 & 10 & 0.5 & 0.01 & 0.01& 1 \\
           Set E  &80& 8 & 10 & 11 & 52 & 0.5 & 0.01 & 0.01& 1 \\
           Set F  &80& 1 & 126 & 11 & 52 & 0.5 & 0.01 & 0.01& 1\\
           Set G  &80& 8 & 126 & 11 & 52 & 0.5 & 0.01 & 0.01& 2\\
           Set H  &80& 8 & 126 & 11 & 52 & 0.5 & 0.01 & 0.01& 5\\
        \end{tabular}
        \label{paramSetsOptimal}
\end{table}
\begin{figure}[!h]
    \centering
    \parbox{0.45\linewidth}{\centering \textbf{Set A}\\
    \includegraphics[trim=60 0 0 0, clip,width=\linewidth]{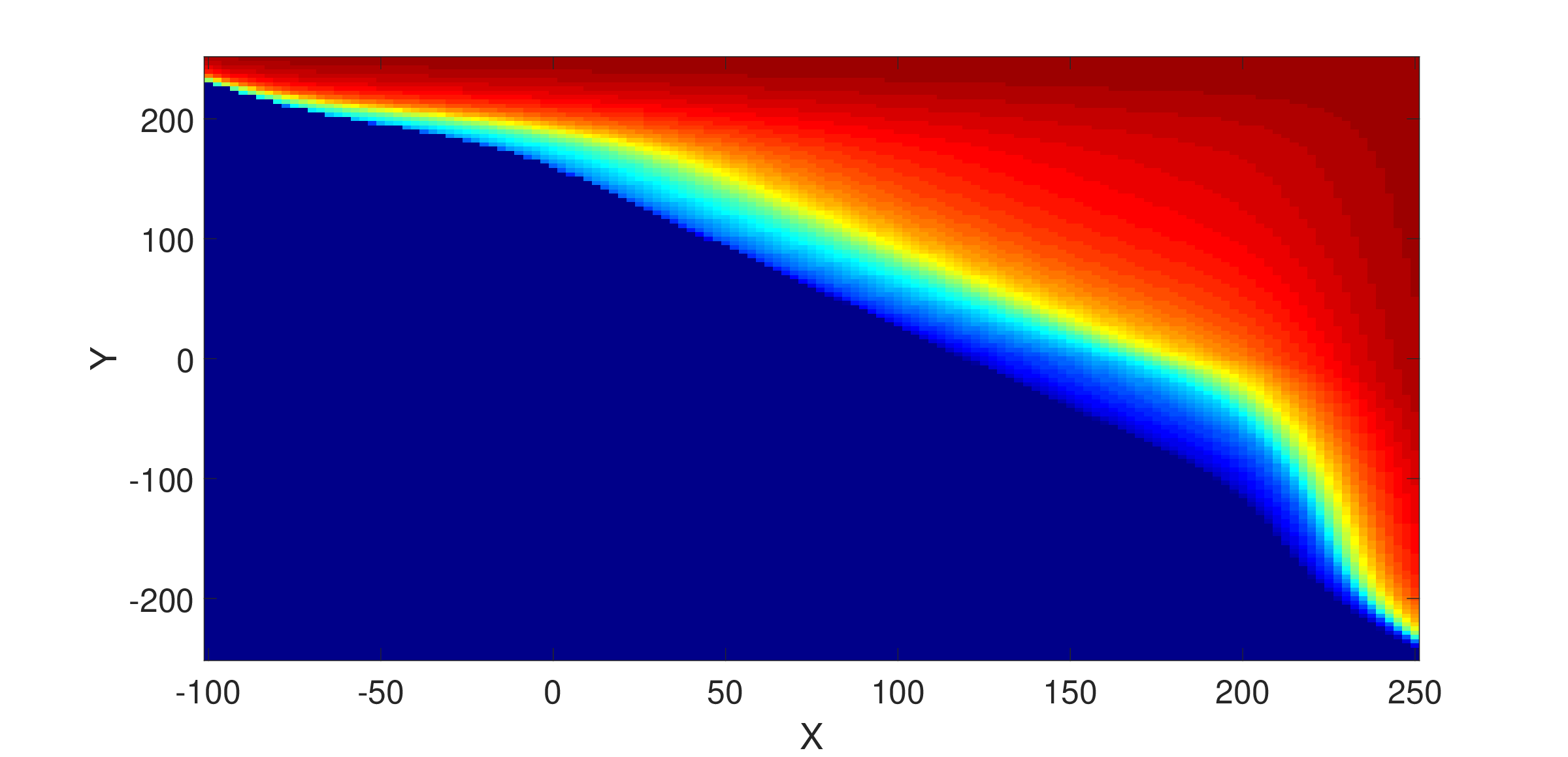}}
    \parbox{0.45\linewidth}{\centering \textbf{Set B}\\
    \includegraphics[trim=60 0 0 0, clip,width=\linewidth]{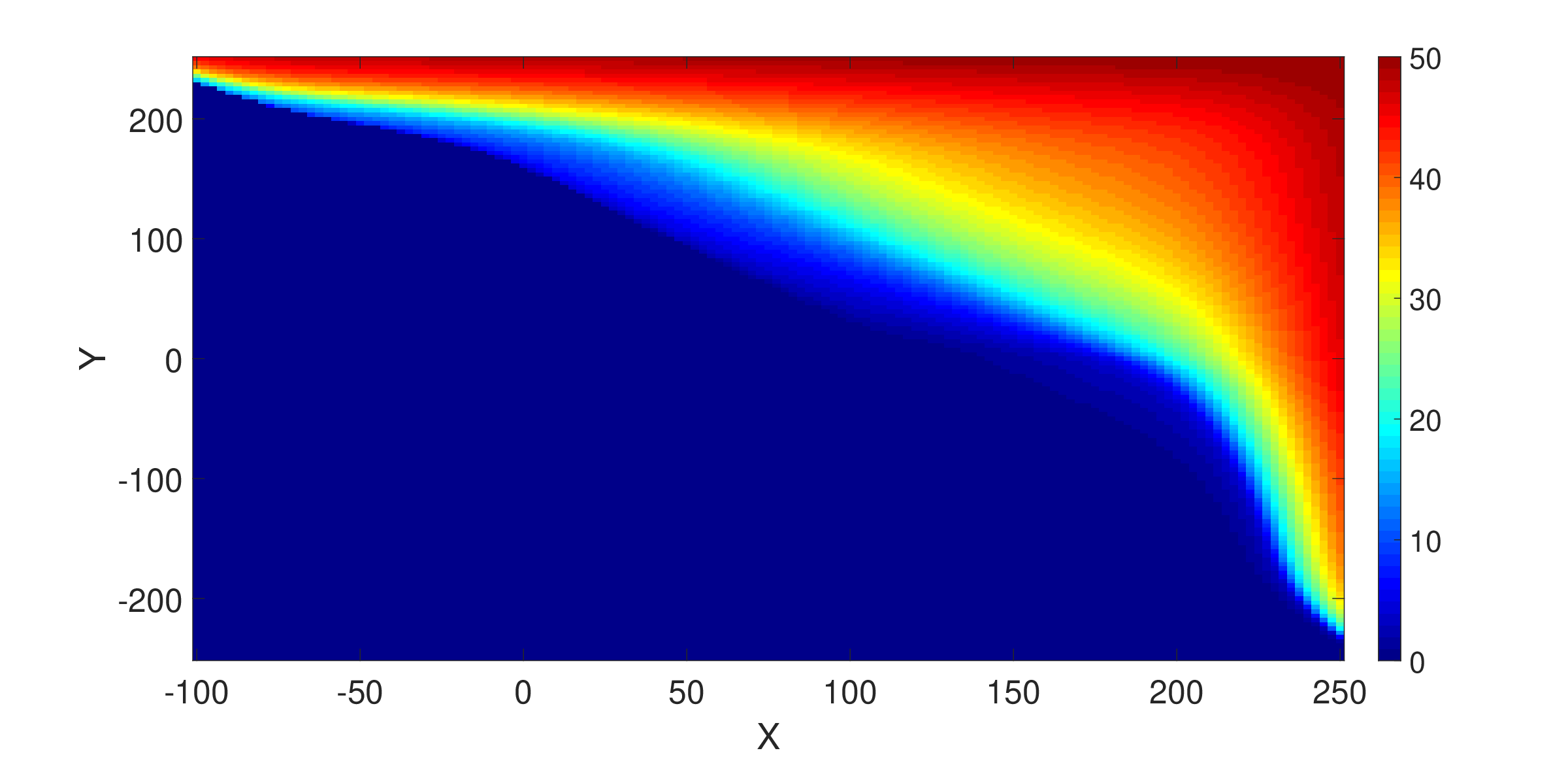}}

\vspace*{3mm}
    \parbox{0.45\linewidth}{\centering \textbf{Set C}\\
    \includegraphics[trim=60 0 0 0, clip,width=\linewidth]{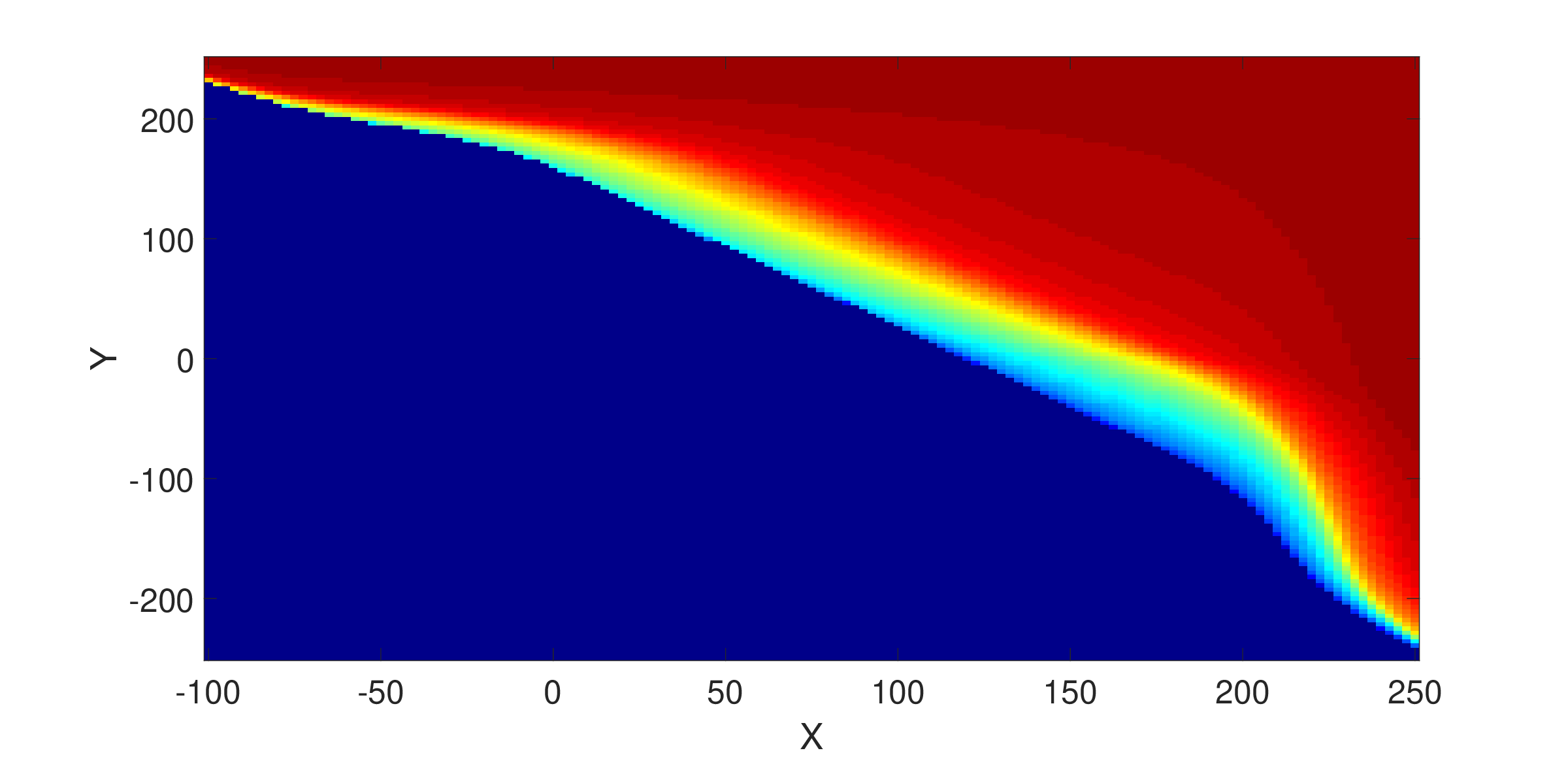}}
    \parbox{0.45\linewidth}{\centering \textbf{Set D}\\
    \includegraphics[trim=60 0 0 0, clip,width=\linewidth]{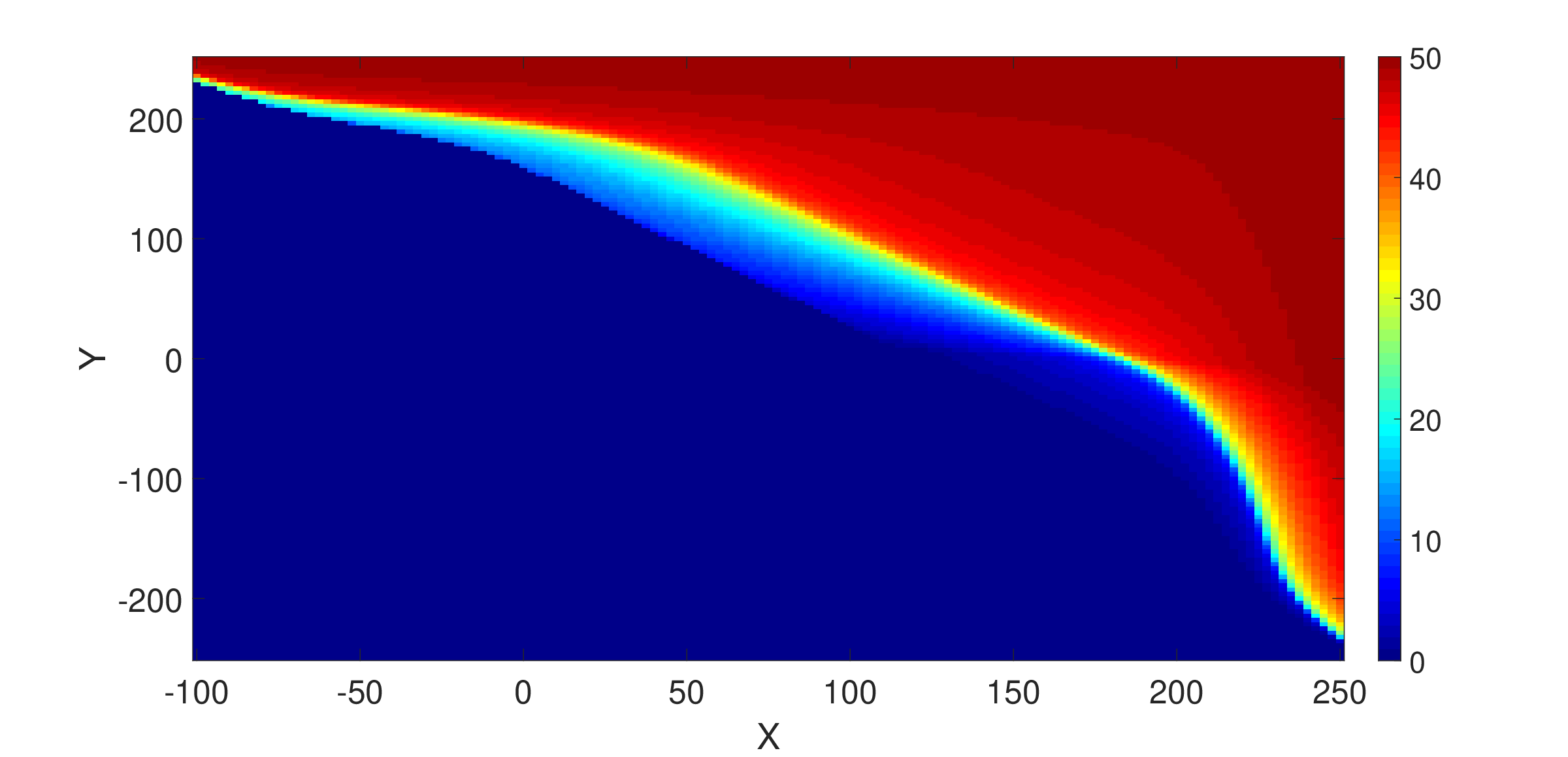}}

\vspace*{3mm}
    \parbox{0.45\linewidth}{\centering \textbf{Set E}\\
    \includegraphics[trim=60 0 0 0, clip,width=\linewidth]{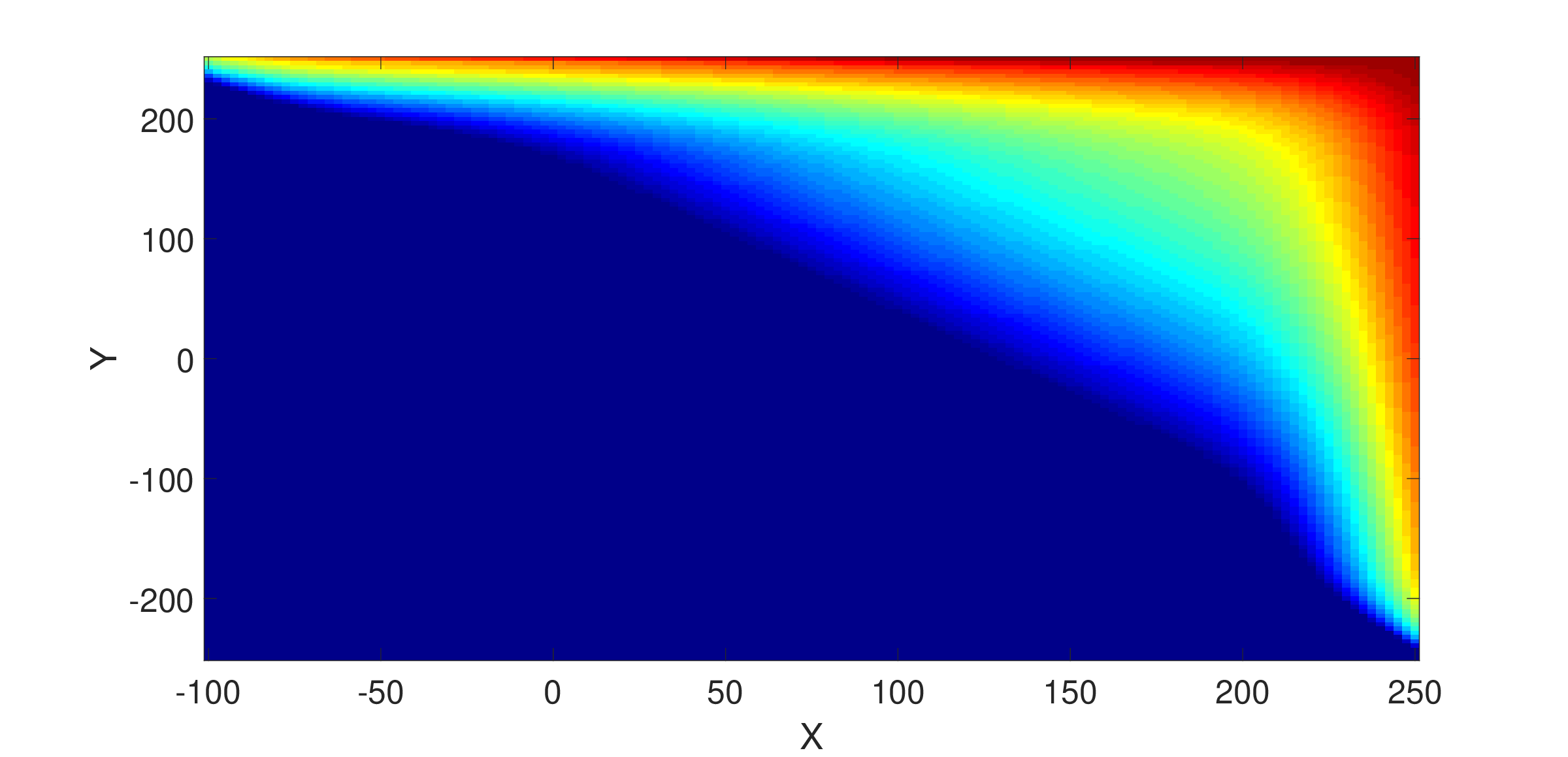}}
    \parbox{0.45\linewidth}{\centering \textbf{Set F}\\
    \includegraphics[trim=60 0 0 0, clip,width=\linewidth]{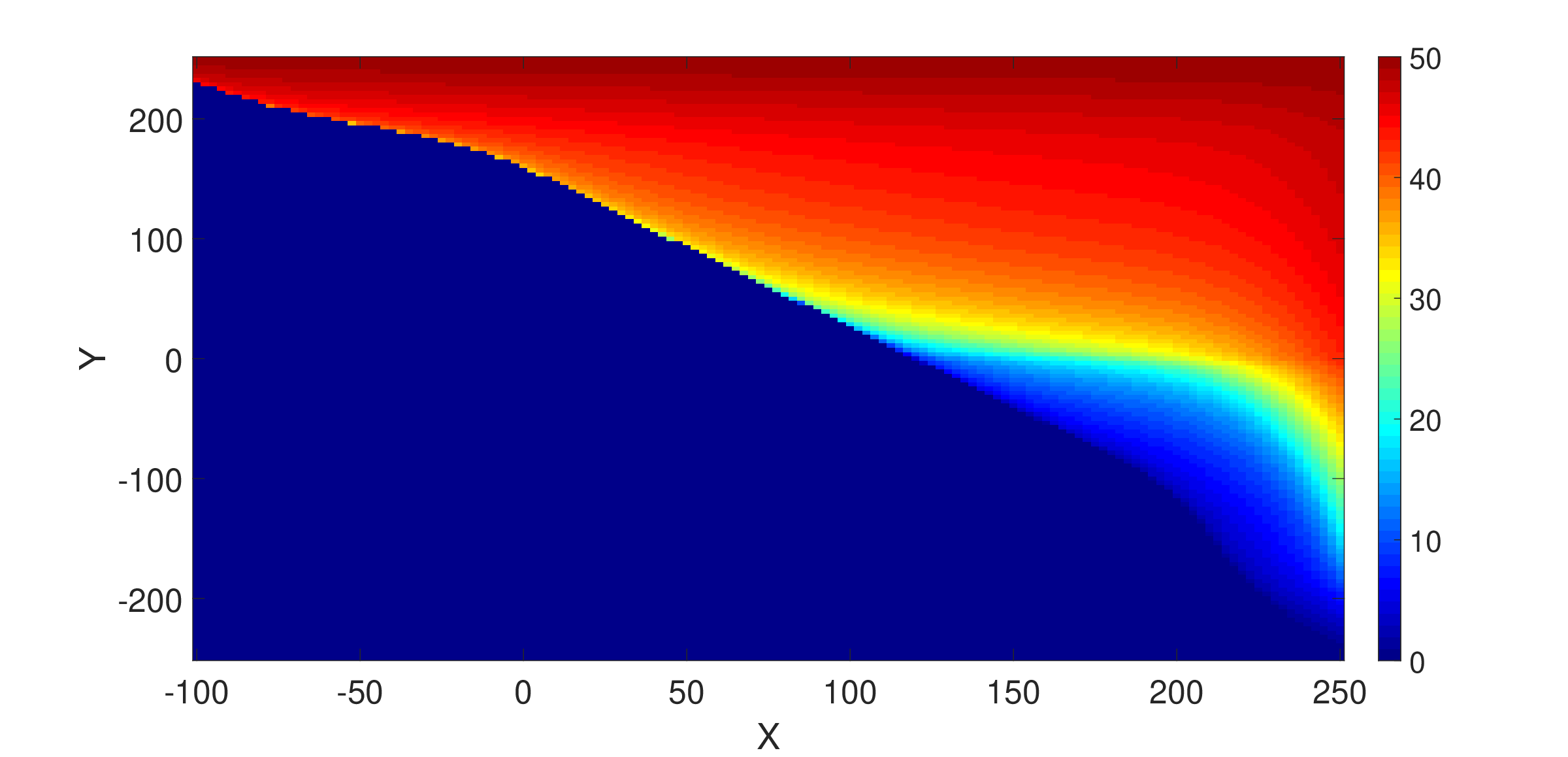}}

\vspace*{3mm}
    \parbox{0.45\linewidth}{\centering \textbf{Set G}\\
    \includegraphics[trim=60 0 0 0, clip,width=\linewidth]{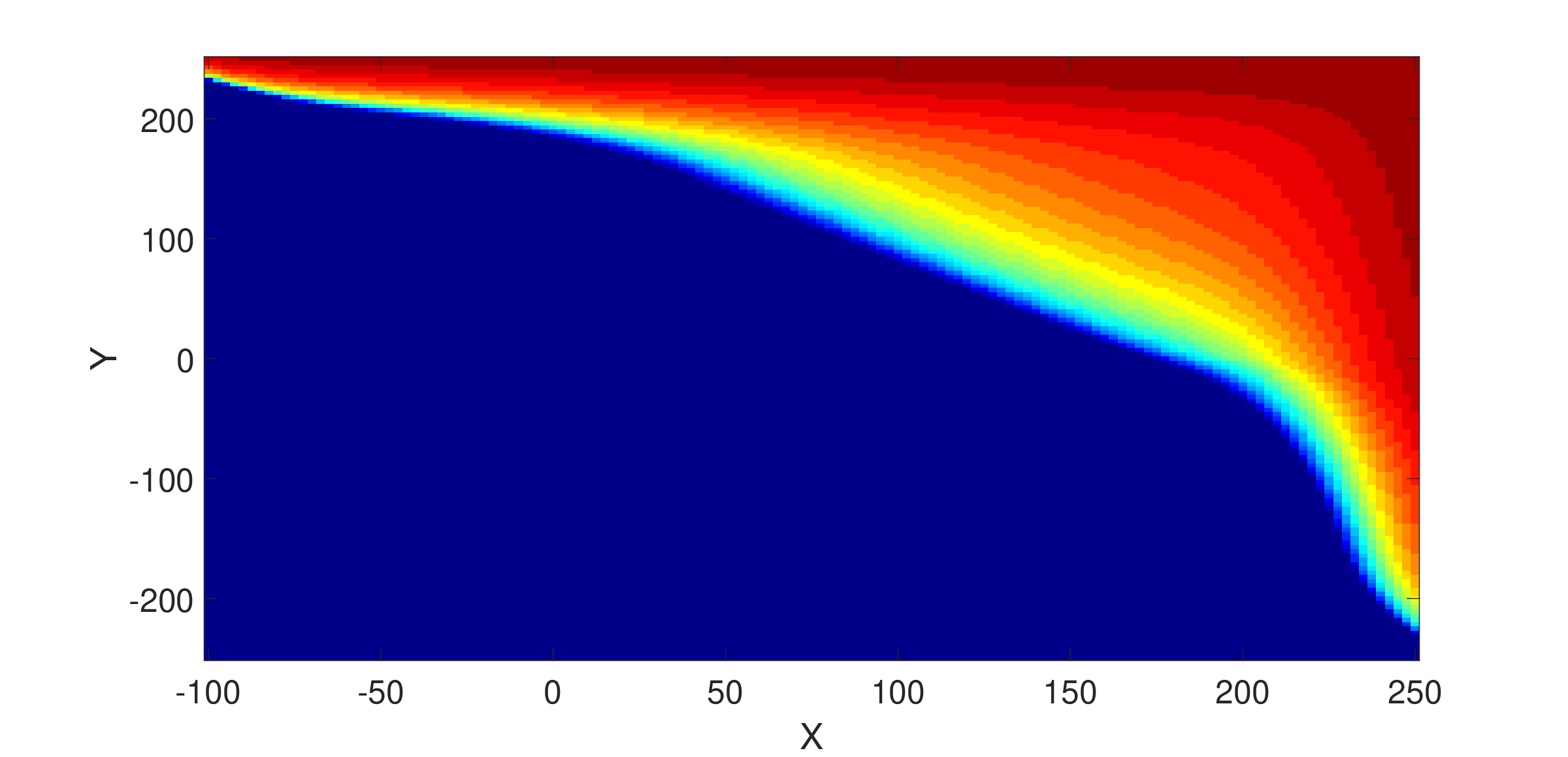}}
    \parbox{0.45\linewidth}{\centering \textbf{Set H}\\
    \includegraphics[trim=60 0 0 0, clip,width=\linewidth]{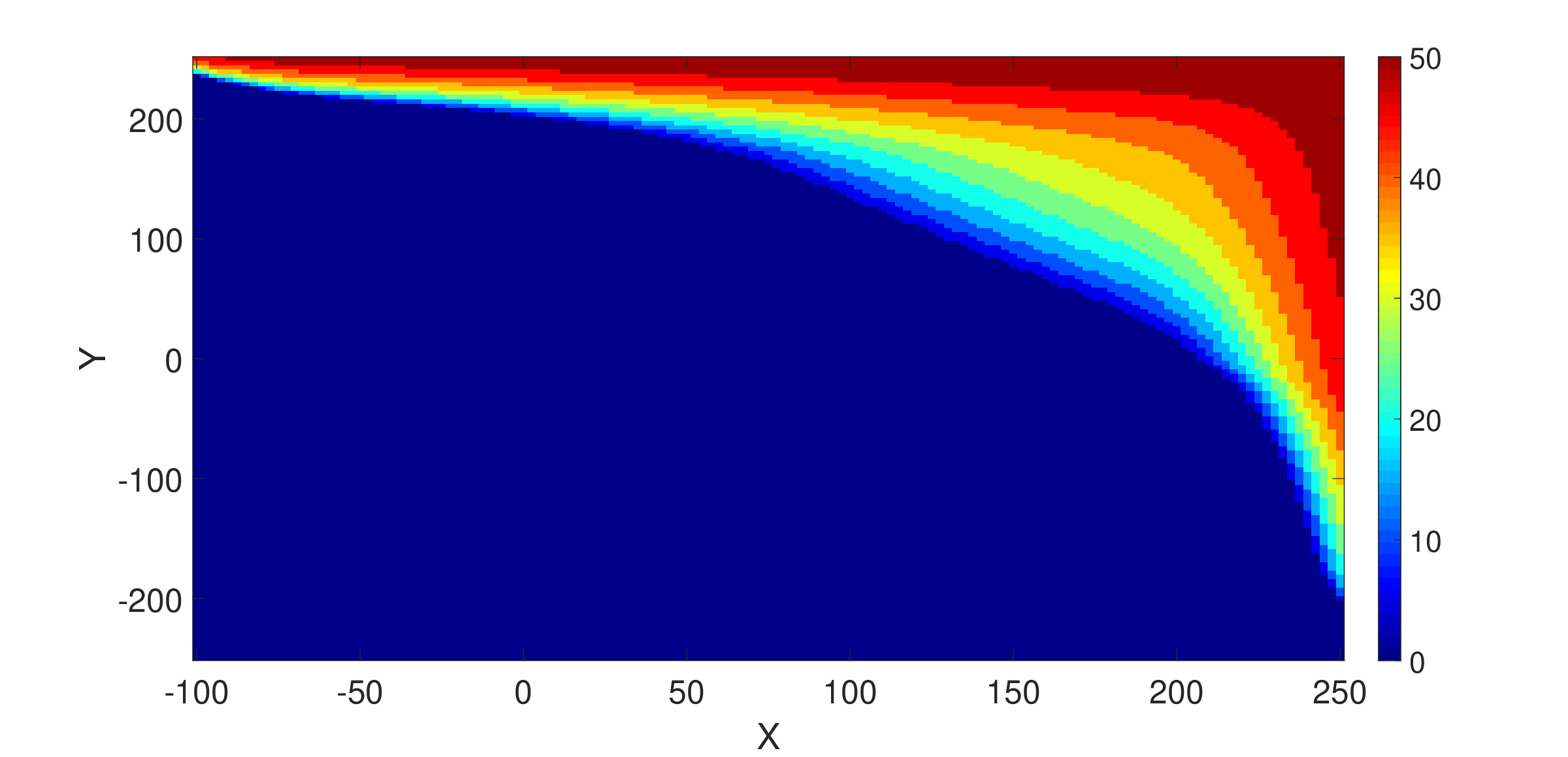}}
    \caption{Cumulative amounts of energy purchased up to forward time $t = 0.5$ under a swing option contract with $K=50$, $T=1$, $r=0.03$, $N_a = 100$, $M = 50$ and the parameter sets given by Table \ref{paramSetsOptimal}. 
}
    \label{OptimalCumulative}
\end{figure}

Compared to set A, upward jumps are more likely for set B, since the value $p$ has increased to $p=0.9$.
It is clear that the dark red region, corresponding to full exercise, is smaller compared to set A. This indicates that the holder of the swing option tends to delay exercising, anticipating higher future prices due to the increased probability of upward jumps. 

Conversely, for set C downward jumps are more probable, as the value $p$ has decreased to $p=0.1$. 
Clearly, the dark red region is now larger compared to that of set A. This suggests that the holder tends to exercise more often, with less emphasis on waiting, as the opportunity to benefit from favourable future price movements is reduced.

Set D has less frequent jumps, since the value $\lambda$ has decreased to $\lambda =10$. 
Again, the dark red region is larger compared to that of set A, indicating that the holder tends to exercise more often as the probability of favourable future price movements is reduced.

Set E has a lower mean reversion speed for the $Y$ process, namely $\beta = 10$. 
Here the dark red region is very small and the region in between dark red and dark blue, corresponding to partial exercise, is large. This suggests that the optimal exercise policy is to wait longer. This could be explained by two factors: (i) downward jumps have a more prolonged impact, discouraging immediate exercise due to sustained low prices, and (ii) favourable price conditions, once reached, tend to persist longer, making it more attractive to wait in anticipation of even better opportunities.

Set F has a lower mean reversion speed for the $X$ process, namely $\alpha = 1$. 
As a distinctive feature, if the $x$-component of the price $x+y$ is high while $y<0$, then the optimal exercise policy tends to delay exercise, see the dark blue region at the bottom right in the subplot. Indeed, it may be argued that, when a downward jump occurs, the slow mean reversion ensures that the high price persists, thus reducing the need to exercise and allowing more time to wait for a more favourable, upward jump.
Next, at the upper left in the subplot a swift transition is observed from dark blue (no exercise) to dark red (full exercise). 
We conjecture that this transition is still smooth.

Sets G and H have a larger number of units of energy that can be bought at each swing action time, with $L=2$ and $L=5$ respectively. 
Here the dark blue region expands. This is conceivable because the option holder can afford to wait longer when the price is low to moderate, in anticipation of more favourable future opportunities that allow buying a larger number of units. 

\section{Conclusions}
In this paper, we investigated the numerical valuation of swing options with discrete action times under an affine two-factor mean-reverting model with jumps. 
For the numerical solution of the pertinent sequence of two-dimensional PIDEs, we studied three methods: SLCNFI, CNFI and DIRKFI.
The first method is based on the semi-Lagrangian approach, whereas the second and third methods are based on a full semidiscretisation approach by suitable finite differences. 
For the time-stepping, the Crank--Nicolson scheme and an $L$-stable DIRK scheme are considered.
Here, the nonlocal integral has been effectively treated by a fixed-point iteration.

Ample numerical experiments demonstrate that all three methods possess a desirable second-order convergence behaviour notwithstanding the convection-dominated property of the PIDE and the nonsmooth initial function and the presence of the nonlocal integral term. Our theoretical analysis confirms that the CNFI and DIRKFI methods are stable and second-order convergent under a smoothness condition. A theoretical convergence analysis of the semi-Lagrangian method is not undertaken in this work and forms an interesting topic for future research. 
The finite difference framework adopted for swing option valuation enables also the efficient approximation of the Delta Greeks, which provides important insight into the impact of the price components $x$ and $y$ on the option values. Moreover, a study of the effect of the financial model parameters on the optimal exercise policy sheds much light on when and how much to exercise.

Among others, a main topic for our future research is the extension of the present numerical solution approach towards more advanced swing option types, such as considered in \cite{Dahlgren}.

\subsection*{Acknowledgements}
The authors acknowledge the support of the Research Foundation - Flanders (FWO)  under grant G0B5623N and the FWO Scientific
Research Network ModSimFIE (FWO WOG W001021N). The third author also acknowledges the financial support of the Research Foundation - Flanders (FWO) through FWO SAB K803124.

    \bibliographystyle{plainnat}
    \bibliography{bibliography}
    \newpage
    \appendix
    \section{Proof of Lemma \ref{bound}}\label{AppA}
     \begin{proof}
        Without loss of generality, we assume that $\tilde{b}>0$. 
        \newline
        For the ease of presentation, write $D_{x,i,i}= :g_i$ and assume there is index $i_{0}$ such that $g_{i_0}=0$. Let $P_0$ denote the part of the inner product $\langle D_x\Tilde{A}_{x}V,V\rangle$ corresponding to the case $i<i_0$.
        We have by \eqref{eqFVij} 
        \[P_0 = \sum_{i=0}^{i_0-1}g_{i}\sum_{k=-1}^2w_kV_{i+k}V_i - w_{-1}g_0V_{-1}V_0 .\]
        The term $w_{-1}g_0V_{-1}V_0$ is subtracted because the inner product does not contain $V_{-1}$ due to the Dirichlet boundary conditions. For the first term of $P_0$, we obtain using \eqref{W3} and the relation $2ab=a^2+b^2-(a-b)^2$:
        \begin{align*}
           & \sum_{i=0}^{i_0-1}g_{i}\sum_{k=-1}^2w_kV_{i+k}V_i\\
           &=\underbrace{\frac 12 \sum_{i=0}^{i_0-1}g_{i}\sum_{k=-1}^2w_kV^2_{i+k}}_{=:A} + \frac 12 \sum_{i=0}^{i_0-1}g_{i}\underbrace{\sum_{k=-1}^2w_k}_{=0}V^2_{i}\underbrace{-
           \frac 12 \sum_{i=0}^{i_0-1}g_{i}\sum_{k=-1}^2w_k(V_{i+k}-V_i)^2}_{=:B}.
        \end{align*}
        Interchanging the summations, invoking again \eqref{W3} and noting that $g_{i-k}-g_i=k\tilde{b}\Delta x$ and $g_{i_0-1+i-k}=g_{i_0-1+i-k}-g_{i_0}=(1-i+k)\tilde{b}\Delta x$ we derive
        \begin{align*}
            A&= \frac 12 \sum_{k=-1}^2 w_k\sum_{i=k}^{i_0-1+k}g_{i-k}V^2_i\\
            & = \frac 12 \sum_{k=-1}^2 w_k\sum_{i=2}^{i_0-2}(g_{i-k}-g_i)V^2_i +\frac 12 \sum_{k=-1}^1 w_k\sum_{i=k}^{1}g_{i-k}V^2_i+\frac 12 \sum_{k=0}^2 w_k\sum_{i=i_0-1}^{i_0-1+k}g_{i-k}V^2_i\\
            &= \frac 12 \tilde{b}\Delta x \sum_{k=-1}^2 k w_k\sum_{i=2}^{i_0-2}V^2_i +\frac 12 \sum_{i=-1}^1 \sum_{k=-1}^{i}w_k g_{i-k}V^2_i +\frac 12 \tilde{b}\Delta x \sum_{i=0}^2\sum_{k=i}^2(1-i+k)w_kV_{i_0-1+i}^2\\
            & = -(w_{-1}+w_2)\tilde{b}\Delta x \sum_{i=2}^{i_0-2}V^2_i +\frac 12 \sum_{i=-1}^1 \sum_{k=-1}^{i}w_k g_{i-k}V^2_i +\frac 12 \tilde{b}\Delta x \sum_{i=0}^2\sum_{k=i}^2(1-i+k)w_kV_{i_0-1+i}^2.
        \end{align*}
        As  $g_{i}\geq 0$ for $i< i_0$ it holds that 
        $\frac{1}{2}g_{i}
        (V_{i+2}-V_{i})^{2}\leq g_{i}(V_{i+1}-V_{i})^{2}+g_{i}(V_{i+2}-V_{i+1})^{2} $. Then, recalling the properties \eqref{W3} and $g_{i-k}-g_i=k\tilde{b}\Delta x$, and defining $w_{-1}^+ = \max(w_{-1},0)$ the term $B$ can be bounded as follows
         \begin{align*}
           B  &= -\frac{1}{2}w_{-1}\sum_{i=0}^{i_0-1}g_i\Big( (V_{i-1}-V_i)^2-(V_{i+1}-V_i)^2\Big) - \frac{1}{2}\underbrace{(w_{-1}+w_1)}_{=-4w_2}\sum_{i=0}^{i_0-1}g_i(V_{i+1}-V_i)^2\\
           & \qquad - \frac{1}{2}w_2\sum_{i=0}^{i_0-1}g_i(V_{i+2}-V_i)^2\displaybreak[0]\\
           & \leq -\frac{1}{2}w_{-1}\sum_{i=0}^{i_0-1}g_i\Big( (V_{i-1}-V_i)^2-(V_{i+1}-V_i)^2\Big) +2w_2 \sum_{i=0}^{i_0-1}g_i(V_{i+1}-V_i)^2\\
          & \qquad -w_2 \sum_{i=0}^{i_0-1}g_i \Big( (V_{i+2}-V_{i+1})^2+(V_{i+1}-V_i)^2\Big) \displaybreak[0]\\
            &= \frac{1}{2}w_{-1}\tilde{b}\Delta x\sum_{i=0}^{i_0-1}(V_{i+1}-V_i)^2-w_2\tilde{b}\Delta x\sum_{i=0}^{i_0-1}(V_{i+2}-V_{i+1})^2-\frac{1}{2}w_{-1}g_0(V_0-V_{-1})^2+w_2g_0(V_1-V_0)^2\displaybreak[0]\\
            & \leq w_{-1}^+\tilde{b}\Delta x \sum_{i=0}^{i_0-1}(V^2_{i+1}+V_i^2)-2w_2\tilde{b}\Delta x\sum_{i=0}^{i_0-1}(V_{i+2}^2+V_{i+1}^2)
            -\frac{1}{2}w_{-1}g_0(V_0-V_{-1})^2+w_2g_0(V_1-V_0)^2\\
          &  = (2w_{-1}^+-4w_2)\tilde{b}\Delta x \sum_{i=2}^{i_0-2} V_i^2+w_{-1}^+\tilde{b}\Delta x (V_0^2+2V^2_1+2V_{i_0-1}^2+V_{i_0}^2)\\
           &\qquad-2w_2\tilde{b}\Delta x (V_1^2+2V_{i_0-1}^2+2V_{i_0}^2+V_{i_0+1}^2) 
           -\frac{1}{2}w_{-1}g_0(V_0-V_{-1})^2+w_2g_0(V_1-V_0)^2.
        \end{align*}
        Combination of the expression for $A$ and the bound for $B$ while noting that $2w_{-1}^+-w_{-1}=|w_{-1}|$ gives
        \begin{align*}
            A+B &  \leq (|w_{-1}|-5w_2) \tilde{b}\Delta x \sum_{i=2}^{i_0-2} V_i^2  +\frac 12 \tilde{b}\Delta x \sum_{i=0}^2\sum_{k=i}^2(1-i+k)w_kV_{i_0-1+i}^2\\
            &\qquad +w_{-1}^+\tilde{b}\Delta x (V_0^2+2V^2_1+2V_{i_0-1}^2+V_{i_0}^2)-2w_2\tilde{b}\Delta x(V_1^2+2V_{i_0-1}^2+2V_{i_0}^2+V_{i_0+1}^2)\\
             &\qquad +\underbrace{\frac 12 \sum_{i=-1}^1 \sum_{k=-1}^{i}w_k g_{i-k}V^2_i -\frac{1}{2}w_{-1}g_0(V_0-V_{-1})^2+w_2g_0(V_1-V_0)^2}_{=:C} .
        \end{align*}
        Estimating further the term $C$ using $w_0=3w_2$ and the inequality $\frac 12 a^2\leq (a-b)^2+b^2$ leads to
        \begin{align*}
            C&=\frac{1}{2}w_{-1}g_0V_{-1}^2 + \frac{1}{2}(w_0g_0+w_{-1}g_1)V_0^2 - \frac{1}{2}w_2g_0V_1^2 -\frac{1}{2}(w_0+2w_{-1})\tilde{b}\Delta xV_1^2\\
            &\qquad -\frac{1}{2}w_{-1}g_0(V_0-V_{-1})^2+w_2g_0(V_1-V_0)^2\\
            &\leq \frac{1}{2}w_{-1}g_0\Big(V_{-1}^2 -(V_0-V_{-1})^2\Big)+\frac{1}{2}(w_2g_0+w_{-1}g_1)V_0^2+ w_2g_0V_0^2 - w_2g_0(V_1-V_0)^2 -w_2g_0V_0^2 \\
            &\qquad -\frac{1}{2}(w_0+2w_{-1})\tilde{b}\Delta xV_1^2 + w_2g_0(V_1-V_0)^2\\
            &= \frac{1}{2}w_{-1}g_0\Big(V_{-1}^2 -(V_0-V_{-1})^2\Big)+\frac{1}{2}(w_2g_0+w_{-1}g_1)V_0^2-\frac{1}{2}(3w_2+2w_{-1})\tilde{b}\Delta xV_1^2.
        \end{align*}
        Moreover, 
        \begin{align*}
            C-w_{-1}g_0V_{-1}V_0&= C + \frac{1}{2}w_{-1}g_0(V_{-1}-V_0)^2-\frac{1}{2}w_{-1}g_0V_{-1}^2-\frac{1}{2}w_{-1}g_0V_{0}^2\\
            &\leq \frac{1}{2}w_2g_0V_0^2-\frac{1}{2}(3w_2+2w_{-1})\tilde{b}\Delta xV_1^2- \frac{1}{2}w_{-1}\tilde{b}\Delta xV_0^2\\
            &\leq -\frac{1}{2}(3w_2+2w_{-1})\tilde{b}\Delta xV_1^2 - \frac{1}{2}w_{-1}\tilde{b}\Delta xV_0^2.
        \end{align*}
        Substituting these bounds into $P_0$ and using the properties \eqref{W3} we find
        \begin{align*}
            P_0&\leq (|w_{-1}|-5w_2)\tilde{b}\Delta x\sum_{i=2}^{i_0-2}V_i^2 +\frac{1}{2}|w_{-1}|\tilde{b}\Delta x V_0^2 + (|w_{-1}|-\frac{7}{2}w_2)\tilde{b}\Delta x V_1^2 \\
            &\quad + (|w_{-1}|-5w_2)\tilde{b}\Delta x V_{i_0-1}^2 + (\half|w_{-1}|-5w_2)\tilde{b}\Delta x V_{i_0}^2 -\frac32w_2\tilde{b}\Delta x V_{i_0+1}^2\\
            &\leq  (|w_{-1}|-5w_2)\tilde{b}\Delta x\sum_{i=0}^{i_0-2}V_i^2 +(|w_{-1}|-5w_2)\tilde{b}\Delta x V_{i_0-1}^2 + (\half|w_{-1}|-5w_2)\tilde{b}\Delta x V_{i_0}^2 -\frac32w_2\tilde{b}\Delta x V_{i_0+1}^2 .
        \end{align*}
        Let $\tilde{g}_i=-g_{m-i}=-g_m-i\tilde{b}\Delta x$ (notice that $\tilde{g}_{m-i_0}=0$) and $\tilde{V}_i=V_{m-i}$.
    By a change of indices, we get
        \begin{align*}
            &\sum_{i=i_0+1}^{m}g_{i}(-w_{-1}V_{i+1} - w_0V_{i} - w_1V_{i-1}-w_2V_{i-2})V_{i}+g_mw_{-1}V_{m}V_{m+1}\\
            &= \sum_{i=0}^{m-i_0-1}\tilde{g_{i}}(w_{-1}\tilde{V}_{i-1} + w_0\tilde{V}_{i} + w_1\tilde{V}_{i+1}+w_2\tilde{V}_{i+2})\tilde{V}_{i}-\tilde{g_0}w_{-1}\tilde{V}_{-1}\tilde{V}_{0}\\
            &\leq (|w_{-1}|-5w_2)\tilde{b}\Delta x\sum_{i=i_0+2}^{m}V_i^2 -\frac{3}{2}w_2\tilde{b}\Delta xV_{i_0-1}^2+(\frac{1}{2}|w_{-1}|-5w_2)\tilde{b}\Delta xV_{i_0}^2 + (|w_{-1}|-5w_2)\tilde{b}\Delta xV_{i_0+1}^2.
        \end{align*}
        Adding up the sums for $i<i_0$ and $i>i_0$, it follows that $\langle D_x\Tilde{A}_{x}V,V\rangle \leq (|w_{-1}|-10w_2) \tilde{b}\, \langle V,V\rangle$ for all vectors $V$, which yields the bound on $\mu_2[D_x\tilde{A}_x]$.
    \end{proof}
  
\section{Proof of Lemma \ref{egf}}\label{AppB}
\begin{proof}
Let $\theta\in [\frac 14,\frac 12]$.
We focus on the case where $0\leq x< \frac{1}{\theta}$, since the result for the case $x\leq 0$ is known, see \cite{CashDIRK}.\\
For any given $x\in[0,\frac{1}{\theta}[$ there holds
\[|R_{\theta}(x + iy)|^2 = \frac{f_{\theta}(y)}{g_{\theta}(y)} \quad (y\in\R),
\]
where 
\[
f_{\theta}(y) = \big(1 + (1-2\theta)x + (\half -2\theta+\theta^2)(x^2-y^2)\big)^2 + \big((1-2\theta)y + 2(\half -2\theta + \theta^2)xy\big)^2
\]
and
\[g_{\theta}(y) = \big((1-\theta x)^2+(\theta y)^2\big)^2.\]
Note that $g_{\theta}(y)>0$ since $0\leq \theta x<1$. We will prove that \[\frac{f_{\theta}(y)}{g_{\theta}(y)}\leq \frac{f_{\theta}(0)}{g_{\theta}(0)}\] or, equivalently,
\[
H_{\theta}(y) := f_{\theta}(y)g_{\theta}(0)-g_{\theta}(y)f_{\theta}(0)\leq 0.
\]
Obviously, $H_{\theta}(0) = 0$. After some straightforward calculations, it follows that the first derivative of $H_{\theta}(y)$ can be written as
\[
 \frac{d H_{\theta}}{dy}(y) = 2yK_{\theta}(y),
\]
where
\begin{align*}
    K_{\theta}(y) &= \big\{-2\big(1 + (1-2\theta)x+(\half-2\theta+\theta^2)(x^2-y^2)\big)(\half -2\theta+\theta^2)\\ 
    & \quad + \big(1-2\theta+2(\half -2\theta+\theta^2)x\big)^2\big\} (1-\theta x)^4\\
    & \quad -2\big(1+(1-2\theta)x+(\half - 2\theta + \theta^2)x^2\big)^2\big((1-\theta x)^2+(\theta y)^2\big)\theta^2.
\end{align*}
If it can be shown that $K_{\theta}(y)\leq 0$ $(y\in\R)$, then one obtains the desired result.
It is clear that $K_{\theta}(y)$ can be expressed as
\[
K_{\theta}(y) = 2K_{\theta,1}(x)y^2 + K_{\theta,2}(x)
\]
with certain terms $K_{\theta,1}(x)$ and $K_{\theta,2}(x)$.
We will prove in the following that $K_{\theta,1}(x)\leq 0$ and $K_{\theta,2}(x)\leq 0$.

For $K_{\theta,1}(x)$ there holds
\begin{align*}
    K_{\theta,1}(x)&= P_{\theta,1}(x)^2-P_{\theta,2}(x)^2
    = (P_{\theta,1}(x)+P_{\theta,2}(x))(P_{\theta,1}(x)-P_{\theta,2}(x)),
\end{align*}
where \[
    P_{\theta,1}(x) = (\half-2\theta+\theta^2)(1-\theta x)^2
\]
and
\[
    P_{\theta,2}(x) = \theta^2\big(1+(1-2\theta)x+(\half-2\theta+\theta^2)x^2\big).
\]
We have
\begin{align*}
    P_{\theta,1}(x)+P_{\theta,2}(x) = 2(\theta-\half)^2 + \theta x L_{\theta}(x)
\end{align*}
with
\[L_{\theta}(x) = 4(\theta-\frac{1}{4})(1-\theta) +2(\theta^2-2\theta+\half)\theta x.\]
It is easily seen that $L_{\theta}(0)\geq0$ and $L_{\theta}(\frac{1}{\theta})\geq 0$. Since $L_{\theta}$ is linear in $x$, it follows that $L_{\theta}(x)\geq 0$ and, hence, that \[P_{\theta,1}(x)+P_{\theta,2}(x)\geq0.\]
Next, one readily verifies that
\[
P_{\theta,1}(x)-P_{\theta,2}(x) = -2(\theta -\frac 14)-\theta(1-3\theta)x \leq 0.
\]
Consequently, $ K_{\theta,1}(x)\leq 0$.

For $K_{\theta,2}(x)$ we obtain after some tedious but straightforward calculations that \[K_{\theta,2}(x)= \frac{1} {2}x(1-\theta x)^2Q_{\theta}(x)\] with
    \[
    Q_{\theta}(x) =x^{2} \left(12 \theta^{4} - 28 \theta^{3} + 14 \theta^{2} - 2 \theta\right) + x \left(- 24 \theta^{3} + 36 \theta^{2} - 12 \theta + 1\right) +  12 \theta^{2} - 12 \theta + 2.
    \]
   If $\theta\not\in  \{ 1 - \frac{\sqrt{2}}{2},\frac 13\}$, then $Q_{\theta}$ is a second-degree polynomial in $x$ and its two roots are given by \[\frac{24 \theta^{3} - 36 \theta^{2} + 12 \theta - 1 \pm (2 \theta - 1) \sqrt{48(\theta + \frac{\sqrt{3}}{6})(\theta - \frac{1}{4})(\theta - \frac{\sqrt{3}}{6})} }{24 \theta(\theta - (1 - \frac{\sqrt{2}}{2}))(\theta-\frac 13)(\theta - (1+\frac{\sqrt{2}}{2}))}.
\]
We distinguish the following five cases, including that of the first-degree polynomial when $\theta\in \{1 - \frac{\sqrt{2}}{2},\frac 13\}$:
\begin{itemize}
    \item  If $\theta \in\{ \frac{1}{4},\frac{\sqrt{3}}{6}\}$, then $Q_{\theta}$ has a double root that is real and negative and the graph of $Q_{\theta}$ is a downward parabola. Thus, $Q_{\theta}(x)\leq 0$.
    \item If $\theta\in \, ]\frac 14,\frac{\sqrt{3}}{6}[$, then the roots are not real numbers and the graph of $Q_{\theta}$ is a downward parabola. Thus, $Q_{\theta}(x)\leq 0$.
    \item If $\theta\in \, ]\frac{\sqrt{3}}{6},1-\frac{\sqrt{2}}{2}[$, then both roots are real and negative and the graph of $Q_{\theta}$ is a downward parabola. Thus, $Q_{\theta}(x)\leq 0$.
    \item If $\theta \in [1-\frac{\sqrt{2}}{2},\frac{1}{3}] $, then $12 \theta^{4} - 28 \theta^{3} + 14 \theta^{2} - 2 \theta \geq 0$ and, together with $x<\frac{1}{\theta}$, this yields
    \begin{align*}
        Q_{\theta}(x)& \leq x(- 12 \theta^{3} + 8 \theta^{2} + 2 \theta - 1) + 12 \theta^{2} - 12 \theta + 2. 
    \end{align*}
    Since $12 \theta^{2} - 12 \theta + 2 \leq 0$, there holds $Q_{\theta}(x)\leq 0$ whenever $- 12 \theta^{3} + 8 \theta^{2} + 2 \theta - 1\leq 0$. Otherwise, 
    \begin{align*}
        Q_{\theta}(x)&\leq \frac{1}{\theta}(- 12 \theta^{3} + 8 \theta^{2} + 2 \theta - 1) + 12 \theta^{2} - 12 \theta + 2 
         = -4\theta +4 -\frac{1}{\theta}
        =-\frac{1}{\theta}(2\theta-1)^2\leq 0.
    \end{align*}
    \item If $\theta\in \, ]\frac{1}{3},\frac{1}{2}]$, then it can be seen that $Q_{\theta}(\frac{1}{\theta}) \leq  0$. Next, the graph of $Q_{\theta}$ is a downward parabola that attains its maximum at $x_{\text{top}} = \frac{ 24 \theta^{3} - 36 \theta^{2} + 12 \theta - 1}{4\theta\left(6 \theta^{3} - 14 \theta^{2} + 7 \theta - 1 \right)}$. It can be verified that $x_{\text{top}}\geq\frac{1}{\theta}$. Thus, $Q_{\theta}(x)\leq 0$.
\end{itemize}
Since $Q_{\theta}(x)\leq 0$, we have $ K_{\theta,2}(x)\leq 0$. Combined with $ K_{\theta,1}(x)\leq 0$, this gives the bound $K_{\theta}(y)\leq 0$.

The rational function $R_{\theta}$ is holomorphic on $Re(z)<\frac{1}{\theta}$. Thus, by the maximum modulus principle, we obtain
    \begin{align*}
        G_{\theta}(x) &= \sup_{Re(z)\leq x}|R_{\theta}(z)|
        = \sup_{Re(z)= x}|R_{\theta}(z)|
        =R_{\theta}(x).
    \end{align*}

    Finally, it is easily seen for the second derivative of $R_{\theta}$ one has \[ \frac{d^2R_{\theta}}{dx^2} (x)= \frac{1+2\theta(1-3\theta)x}{(1-\theta x)^2} \ge 0 \quad (0\leq x< \frac{1}{\theta}).\] Thus $R_{\theta}$ is convex on $[0,\frac{1}{\theta}[$ and consequently, if $0<\nu<\frac{1}{\theta}$, then \[R_{\theta}(x)\leq 1 +\frac{R_{\theta}(\nu)-1}{\nu}\, x\;  \quad (0\leq x \leq \nu).\]
\end{proof}
\end{document}